\documentclass[reqno,11pt]{article}
\usepackage{a4wide,graphicx,bm,amssymb,ifthen,url}
\usepackage{color,epsf,cancel}
\usepackage[arrow, matrix, curve]{xy}
\usepackage{color}
\usepackage{colortbl}
\usepackage{lmodern}
\usepackage{MnSymbol}
\usepackage{amsthm}
\usepackage{amsmath,amssymb,amstext}
\usepackage{graphicx}
\usepackage{tikz}
\usepackage{nicefrac}

\newtheorem{theorem}{Theorem}[section]
\newtheorem{lemma}[theorem]{Lemma}
\newtheorem{proposition}[theorem]{Proposition}
\newtheorem{corollary}[theorem]{Corollary}
\theoremstyle{definition}
\newtheorem{definition}[theorem]{Definition}

\theoremstyle{remark}
\newtheorem{remark}{Remark}
\newtheorem{example}[theorem]{Example}

\numberwithin{equation}{section}

\definecolor{magenta}{rgb}{0.75,0,0.25}

\definecolor{violet}{rgb}{0.25,0,0.75}

\definecolor{dgreen}{rgb}{0.0,0.5,0.0}

\newcommand{\w} {\mathrm{white}}
\newcommand{\y} {\mathrm{yellow}}
\newcommand{\fr}{\partial}
\newcommand{\linfw}{L^{\mathrm{white}}_{\infty}}
\newcommand{\linfy}{L^{\mathrm{yellow}}_{\infty}}
\newcommand{\Gw}{{\mathcal G}^{\mathrm{white}}}
\newcommand{\Gy}{{\mathcal G}^{\mathrm{yellow}}}

\newcommand {\D}  {{\mathcal D}}
\newcommand {\T}  {{\mathcal T}}
\newcommand {\W}  {{\mathcal W}}
\newcommand {\Y}  {{\mathcal Y}}
\newcommand {\V}  {{\mathcal V}}
\newcommand {\G}  {{\mathcal G}}
\newcommand {\B}  {{\mathcal B}}
\newcommand {\C}  {{\mathcal C}}
\newcommand {\E}  {{\mathcal E}}
\renewcommand {\P}  {\mathcal{P}}

\newcommand {\HH}  {{\mathcal H}}

\newcommand {\NN}  {{\mathbb N}}

\newcommand {\RR}  {{\mathbb R}}

\newcommand{\abs}[1]{{\left\lvert#1\right\rvert}}

\title{Triangular labyrinth fractals}

\author{Ligia L. Cristea \thanks{This author's work was supported by the Austrian Science Fund (FWF), 
Project P 27050-N26.} 
\\Karl-Franzens-Universit\"at Graz\\ Institut f\"ur Mathematik und Wissenschaftliches Rechnen
\\Heinrichstra{\ss}e 36, 8010 Graz,Austria\\ \tt{strublistea@gmail.com} 
 \and Paul Surer \thanks{This author's work was supported by the Austrian Science Fund (FWF), 
Project P 28991-N35.}\\ Universit\"at f\"ur Bodenkultur,\\Institut f\"ur Mathematik,\\ Gregor Mendel Stra{\ss}e 33, 1180 Wien, Austria\\ \tt{paul.surer@boku.ac.at} }

\date{\today}

\begin{document}

\maketitle

\emph{This article is dedicated to Christian Krattenthaler on the occasion of his $60$th birthday.}

\abstract{
We define and study a class of fractal dendrites called triangular labyrinth fractals. 
For the construction we use triangular labyrinth pattern systems, consisting of two triangular patterns: a white and a yellow one. Correspondingly, we have two fractals: a white and a yellow one. The fractals studied here are self-similar, and fit into the framework of graph directed constructions.
The main results consist in showing how special families of triangular labyrinth patterns systems, which are defined based on some shape features, 
can generate exactly three types of dendrites: labyrinth fractals where all non-trivial arcs have infinite length, fractals where all non-trivial arcs have finite length, or fractals where the only arcs of finite lengths are line segments parallel to a certain direction.
We also study the existence of tangents to arcs.
The article is inspired by research done on labyrinth fractals in the unit square that have been studied during the last decade. In the triangular case, due to the geometry of triangular shapes, some new techniques and ideas are necessary in order to obtain the results.
}

\section{Introduction}
Labyrinth fractals were introduced by Cristea and Steinsky \cite{laby_4x4,laby_oigemoan} as special families of  (self-similar) Sierpi\'nski carpets. They are dendrites that can be obtained by a successive application of a replacement rule on a  square. Recently, they have  been studied in more general settings  in  \cite{mixlaby, cristealeobacher_arcs,  supermix}. Self-similar labyrinth fractals \cite{laby_4x4,laby_oigemoan} can be seen as a special case of the self similar objects called ``fractal squares'' studied a few years later by Lau et. al. \cite{LauLuoRao}, who concentrated on the topological structure (connectedness properties). 
In the present article we adopt the concept from \cite{laby_4x4,laby_oigemoan} and introduce triangular labyrinth fractals.
They are obtained as the limit of an iterative process: we start with an equilateral triangle  $T_1$ of side length one, divide it into $m^2$ equilateral triangles of side-length $\nicefrac{1}{m}$,  
and then colour some of them in black 
\begin{figure}[h]
\begin{center}
\includegraphics[width=0.2\textwidth]{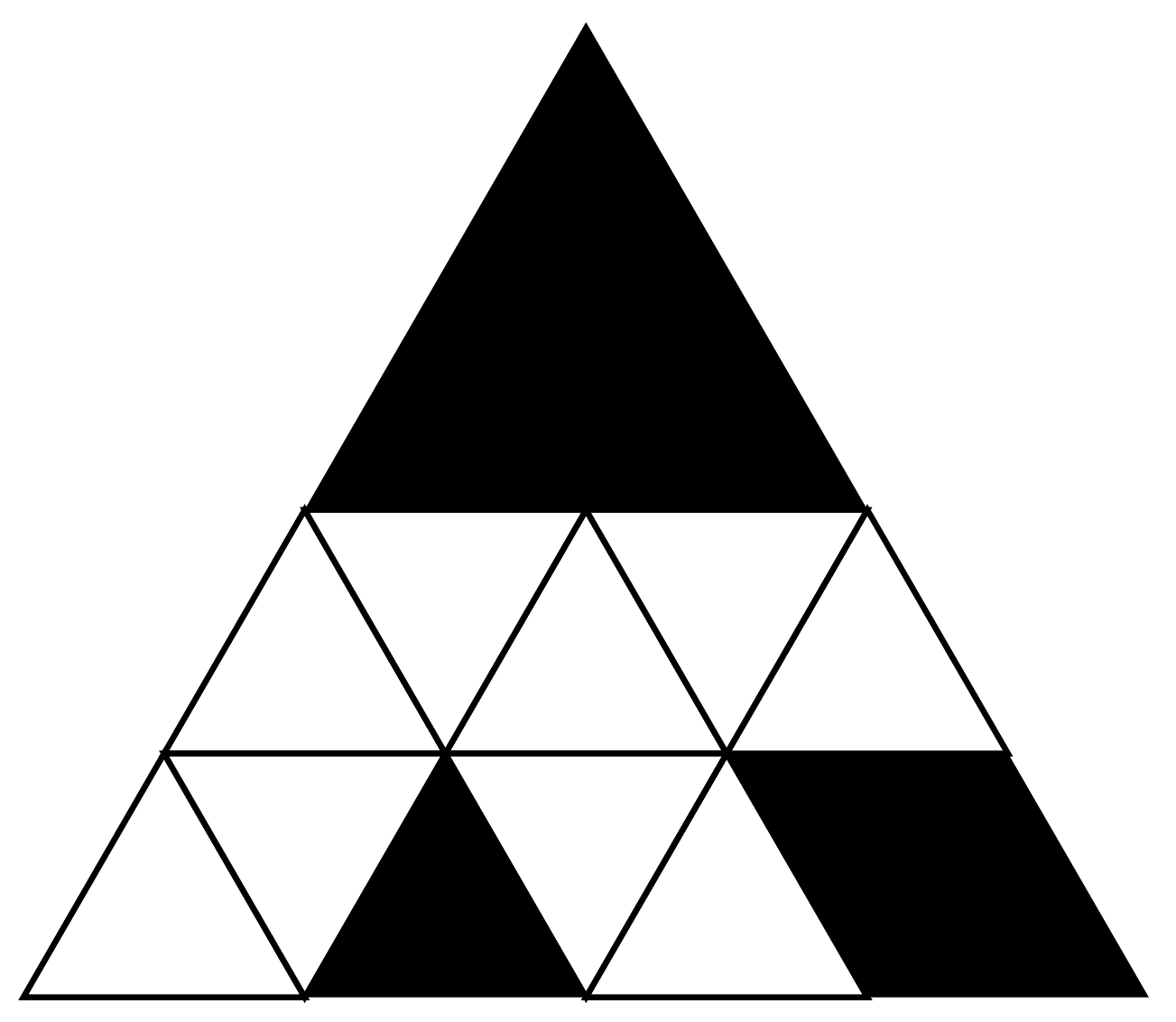}
\hspace{3cm}
\includegraphics[width=0.2\textwidth]{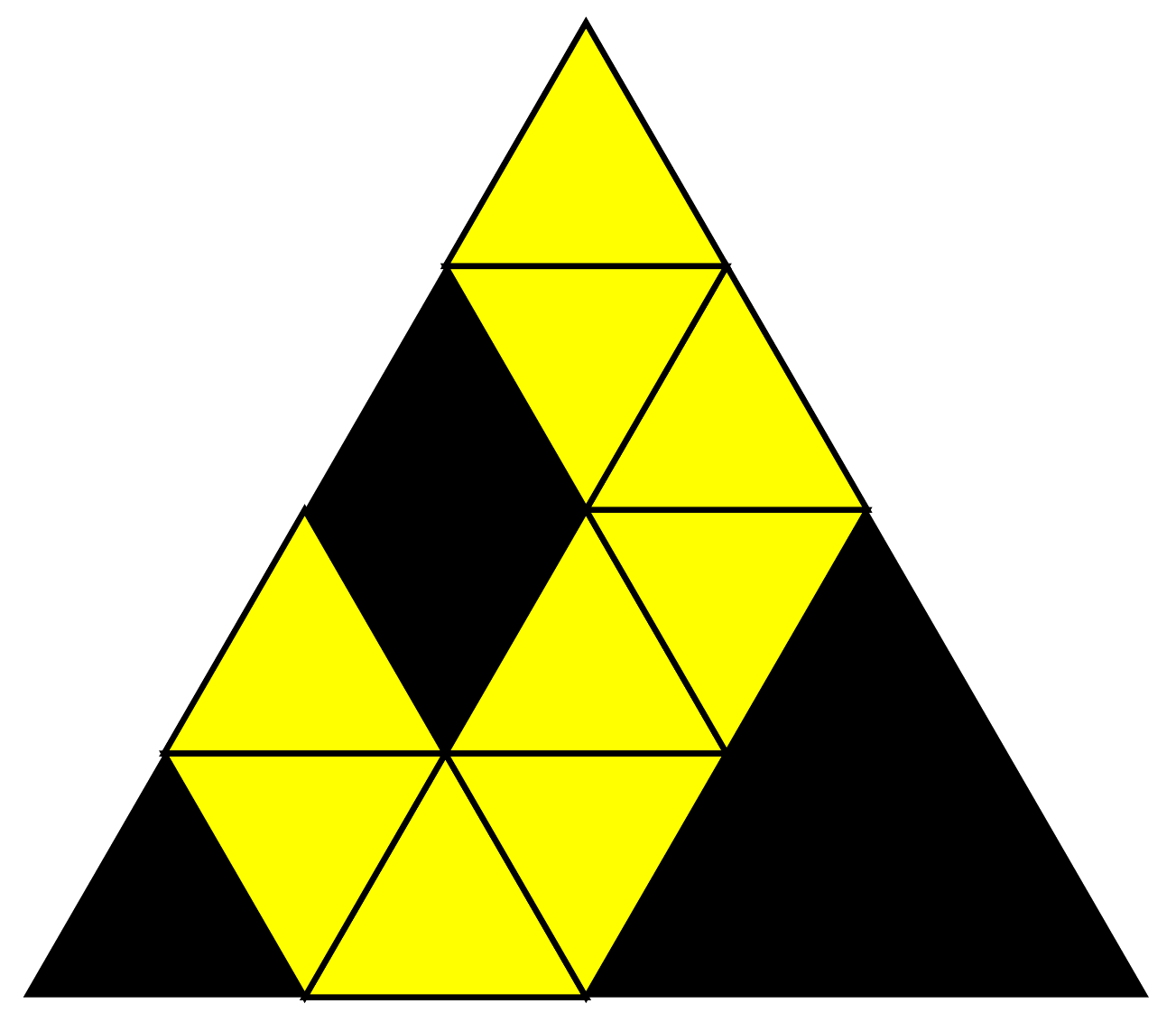}
\caption{An example of a triangular labyrinth patterns system. The white one on the left coincides with the  first step of the iterative process.}
\label{fig:Bsp1WY1}
\end{center}
\end{figure}
(indicating the sets that will be cut out throughout the iterative construction) and the rest in white 
%(or, in a yellow pattern, in yellow), 
according to a given triangular pattern. Then we repeat the process with the remaining white triangles.
The construction rule is determined by a triangular labyrinth patterns system that consists of a pattern for the $\nicefrac{m(m+1)}{2}$ ``upright'' triangles (the white pattern), and one for the $\nicefrac{m(m-1)}{2}$ ``upside down'' triangles (the yellow pattern).
\begin{figure}[h]
\begin{center}
\includegraphics[width=0.30\textwidth]{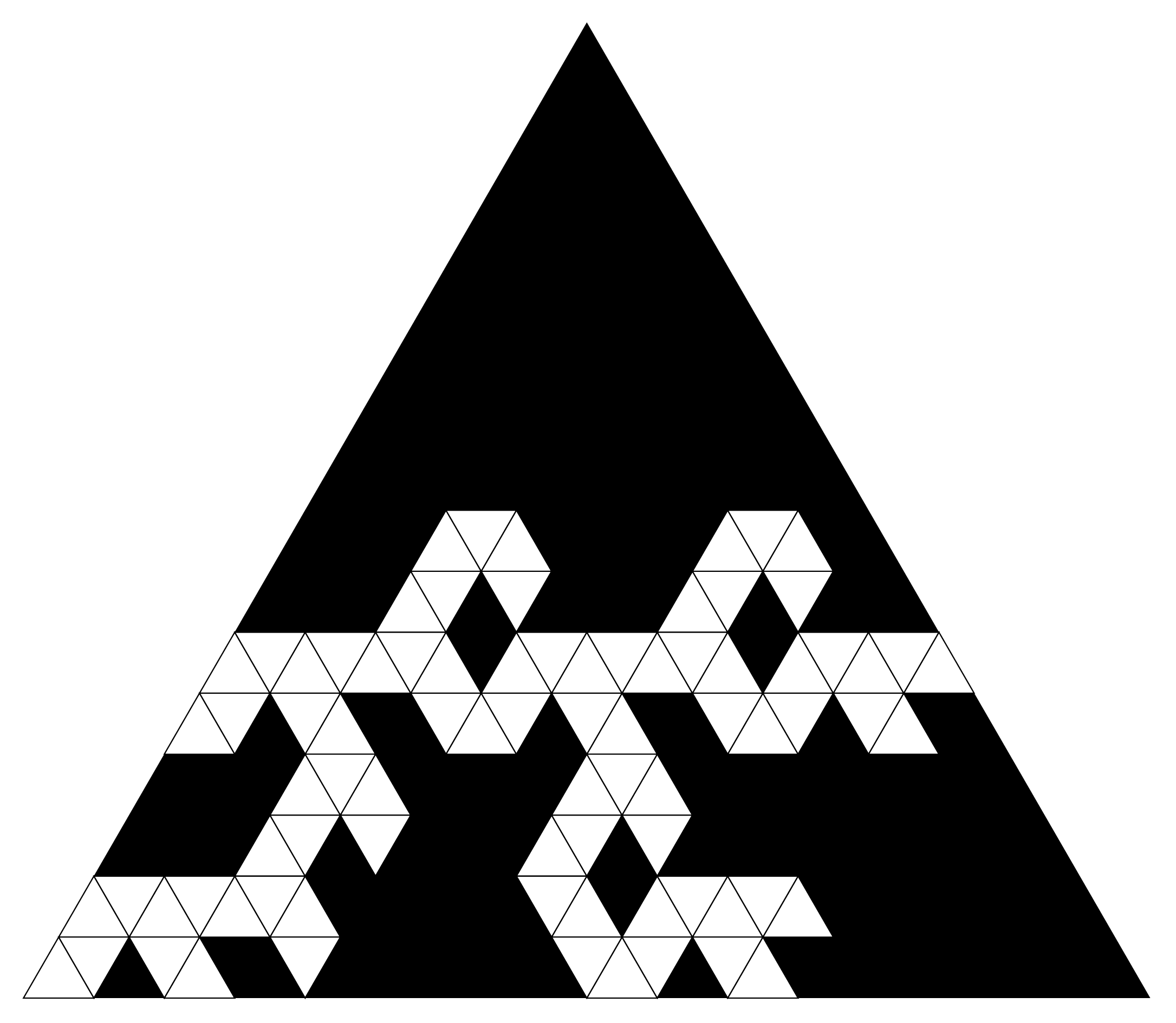}
\hspace{2.00cm}
\includegraphics[width=0.30\textwidth]{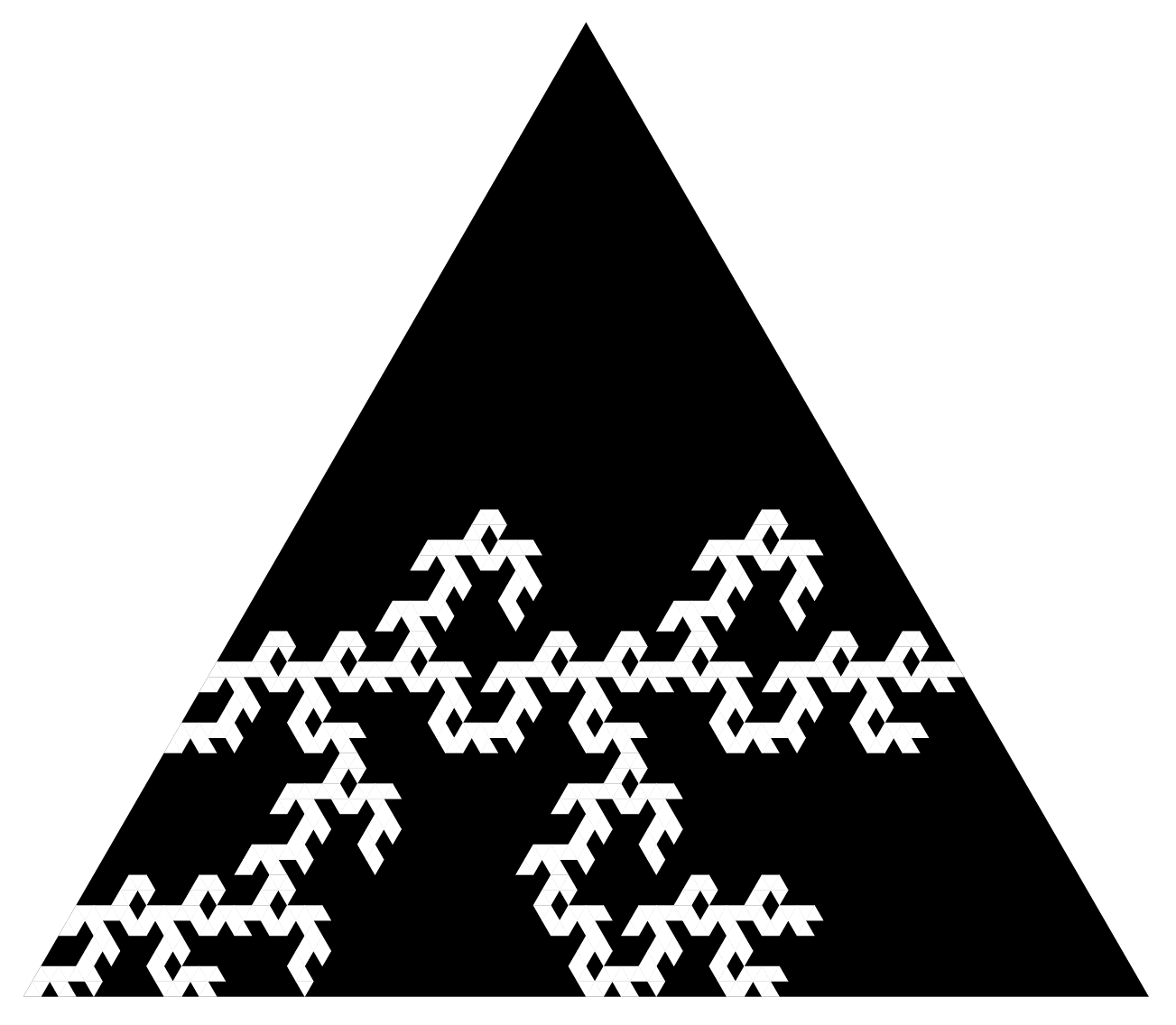}
\caption{The second and third step of the iterative process.}
\label{fig:Bsp1W23}
\end{center}
\end{figure}

While in the case of (square) labyrinth fractals, all squares in the patterns have the 
same position, i.e., they all ``look the same'', since they are just translated images of each other, 
in the triangular case some of the small triangles are ``upright'', like the initial triangle $T_1$, and the others are 
``upside down''. So in this case, in order to perform the iterations in the construction of the fractal, one uses two triangular patterns:
one for the upright triangles, and a second one for the upside down triangles.

The triangular labyrinth patterns used in the construction of triangular labyrinth fractals have to satisfy three conditions, given by the `tree property'', the ``exits property'', and the ``corners property'' which provide the labyrinth shape of the sets obtained at each iteration and, subsequently, the dendrite structure of the fractals.
Of course, one could also use one single pattern, but the symmetry requirements that it would have to satisfy would essentially reduce the class of patterns that are eligible 
for the construction and therefore also the richness of the class of resulting triangular labyrinth fractals. The fact that we use two patterns in the construction has the effect that triangular labyrinth fractals always appear in pairs since we can exchange the role of the white pattern and the yellow one. We therefore speak of the white  and yellow triangle labyrinth fractal.
\begin{figure}[h]
\begin{center}
\includegraphics[width=0.30\textwidth]{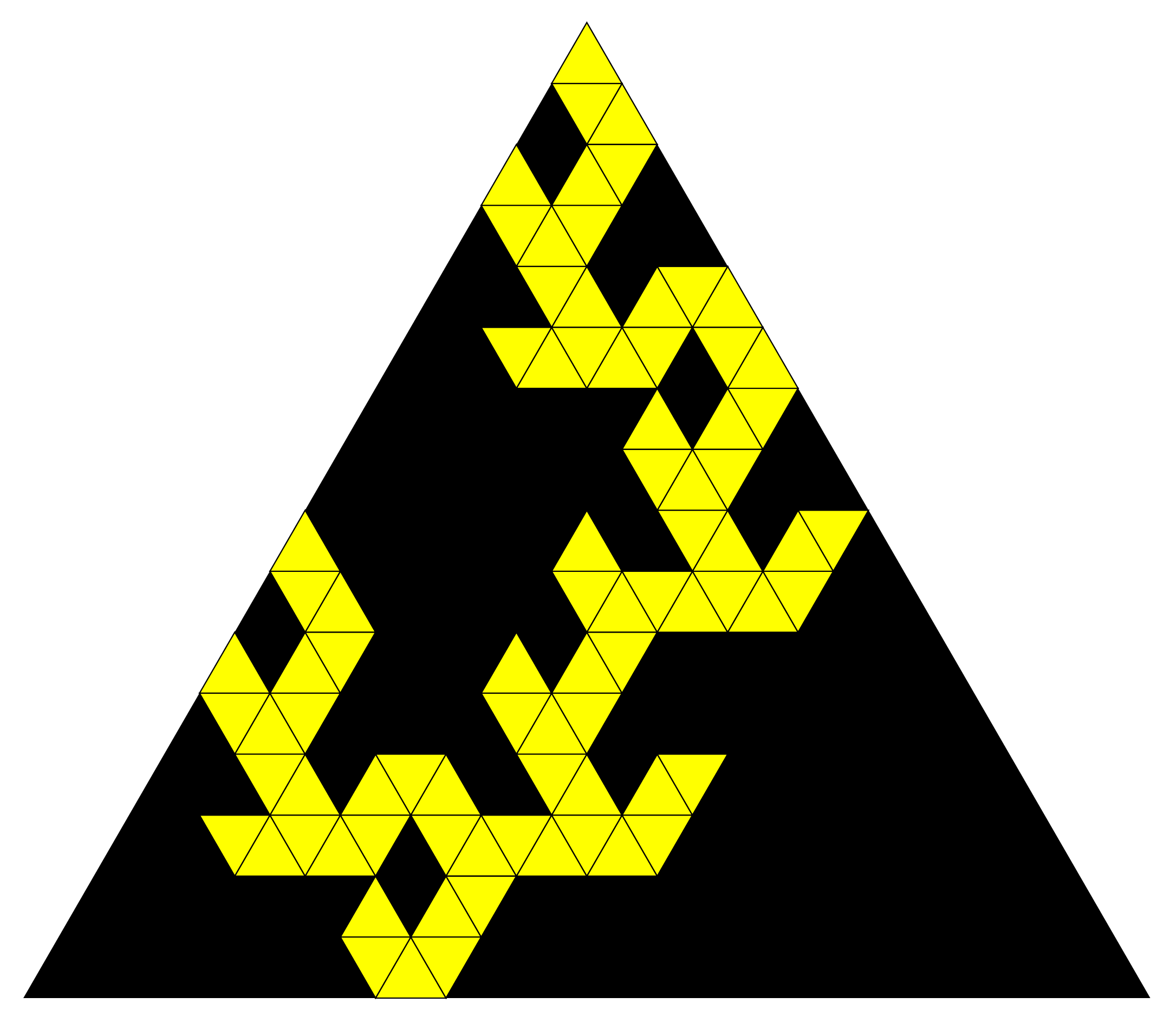}
\hspace{2.00cm}
\includegraphics[width=0.30\textwidth]{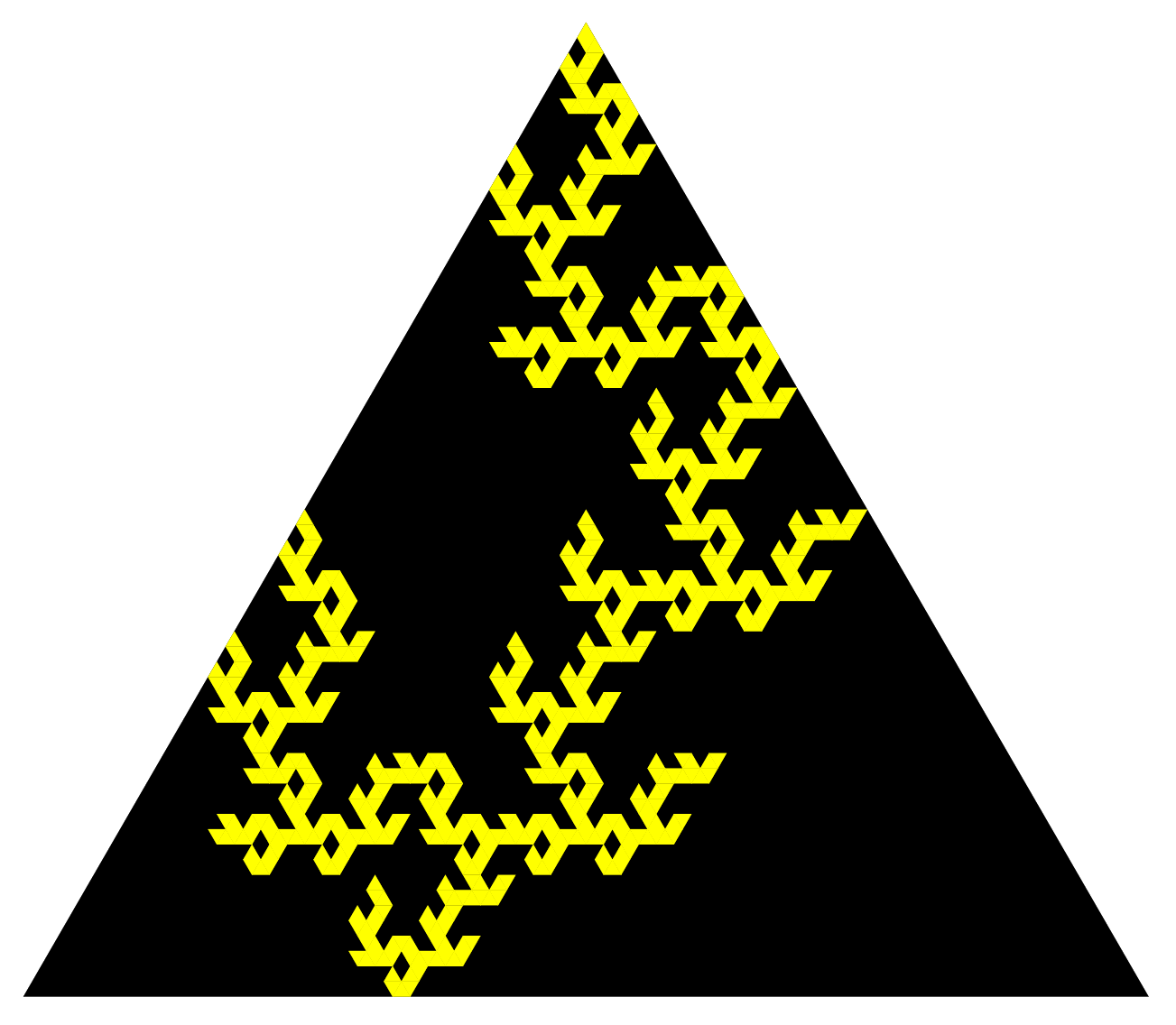}
\caption{Exchanging the role of the white and yellow patten yields the iterative construction of 
the yellow fractal. The figure shows the second and third step.
The first step corresponds to the yellow pattern on the right hand side of Figure~\ref{fig:Bsp1WY1}.}
\label{fig:Bsp1Y23}
\end{center}
\end{figure}
In fact, we consider the white and the yellow fractal in parallel since this is the more natural and also more convenient way.

The dual character of triangular labyrinth fractals
is the main difference to the classical labyrinth fractals from \cite{laby_oigemoan}.
While these (square) labyrinth fractals can be defined via an iterated function system (IFS, cf.~\cite{falconerbook}), triangular labyrinth fractals  correspond to a graph directed iterated function system (GIFS) in the sense of \cite{MW}. 
\begin{figure}[h]
\begin{center}
\includegraphics[width=0.3\textwidth]{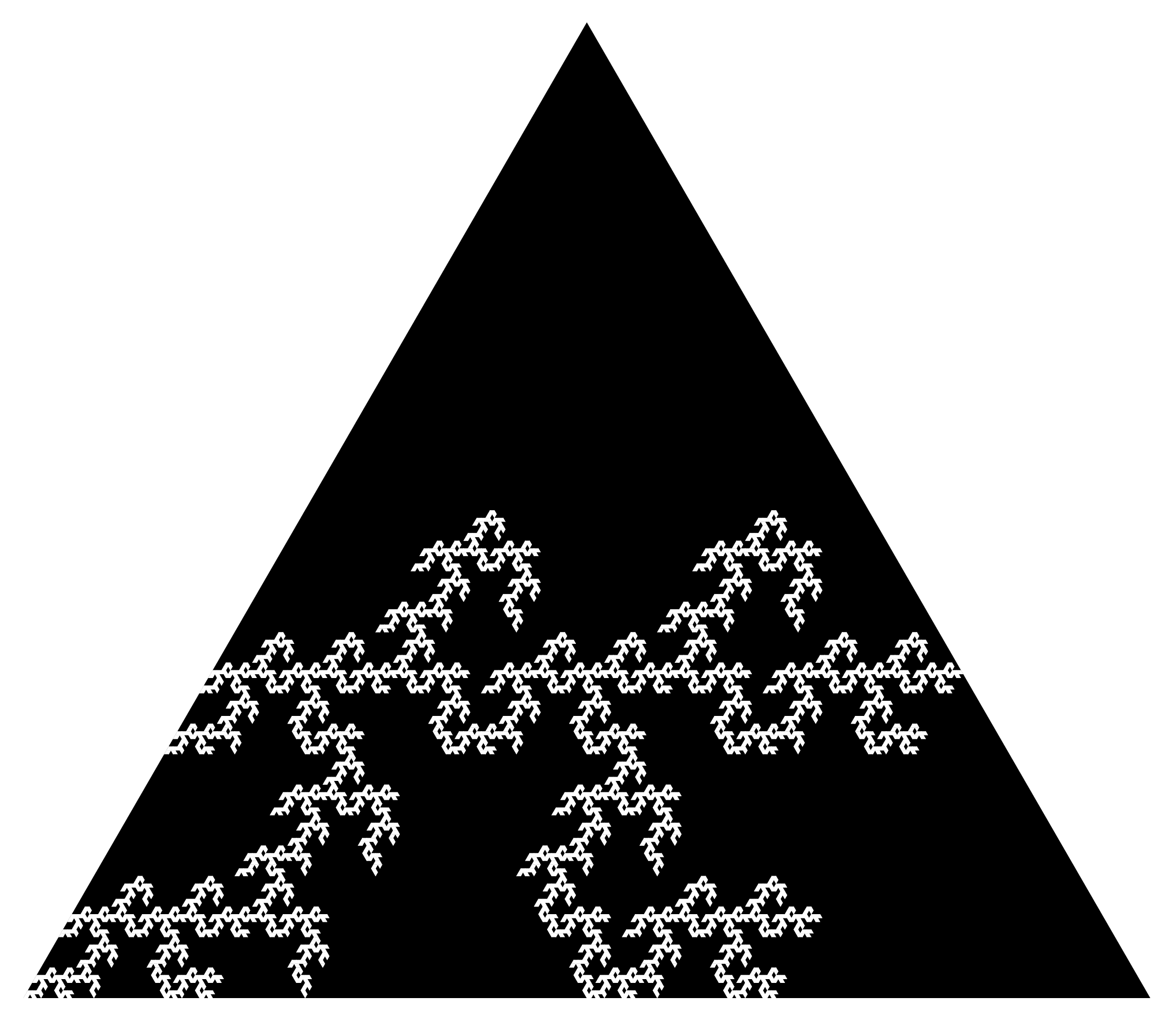}
\hspace{2cm}
\includegraphics[width=0.3\textwidth]{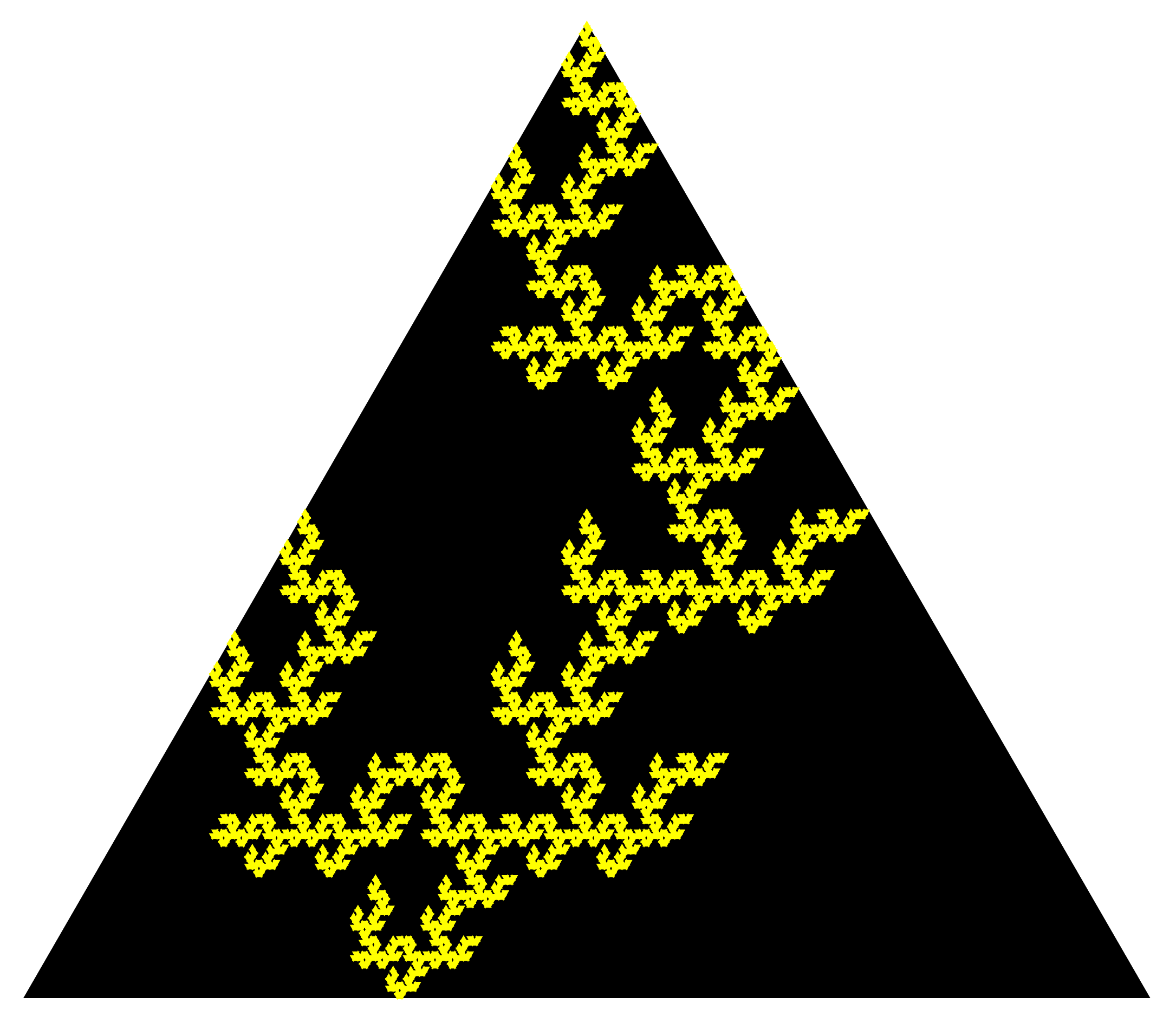}
\caption{We study both triangular labyrinth fractals in parallel. After 4 steps we already have a good approximation of them.}
\label{fig:Bsp1WY4}
\end{center}
\end{figure}

As already mentioned, triangular labyrinth fractals are dendrites. We give an explicit formula for the Hausdorff dimension that depends on the number of the non-black triangles in the triangular labyrinth patterns system.
This is proven based on the GIFS structure of these fractals. 
A large part of the paper deals with the arcs in the fractal. We prove that, in general, the dendrite has three main branches each of which ends up on one side of the initial triangle $T_1$, in a  so-called exit. The length of the arcs between these exits, i.e., their one-dimensional Lebesgue measure, can be finite or infinite.

The finite length of these arcs occurs under special circumstances, namely, when both patterns of the triangular labyrinth patterns system have, roughly speaking, no ``detours''.
We call such systems unblocked.  
To prove this we first show that these arcs are also given by another GIFS in order to afterwords apply the results form \cite{MW}. Subsequently, we extend the obtained results to arbitrary arcs in the fractals.
The main results consist in showing how special families of triangular labyrinth patterns systems, which are defined based on some shape features, 
can generate exactly three types of dendrites: labyrinth fractals where all non-trivial arcs have infinite length, fractals where all non-trivial arcs have finite length, or fractals where the only arcs of finite lengths are line segments parallel to a certain direction.

 Let us mention that the objects studied in this paper are related to other triangle-based fractals and dendrites. One could also view triangular labyrinth fractals as a family of more general Sierpi\'nski gaskets, (we refer, e.g., to \cite{freiberghamblyhutchinson_Vvariable} for $V$-variable Sierpi\'nski gaskets). 
Triangular patterns also occur in 
other  research articles of the last decade, for example
\cite{lipschitz_gaskets}, where Lipschitz equivalence of totally disconnected self similar sets generated by such patterns is studied.  In that case, only one pattern is used for generating the fractal and neighbouring triangles in the pattern share only a common vertex (and not a common side). More precisely, the triangular patterns consist only of  upright triangles, while in the case of our triangular labyrinth fractals we also consider triangles in the upside down position.
We also remark that there is a very recent article on fractal dendrites \cite{Samuel_selfsimilar_dendrites}, where self-similar  dendrites are constructed by using polygonal systems in the plane, a method based on IFS that is different from  the construction method of both square-related labyrinth fractals \cite{laby_4x4, laby_oigemoan} and  triangular labyrinth fractals.

The results on square based labyrinth fractals
have already several applications in physics,  e.g., in the study of planar nanostructures
\cite{GrachevPotapovGerman2013}, the fractal reconstruction of complicated
images, signals and radar backgrounds \cite{PotapovGermanGrachev2013}, and 
the construction of prototypes of ultra-wide band radar antennas
\cite{PotapovZhang2016}. Very recently,
 fractal labyrinths are used in combination
with genetic algorithms for the synthesis of big robust antenna arrays and
nano-antennas in  telecommunication \cite{PotapovPotapovPotapov_dec2017}. It is straightforward that triangle labyrinth fractals could be used in  similar contexts. 
Concerning further applications of triangular labyrinth fractals, 
we want to mention very recent results in materials engineering \cite{JanaGarcia_lithiumdendrite2017}, which show that dendrite growth, a largely unsolved problem,  plays an essential role when dealing with high power and energy lithium-ion batteries. Furthermore, there is recent research on crystal growth \cite{tarafdar_multifractalNaCl2013}, which suggests that, depending on the geometric structure of the studied material, labyrinth fractals are a suitable model for studying various phenomena and objects occurring in other fields of science.
Physicists' work regarding the
modelling of porous structures   by using Sierpi\'nski carpets \cite{Tarafdar_modelporstructrepeatedSC2001} indicate that the triangular labyrinth fractals could also provide such models when the basic structure of the studied material is not square- but triangle-related.
Finally, we remark that Koch curves, which are of interest in
theoretical physics in the context of diffusion processes ({e.g., }
\cite{SeegerHoffmannEssex2009_randomKoch}), are related to arcs  in triangular labyrinth fractals.  

The article is organised in the following way: in Section~\ref{sec:construction} we introduce the formalism that we need for the (iterative) construction. 
In Section~\ref{sec:labyrinth_patterns_systems} we give the definition of a
labyrinth patterns system and
associate with each step of the iteration a finite graph whose vertices correspond to the white (yellow, respectively), triangles. Two triangles that share a common side are linked by an edge in the graph. These graphs are an important tool throughout the entire article.
In Section~\ref{sec:TLF} we introduce the labyrinth fractals as the limit of the iterative process. We show that they fit into the framework of GIFS and calculate the Hausdorff dimension based on this fact.
In Section~\ref{sec:topological_prop} we study the topology of labyrinth fractals, especially, we prove that they are dendrites. In the rest of the article we are mainly interested in arcs within the fractals. 
We first study the arcs between exits  of the dendrite. 
Therefore, in Section~\ref{sec:paths_TLG} we first deal with the paths between exits in the above mentioned graphs.  
We introduce the path matrix of a labyrinth patterns system, which is an important instrument in several further considerations throughout the paper.
In Section~\ref{sec:arcs_TLF} we show how to construct these arcs.
All in all, we have six arcs between exists, three in the white labyrinth fractal and three in the yellow one. It turns out that these arcs can also be obtained from an GIFS.
They have Hausdorff dimension strictly larger than $1$ (and thus, infinite length) if the labyrinth pattern system satisfies the condition of being ``globally blocked''. The respective calculations are done in  Section~\ref{sec:blocked}.
In Section~\ref{sec:arc_length} we use the results on the arcs between exists in order to see that (in the case of a globally blocked system)  all non-trivial arcs in the fractal have infinite length. Finally, in Section~\ref{sec:partially_blocked} we investigate labyrinth fractals defined by not globally blocked (i.e., unblocked) triangular patterns systems.

\section{Construction}\label{sec:construction}

For three non-collinear points $\mathbf{A}_1$, $\mathbf{A}_2$, $\mathbf{A}_3$ we denote by
$\Delta(\mathbf{A}_1, \mathbf{A}_2,\mathbf{A}_3)$ the full triangle spanned by $\mathbf{A}_1$, $\mathbf{A}_2$, $\mathbf{A}_3$.
We denote by $H\subset \RR^3$ the set of homogeneous coordinates, that is
\[H=\{(\alpha_1, \alpha_2, \alpha_3) \in [0,1]^3: \, \alpha_1+\alpha_2+\alpha_3=1\}.\]
Every $\mathbf{x} \in \Delta(\mathbf{A}_1, \mathbf{A}_2,\mathbf{A}_3)$ can be uniquely represented by $(\alpha_1, \alpha_2, \alpha_3) \in H$ such that
$\mathbf{x}=\alpha_1\mathbf{A_1} + \alpha_2\mathbf{A_2}+\alpha_2\mathbf{A_2}$.

We define $T_1$ to be the equilateral triangle with side-length one
$T_1=\Delta(\mathbf{P}_1, \mathbf{P}_2,\mathbf{P}_3)$ where  
\[\mathbf{P}_1 = \left(0,0\right), \qquad \mathbf{P}_2 = \left(1,0 \right)
, \qquad \mathbf{P}_3 = \left(\nicefrac12, \nicefrac{\sqrt{3}}{2} \right)\]
(actually, $\mathbf{P}_1$, $\mathbf{P}_2$ and $\mathbf{P}_3$ can be arbitrary non-collinear elements of $\RR^2$).

The boundary of $T_1$ consists of  three subsets
$\partial T_1 = \varrho_1 \cup \varrho_2 \cup \varrho_3$ where the line segment $\varrho_i$ is the  side of $T_1$ that  lies opposite  to $\mathbf{P}_i$, for all $i \in \{1,2,3\}$, that is
\[\varrho_i = \{\alpha_1\mathbf{P}_1+\alpha_2\mathbf{P}_2+\alpha_3\mathbf{P}_3: (\alpha_1, \alpha_2, \alpha_3) \in H, \alpha_i=0\}.\]

For non-collinear points $\mathbf{A}_1, \mathbf{A}_2, \mathbf{A}_3 \in T_1$ we define
\emph{the projection map} onto the triangle $\Delta(\mathbf{A}_1, \mathbf{A}_2,\mathbf{A}_3)$
\begin{equation*}
P_{\Delta(\mathbf{A}_1, \mathbf{A}_2,\mathbf{A}_3)}: T_1 \longrightarrow T_1, 
\mathbf{x} \longmapsto \alpha_1 \mathbf{A}_1+ \alpha_2 \mathbf{A}_2 +\alpha_3 \mathbf{A}_3,
\end{equation*}
where $(\alpha_1, \alpha_2, \alpha_3) \in H$ are the homogeneous coordinates of $\mathbf{x}$ (with respect to $\mathbf{P}_1$, $\mathbf{P}_2$ and $\mathbf{P}_3$).
Clearly, $P_{\Delta(\mathbf{A}_1, \mathbf{A}_2,\mathbf{A}_3)}$ maps $T_1$ onto $\Delta(\mathbf{A}_1, \mathbf{A}_2,\mathbf{A}_3)$. In particular, if $A_1, A_2, A_3$ are the vertices of a equilateral triangle, then $P_{\Delta(\mathbf{A}_1, \mathbf{A}_2,\mathbf{A}_3)}$ is a similarity mapping.

Now let $m \geq 2$ and define for integers $k_1, k_2, k_3$ 
\begin{equation}\label{formel_Tm}
\def\arraystretch{1.5}
\begin{array}{rlrl}
T_{m}(k_1,k_2,k_3) = & \Delta(\mathbf{A}_1, \mathbf{A}_2, \mathbf{A}_3) & \text{with } &
\mathbf{A}_1=\frac{(k_1+1)\mathbf{P}_1 + k_2\mathbf{P}_2 + k_3\mathbf{P}_3}{m}, \\
&&&\mathbf{A}_2= \frac{k_1\mathbf{P}_1 + (k_2+1)\mathbf{P}_2 + k_3\mathbf{P}_3}{m}, \\
&&&\mathbf{A}_3 = \frac{k_1\mathbf{P}_1 + k_2\mathbf{P}_2 + (k_3+1)\mathbf{P}_3}{m} \\
T'_{m}(k_1,k_2,k_3) = & \Delta(\mathbf{A}_1, \mathbf{A}_2, \mathbf{A}_3) & \text{with } &
\mathbf{A}_1=\frac{k_1\mathbf{P}_1 + (k_2+1)\mathbf{P}_2 + (k_3+1)\mathbf{P}_3}{m},\\
&&&\mathbf{A}_2 = \frac{(k_1+1)\mathbf{P}_1 + k_2\mathbf{P}_2 + (k_3+1)\mathbf{P}_3}{m},\\
&&&\mathbf{A}_3 = \frac{(k_1+1)\mathbf{P}_1 + (k_2+1)\mathbf{P}_2 + k_3\mathbf{P}_3}{m}.
\end{array} 
\end{equation}
Each of these triangles is similar to $T_1$. If $k_1, k_2, k_3$ are non-negative then 
$T_{m}(k_1,k_2,k_3) \subset T_1$ provided that
$k_1+k_2+k_3=m-1$ and $k_1+k_2+k_3=m-2$ implies that $T'_{m}(k_1,k_2,k_3) \subset T_1$.
More precisely, these triangles provide a decomposition  of $T_1$ into $m^2$ triangles of side length $1/m$.
Define
\begin{align*}
\T_{m} = & \{T_m(k_1,k_2,k_3): (k_1,k_2,k_3) \in\NN^3, k_1+k_2+k_3=m-1\} \\
\T'_{m} = & \{T'_m(k_1,k_2,k_3): (k_1,k_2,k_3) \in\NN^3, k_1+k_2+k_3=m-2\}.
\end{align*}
The elements of $\T_{m} \cup \T'_{m}$ have pairwise disjoint interior and
\[T_1= \bigcup_{T_m \in \T_m} T_m \cup \bigcup_{T'_m \in \T'_m} T'_m.\]
Based on our choice of $\mathbf{P}_1$, $\mathbf{P}_2$ and $\mathbf{P}_3$  we also refer to the 
triangles contained in $\T_{m}$ as {\it upright triangles}, while the
the elements of $\T'_{m}$ are the {\it upside down} triangles.
A side in a triangle $T \in \T_m \cup \T'_m $ that is parallel to $\varrho_i$ is called the \emph{side of type} $i$ of $T$, for $i\in\{1,2,3\}$.

Observe that for a fixed $k_1:=K$, with $K \in\{0, \dots, m-1  \}$, the set of triangles 
\begin{multline*}
\{T_m(K,k_2,k_3): k_2,k_3 \in \NN, K+k_2+k_3=m-1\}  \cup \\
\{T'_m(K,k_2,k_3): k_2,k_3 \in \NN, K+k_2+k_3=m-2\} \subset \T_{m} \cup \T'_{m}
\end{multline*}
builds a strip ``parallel'' to $\varrho_1$. We call it a \emph{strip of type $1$}. In particular, for $K=0$ we get the border strip of type $1$, where all triangles  $T \in  \T_{m} \cup \T'_{m}$ situated in the strip have non-empty intersection with $\varrho_1$. 
Analogously, by fixing $k_2$ ($k_3$, respectively) we obtain strips of type $2$  (type $3$, respectively) that are ``parallel'' to $\varrho_2$ ($\varrho_3$, respectively).

Two distinct elements of $\T_{m} \cup \T'_{m}$ are {\it neighbours} if they have non-empty intersection. This intersection is either a singleton or a line segment.
In the latter case, one triangle is necessarily  contained in $\T_m$ and the other one is an element of 
$\T'_{m}$. More precisely, 
if $T=T_m(k_1,k_2,k_3)$ and $T'=T'_m(k'_1,k'_2,k'_3)$ then $T \cap T'$ is a line segment if and only if there exists a $j \in \{1,2,3\}$ such that $k_j=k'_j+1$ and $k_i=k'_i$ for $i \not= j$.
We call such a pair of triangles {\it $j$-neighbours}. We remark that the intersection of $j$-neighbours is the side of type $j$ of both triangles. This situation, for our choice of $\mathbf{P}_1$, $\mathbf{P}_2$ and $\mathbf{P}_3$, is sketched in Figure~\ref{fig:neighbours}.

\begin{figure}[h]
\begin{center}
 \begin{tikzpicture}
[scale=1.0]

%1-neighbours
\draw[line width=1.2pt, draw= black] (-4-1.299,-0.75) -- (-4-0.433,0.75);
\draw[line width=1.2pt, draw= black] (-4-0.433,0.75) -- (-4+0.433,-0.75);
\draw[line width=1.2pt, draw= black] (-4+0.433,-0.75) -- (-4-1.299,-0.75);
\draw[line width=1.2pt, draw= black] (-4+0.433,-0.75) -- (-4+1.299,0.75) ;
\draw[line width=1.2pt, draw= black] (-4+1.299,0.75) -- (-4-0.433,0.75);

% 2-neighbours
\draw[line width=1.2pt, draw= black] (1.299,-0.75) -- (0.433,0.75);
\draw[line width=1.2pt, draw= black] (0.433,0.75) -- (-0.433,-0.75);
\draw[line width=1.2pt, draw= black] (-0.433,-0.75) -- (1.299,-0.75);
\draw[line width=1.2pt, draw= black] (-0.433,-0.75) -- (-1.299,0.75) ;
\draw[line width=1.2pt, draw= black] (-1.299,0.75) -- (0.433,0.75);

% 3-neighbours
\draw[line width=1.2pt, draw= black] (4+0,1.5) -- (4+0.866,0);
\draw[line width=1.2pt, draw= black] (4+0.866,0) -- (4-0.866,0);
\draw[line width=1.2pt, draw= black] (4+0,1.5) -- (4-0.866,0);
\draw[line width=1.2pt, draw= black] (4+0,-1.5) -- (4+0.866,0);
\draw[line width=1.2pt, draw= black] (4+0,-1.5) -- (4-0.866,0);
\end{tikzpicture}
\end{center}
\caption{From left to right: $1$-neighbours, $2$-neighbours and $3$-neighbours.}
\label{fig:neighbours}
\end{figure}
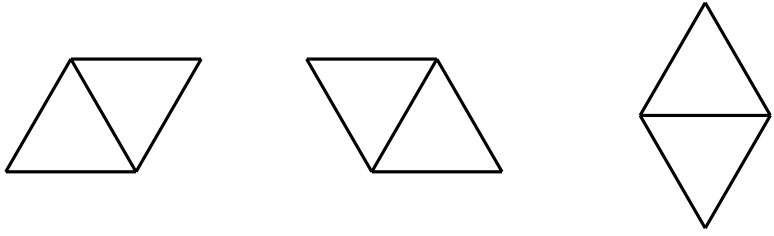

A triangle $T \in \T_m$ is a \emph{border triangle} if it has non-empty intersection with the border of $T_1$, {\it i.e.} $T \cap \partial T_1 \not=\emptyset$. We easily see that $T$ is a border triangle if and only if $T=T_m(k_1,k_2,k_3)$ where $k_1k_2k_3=0$.
A {\it corner triangle} is an element $T \in \T_m$ that contains $\mathbf{P}_1$, $\mathbf{P}_2$ or $\mathbf{P}_3$, i.e., $T=T_m(k_1,k_2,k_3)$ with $k_1k_2 + k_1k_3+ k_2k_3=0$.

Let $\W, \Y \subset \T_{m}$ and $\W', \Y' \in \T'_{m}$. We call $(\W \cup \W',\Y \cup \Y')$ an \emph{$m$-triangular patterns system}, where  we often omit the $m$ for simplicity.
Starting with such a triangular pattern system, we perform an inductive construction defined by it. We proceed as follows.
Let $\W_1:=\W,\;  \W'_1:=\W'$ and $\Y_1:=\Y, \; \Y'_1:=\Y'.$ 
The set $\W_1 \cup \W'_1$ is called \emph{the set of white triangles of level} $1$. 
We define the white triangles of higher order recursively.  To obtain the white triangles of level $2$ we, roughly speaking,  replace each triangle in $\W_1$ by smaller copies of $\W_1 \cup \W_1'$ and each triangle in $\W'_1$ by smaller copies of $\Y_1 \cup \Y_1'$.

For $n\ge 1$,  $\W_n \cup \W'_n$ is called \emph{the set of white triangles  of level} $n$ generated by the triangular patterns system $(\W \cup \W',\Y \cup \Y')$. For $n>1$ the sets of triangles $\W_n$ and $\W'_n$ are formally defined in the following way:
\begin{align*}
\W_n:=\{P_{W_{n-1}}(W_1): W_1 \in \W_1, W_{n-1} \in \W_{n-1}\} \cup \{P_{W'_{n-1}}(Y'_1): Y'_1 \in \Y'_1, W'_{n-1} \in \W'_{n-1}\} \\
\W'_n:=\{P_{W_{n-1}}(W'_1): W'_1 \in \W'_1, W_{n-1} \in \W_{n-1}\} \cup \{P_{W'_{n-1}}(Y_1): Y_1 \in \Y_1, W'_{n-1} \in \W'_{n-1}\}.
\end{align*}
For all $n \geq 1$ we have $\W_n \subset \T_{m^n}$ and $\W'_n \subset \T'_{m^n}$.
We can do the same for $\Y_n \cup \Y'_n$, with $n \ge 1$, which we will refer to as the \emph{set of yellow triangles of level} $n$ 
generated by the triangular patterns system $(\W\cup \W',\Y \cup \Y')$. Here we have, for $n>1$,
\begin{align*}
\Y_n:=\{P_{Y_{n-1}}(Y_1): Y_1 \in \Y_1, Y_{n-1} \in \Y_{n-1}\} \cup \{P_{Y'_{n-1}}(W'_1): W'_1 \in \W'_1, Y'_{n-1} \in \Y'_{n-1}\} \\
\Y'_n:=\{P_{Y_{n-1}}(Y'_1): Y'_1 \in \Y'_1, Y_{n-1} \in \Y_{n-1}\} \cup \{P_{Y'_{n-1}}(W_1): W_1 \in \W_1, Y'_{n-1} \in \Y'_{n-1}\}.
\end{align*}

\begin{example}\label{Beispiel1}
Let $m=4$ and
\begin{align*}
\W_{1}:= & \{T_{4}({\bf k}): {\bf k} \in \{(0, 2, 1), (1, 1, 1), (1, 2, 0), (2, 0, 1), (3, 0, 0)\}\}, \\
\W'_{1} := & \{T'_{4}({\bf k}): {\bf k} \in \{(0, 1, 1), (1, 0, 1), (1, 1, 0), (2, 0, 0)\}\}, \\
\Y_{1}:= & \{T_{4}({\bf k}): {\bf k} \in \{(0, 0, 3), (0, 1, 2), (1, 1, 1), (2, 0, 1), (2, 1, 0)\}\}, \\
\Y'_{1} := & \{T'_{4}({\bf k}): {\bf k} \in \{(0, 0, 2), (0, 1, 1), (1, 1, 0), (2, 0, 0)\}\}.
\end{align*}
Then $(\W \cup \W',\Y \cup \Y')$ is the $4$-triangular patterns system that we presented in the introduction.
Figure~\ref{fig:Bsp1WY1}  shows exactly the white and yellow triangles of level $1$. On the left we see $T_1$. The upright white triangles are the elements of $\W_1$ while the upside down ones correspond to the elements of $\W'_1$. The right hand side shows the yellow analogon. Both $\W_{1}$ and  $\Y_{1}$ contain one corner triangle, namely $T_4(3, 0, 0)$ and $T_4(0, 0, 3)$, respectively. Each corner triangle is also a border triangle. Both sets, $\W_{1}$ and $\Y_{1}$, contain $3$ more border triangles.

In order to obtain the white triangles of level $2$ we
 replace every element of $\W_1$ ({\it i.e.} the upright white triangles of level $1$) by smaller copies (ratio $\frac{1}{4}$) of the white triangles of level $1$ and each element of $\W'_1$ (i.e., the upside down white triangles of level $1$) by smaller copies of the yellow triangles of level $1$, coloured in white.
One obtains the white and yellow triangles of level $3$ in a similar way. For the yellow triangles of higher order the construction is done in the corresponding analogous way.
The white and yellow triangles of level $2$ and $3$ are depicted in
Figure~\ref{fig:Bsp1W23} and Figure~\ref{fig:Bsp1Y23}, respectively. Figure~\ref{fig:Bsp1WY4} shows those of level $4$.
\end{example}

%-----------------------------------------Section 2------------------------------------------

\begin{section}{Triangular labyrinth patterns systems}\label{sec:labyrinth_patterns_systems}

Recall that a graph $\G=\left(\V(\G),\E(\G)\right)$ consists of a set of vertices $\V(\G)$ and a set of edges
$\E(\G)$. 
In the present section we consider undirected graphs.
In this case $\E(\G)$ is a subset of the unordered pairs of $\V(\G)$.
If $\{v_1,v_2\} \in \E(\G)$ then we say that the two vertices $v_1$ and $v_2$ are connected by an edge, that they are adjacent or, simply, that they are neighbours in the graph.
A path in an undirected graph is a sequence of pairwise distinct vertices $v_1, \ldots, v_k$, $k\ge 1$, such that for every $j \in \{1, \ldots, k-1\}$ the vertices $v_j$ and $v_{j+1}$ are connected by an edge. 
We call $v_1$ the initial vertex  and $v_k$ the terminal vertex of the path. 
If there exists an edge that connects the initial vertex $v_1$ with the terminal vertex $v_k$  then the path is called a cycle (provided that $k \geq 3$).

A graph is connected if for any two distinct vertices $v, v'$ in the graph there exists a path with initial vertex $v$ and terminal vertex $v'$. A tree is a connected graph that contains no cycles.

For any $m$-triangular patterns system $(\W \cup \W',\Y \cup \Y')$
we define the graph $\G(\W \cup\W')$ as the undirected graph with vertex set $\W \cup \W'$ where two vertices are connected by an edge  if and only if the corresponding two triangles are $j$-neighbours for a $j \in \{1,2,3\}$. We mark the edge with the label $j$.
The graph $\G(\Y \cup\Y')$ with set of vertices $\Y \cup \Y'$ is defined analogously.

We remark that  the graphs $\G(\W_n \cup\W'_n)$ and $\G(\Y_n \cup\Y'_n)$ are bipartite. Indeed,
every edge in $\G(\W_n \cup\W'_n)$  connects an element of $\W_n$ with an element of  $\W'_n$ and every edge in $\G(\Y_n \cup\Y'_n)$  connects an element of  $\Y_n$ with an element of  $\Y'_n$.

Since for $n\ge 1$ the triangles of level $n$ yield the $m^n$-triangular patterns system $(\W_n \cup \W'_n,\Y_n \cup \Y'_n)$,
we can analogously define the graphs $\G(\W_n \cup\W'_n)$ and $\G(\Y_n \cup\Y'_n)$.

It is important to keep in mind that for $\G(\W \cup\W')$ and $\G(\Y \cup\Y')$ two vertices are neighbours in the graph if and only if the respective triangles are $j$-neighbours.
Thus, here the property ``neighbour in the graph''  is not equivalent to the  notion of neighbours given in Section \ref{sec:construction} for triangles of $\T_m \cup \T'_m$, which only requires that two triangles have non-empty intersection.

\begin{definition}\label{def:triangular labyrinth patterns system}
Let $(\W \cup \W', \Y\cup \Y')$ be an $m$-triangular patterns system, with the corresponding graphs $\Gw$ and $\Gy$, as defined above. \\We call $(\W \cup \W', \Y \cup \Y')$ an $m$-\emph{triangular labyrinth patterns system} if it has the following three properties:
\begin{enumerate}
\item{{\bf The tree property.} $\Gw$ as well as $\Gy$ are trees.}
\item{{\bf The exits property.} There exists exactly one triple $(k_1,k_2,k_3)$ of non-negative integers (each of which is smaller than $m$) such that
\begin{align*}
T_{m}(0,k_1,m-1-k_1) \in \W \wedge & T_{m}(0,m-1-k_1,k_1) \in \Y & \text{and} \\
T_{m}(k_2,0,m-1-k_2) \in \W \wedge & T_{m}(m-1-k_2,0,k_2) \in \Y & \text{and} \\
T_{m}(k_3,m-1-k_3,0) \in \W \wedge & T_{m}(m-1-k_3,k_3,0)  \in \Y.
\end{align*}
We call the above triangles $T_{m}(0,k_1,m-1-k_1)\in \W$ and $ T_{m}(0,m-1-k_1,k_1) \in \Y$ the exit $e_1^{\w}$ of $\W$ and, respectively, the exit $e_1^{\y}$ of $\Y$.
Correspondingly, we have exits $e_2^{\w}$, $e_2^{\y}$, $e_3^{\w}$ and $e_3^{\y}.$ 
} 
\item{{\bf The corners property.} Denote by $\C(\W)$ and $\C(\Y)$ the sets of corner triangles contained in 
$\W$ and $\Y$, respectively, that is
\begin{align*}
\C(\W):=&\left\{W \in \W: W \cap \{\mathbf{P}_1, \mathbf{P}_2, \mathbf{P}_3\} \not= \emptyset\right\}, \\
\C(\Y):=&\left\{Y \in \Y: Y \cap \{\mathbf{P}_1, \mathbf{P}_2, \mathbf{P}_3\} \not= \emptyset\right\}.
\end{align*}
Then  $\C(\W)$ and $\C(\Y)$ are disjoint sets each of which contains at most one element, i.e.,
\[\C(\W) \cap \C(\Y) = \emptyset \wedge \#\C(\W) \leq 1 \wedge \#\C(\Y) \leq 1.\]
}
\end{enumerate}
\end{definition}

One can easily check that the triangular patterns system presented in Example~\ref{Beispiel1} satisfies the conditions of the above definition and is therefore a $4$-triangular labyrinth patterns system.

\begin{proposition}\label{prop:triangle_laby_system_order_n}
Let $(\W \cup \W', \Y\cup \Y')$ be an $m$-triangular labyrinth patterns system, and let $\W_1:=\W$, $\W'_1:=\W'$, $\Y_1:=\Y$, and $ \Y'_1:=\Y'$. For every $n\ge 1$ let
$\W_n \cup \W'_n$, be the set of white triangles  of level $n$ generated by the triangular system $(\W_1\cup \W'_1,\Y_1 \cup \Y'_1)$ and $\Y_n \cup \Y'_n$ be the set of yellow triangles  of level $n$ generated by the triangular system $(\W_1\cup \W'_1,\Y_1\cup \Y'_1)$.
Then, the pair  $(\W_n \cup \W'_n, \Y_n \cup \Y'_n)$ (see Definition \ref{def:triangular labyrinth patterns system}) has the tree property, the exits property and the corner property.
\end{proposition}

We call $\W_n \cup \W'_n$ the \emph{white triangular labyrinth set of level} $n$ and  $\Y_n \cup \Y'_n$ \emph{the yellow triangular labyrinth set of level} $n$.

\begin{proof}
We prove the above assertion by induction. For $n=1$ it holds, because $(\W \cup \W', \Y\cup \Y')$ is an $m$-triangle labyrinth patterns system.
Now, we need to prove that $(\W_n \cup \W'_n, \Y_n \cup \Y'_n)$ has the tree, the exits and the corner property, for $n\ge 2$. Therefore, we now assume that these properties hold for $(\W_{n-1} \cup \W'_{n-1}, \Y_{n-1} \cup \Y'_{n-1})$. 

From the structure of an $m$-triangle labyrinth patterns system and the construction  of the white and yellow labyrinth set of level $n$, respectively, one can see that both the exits property and the corners property are satisfied by the sets $\W_n$ and $\Y_n$.

In order to prove the tree property, let us first notice that $\G(\W \cup\W')$ and $\G(\Y \cup\Y')$ are trees, by the proposition's assumption and, by the induction hypothesis, $\G(\W_{n-1} \cup\W'_{n-1})$ and $\G(\Y_{n-1} \cup\Y'_{n-1})$ are trees, which ensures that both graphs are connected. Therefore, by the construction of $\W_n, \W'_n, \Y_n$ and $\Y'_n$, $\G(\W_{n} \cup\W'_{n})$ and $\G(\Y_{n} \cup\Y'_{n})$ are connected. We have to prove that there is no cycle in $\G(\W_{n} \cup\W'_{n})$ and $\G(\Y_{n} \cup\Y'_{n})$. Due to the dual character of the construction of these sets and graphs, we only prove this property for $\G(\W_{n} \cup\W'_{n})$, since the proof for $\G(\Y_{n} \cup\Y'_{n})$ is analogous.

In order to give an indirect proof, let us assume that there exists a cycle 
$c=\{U_0,U_1,\dots,U_s\}$ in $\G(\W_{n} \cup\W'_{n})$. For $U \in \W_{n} \cup\W'_{n}$ let $t(U)$ be the white triangle in $\W_{n-1} \cup\W'_{n-1}$ that contains $U$ (as a subset of points in the plane). Let $j_0:=0$, $V_0:=t(u_0)$ and $j_k:=\min\{i: t(U_i)\neq V_{k-1}, j_{k-1}<i\le s\}$, $V_k:=t(U_{j_k})$ for $k\ge 1$. Now we choose $r$ minimal such that the set $\{  i: \, t(U_i)\ne V_r, j_r<i\le s \}$ is empty.  For $i=1, \dots, r$ the vertex {$V_{i-1}$ is a neighbour of $V_i$ in $\G(\W_{n-1} \cup\W'_{n-1})$}. If the graph which is induced by $\G(\W_{n-1} \cup\W'_{n-1})$ on the set of vertices $\{ V_0, V_1, \dots,V_r\}$ contains a cycle in $\G(\W_{n-1} \cup\W'_{n-1})$, this contradicts the induction hypothesis. On the other hand,  $\G(\W_{1} \cup\W'_{1})$ is a tree, thus not all (white) triangles of the cycle $c$ in  $\G(\W_{n} \cup\W'_{n})$ can be contained in $V_0$, which implies $r\ge 1$. Thus the graph induced by $\G(\W_{n-1} \cup\W'_{n-1})$ on the set $\{ V_0, V_1, \dots,V_r\}$ is a tree with more than one vertex. This implies that  the cycle $c$ returns to $V_0$ through the same side where it leaves the triangle $V_0$. Since $\W_1$ and $\W'_1$ have the exits property (only one exit on each side of the initial triangle), this comes in contradiction with the assumption that $U_0,U_1,\dots,U_s$  are pairwise distinct.
\end{proof}

\begin{remark}
In other words, the triangular system $(\W_n \cup \W'_n, \Y_n \cup \Y'_n)$ consisting of the white and the yellow labyrinth set of level $n$ obtained in the $n$-th iteration is, in fact, a $m^n$-triangular labyrinth patterns system. 
Nevertheless, we  use here two distinct notions, of patterns and sets, in order to emphasise the idea that the construction of triangular labyrinth sets of level $n$ is based on triangular labyrinth patterns systems.
\end{remark}

For  a  triangular labyrinth patterns system $(\W \cup \W', \Y\cup \Y')$ and $n \geq 1$  we denote the exits of $\W_n \cup \W'_n$ and  $\Y_n\cup \Y'_n$ by $e_i^{\w}(n)$ and $e_i^{\y}(n)$, respectively, where $i \in \{1,2,3\}$. With this notation we obviously have, $e_i^{\w}= e_i^{\w}(1)$ and $e_i^{\y}= e_i^{\y}(1)$. We call these the exits of the labyrinth sets of level $n$.

\begin{proposition}\label{prop:no_double_exits}
Let $(\W \cup \W', \Y \cup \Y')$ be an $m$-triangular labyrinth patterns system. Then, for any $i,j \in \{1,2,3\}$ with $i \ne j$ we have $e_i^{\w}\ne e_j^{\w}$ and $e_i^{\y}\ne e_j^{\y}.$ 
\end{proposition}
\begin{proof} We give, w.l.o.g, an indirect proof for   $i=1$ and $j=2$ in the white pattern.
Let $e^\w_1 = T_m(0,k_1,m-1-k_1)$ and $e^\w_2 = T_m(k_2,0,m-1-k_2)$.
Assume  $e^\w_1 = e^\w_2$, which implies $k_1=k_2=0$.
By the exits property in Definition~\ref{def:triangular labyrinth patterns system} we conclude that $\Y$ contains $T_m(0,m-1-k_1,k_1)=T_m(0,m-1,0)$ as well as $T_m(m-1-k_2,0,k_2)=T_m(m-1,0,0)$ which are both corner triangles. But this is a contradiction to the corner property. This completes the proof.
\end{proof}
In other words, the above proposition asserts that no corner triangle in $\W$ or $\Y$ can be a ``double'' exit. 
One can see that Proposition~\ref{prop:no_double_exits} also holds for the pair $(\W_n \cup \W'_n, \Y_n \cup \Y'_n)$,  for  all  $n\ge 1$, due to Proposition \ref{prop:triangle_laby_system_order_n}.

%-----------------------------Section 3------------------------------------------

\section{Triangular labyrinth fractals}\label{sec:TLF}

In this section we introduce triangular labyrinth fractals and study them, e.g., with respect to their fractal dimension.  Nevertheless, the considerations throughout the section hold for fractals obtained from $m$-triangular patterns systems in general.

Let $(\W \cup \W', \Y\cup \Y')$ denote an $m$-triangular labyrinth  patterns system and, for $n\ge 1$, $\W_n \cup \W'_n$ and $\Y_n \cup \Y'_n$ the corresponding sets of white and, respectively, of yellow triangles of level $n$, as defined earlier. 

Let 
\[
L^{\w}_n:=\bigcup_{W \in \W_n} W \cup \bigcup_{W' \in \W'_n}W' \text  { and }  L^{\y}_n:=\bigcup_{Y \in \Y_n} Y \cup \bigcup_{Y' \in \Y'_n}Y'.
\] 
Then $\{L^{\w}_n\}_{n\ge 1}$ and $\{L^{\y}_n\}_{n\ge 1}$  are decreasing sequences of (white, and respectively, yellow) compact sets in the plane, i.e., 
\begin{align*}
&T_1 \supset  L^{\w}_1 \supset L^{\w}_2 \supset L^{\w}_3 \supset \cdots \\
&T_1 \supset  L^{\y}_1 \supset L^{\y}_2 \supset L^{\y}_3 \supset \cdots
\end{align*}

\begin{definition}\label{def:TLF}
We call the limit set of the nested sequence $\{L^{\w}_n\}_{n\ge 1}$,
\begin{equation}
\linfw:= \bigcap_{n\ge1} \Big(\bigcup_{W \in \W_n} W \cup \bigcup_{W' \in \W'_n}W' \Big)  \, =\bigcap_{n\ge1}L^{\w}_n 
\end{equation}
the \emph{white triangular labyrinth fractal} generated by the triangular labyrinth patterns system  $(\W \cup \W', \Y\cup \Y')$. Correspondingly, we call  the limit set of the nested sequence $\{L^{\y}_n\}_{n\ge 1}$
\begin{equation}
\linfy:=  \bigcap_{n\ge1} \Big(\bigcup_{Y \in \Y_n} Y \cup \bigcup_{Y' \in \Y'_n}Y' \Big)  \, = \bigcap_{n\ge1}L^{\y}_n   
\end{equation}
the \emph{yellow triangular labyrinth fractal} generated by the triangular labyrinth pattern system  $(\W \cup \W', \Y\cup \Y')$
\end{definition}

In order to study the sets $\linfw$ and $\linfy$ we first show 
that they can be obtained from a graph directed construction in the sense of Mauldin and Williams~\cite{MW}. 
This provides certain informations about these fractals, like, e.g., the Hausdorff dimension.

Therefore, let us now introduce the GIFS (graph directed iterated functions systems) framework. (For more details on GIFS see, e.g., the references  \cite{Edgar, EM, MW}.)
We now consider a directed multigraph $\HH$, i.e., each edge has a direction given by the ordered pair of vertices that it connects and 
there is possibly more than one edge from one vertex to another one. The edges can be distinguished by their labels. Furthermore, we also allow edges that start and end in the same vertex (loops). A path in a directed (multi)graph is a sequence of edges (differently from the case of undirected graphs mentioned earlier). 

Let $\HH = (\V(\HH), \E(\HH))$ denote the directed multigraph with (two) vertices $\V(\HH):=\{\Omega, \Gamma\}$. The set of edges $\E(\HH)$ is determined by the sets $\W$, $\W'$, $\Y$, and $\Y'$ of white and yellow triangles. Concretely,
for every $W \in \W$ there is an edge from $\Omega$ to $\Omega$ labelled by $W$ and for every $W' \in \W'$ there is an edge from $\Omega$ to $\Gamma$ labelled by $W'$. Likewise, every $Y \in \Y$ induces an edge from $\Gamma$ to $\Gamma$ labelled by $Y$ and every $Y' \in \Y'$ gives an edge from $\Gamma$ to $\Omega$ labelled by $Y'$. 
For a positive integer $n$ let $\HH^n$ be the set of paths of length $n$ of $\HH$. 
The paths are denoted by the corresponding sequence of labels.
We write $\HH^n(\Omega)$ for the set consisting of elements of $\HH^n$ that start in $\Omega$ and 
 $\HH^n(\Gamma)$ for those starting in $\Gamma$. 

Note that the labels of the edges of $\HH$ correspond to triangles. We associate with every edge $\Delta\in \E(\HH)$ the map $P_\Delta$. For a path $(\Delta_{j})_{j=1}^n$ we write $(\Delta_{j})_{j=1}^n = \Delta_1, \ldots, \Delta_n \in \HH^n$, and let $P_{\Delta_1, \ldots, \Delta_n}$  denote the composition $P_{\Delta_1} \circ \cdots \circ P_{\Delta_n}$.

\begin{lemma}\label{lemma:pathsinH}
For every integer $n \geq 1$ we have
\begin{align*}
\W_n = & \{P_{\Delta_1, \ldots, \Delta_n}(T_1) : (\Delta_{j})_{j=1}^{n} \in \HH^n(\Omega) \text{ is a path that terminates in } \Omega\}, \\
\W'_n = & \{P_{\Delta_1, \ldots, \Delta_n}(T_1) : (\Delta_{j})_{j=1}^{n} \in \HH^n(\Omega) \text{ is a path that terminates in } \Gamma\}, \\
\Y_n = & \{P_{\Delta_1, \ldots, \Delta_n}(T_1) : (\Delta_{j})_{j=1}^{n} \in \HH^n(\Gamma) \text{ is a path that terminates in } \Gamma\}, \\
\Y'_n = & \{P_{\Delta_1, \ldots, \Delta_n}(T_1) : (\Delta_{j})_{j=1}^{n} \in \HH^n(\Gamma) \text{ is a path that terminates in } \Omega\}. 
\end{align*}
\end{lemma}
\begin{proof}
We prove the lemma by induction on $n$. 
First, observe that $P_{\Delta}(T_1)=\Delta = P_{T_1}(\Delta)$ for every $\Delta \subset T_1$.
This yields the assertion for $n=1$ .

Now let $n >1$ and assume the assertions of the lemma to be true for $n-1$. We restrict ourselves to $\W_n$ (for the other sets the proof works analogously).
By definition we have
\[\W_n:=\{P_{W_{n-1}}(W_1): W_1 \in \W_1, W_{n-1} \in \W_{n-1}\} \cup \{P_{W'_{n-1}}(Y'_1): Y'_1 \in \Y'_1, W'_{n-1} \in \W'_{n-1}\}.\]
By the induction hypothesis and the structure of $\HH$ we obtain
\begin{align*}
\W_n=&\left\{P_{P_{\Delta_1,\ldots,\Delta_{n-1}}(T_1)}(\Delta_n): \substack{(\Delta_{j})_{j=1}^{n-1} \in \HH^{n-1}(\Omega) \text{ is a path that terminates in } \Omega \\ \text{and } \Delta_n \in \HH^1(\Omega) \text{ is an edge that terminates in } \Omega}\right\} \\
& \cup \left\{P_{P_{\Delta_1,\ldots,\Delta_{n-1}}(T_1)}(\Delta_n): \substack{(\Delta_{j})_{j=1}^{n-1} \in \HH^{n-1}(\Omega) \text{ is a path that terminates in } \Gamma \\ \text{and } \Delta_n \in \HH^1(\Gamma) \text{ is an edge that terminates in } \Omega}\right\} \\
= & 
\left\{P_{P_{\Delta_1}\circ \cdots \circ P_{\Delta_{n-1}(T_1)}}(\Delta_n): (\Delta_{j})_{j=1}^{n} \in \HH^{n}(\Omega) \text{ is a path that terminates in }  \Omega \right\}.
\end{align*}
Now we claim that for $\Delta_A:=\Delta(\mathbf{A}_1, \mathbf{A}_2,\mathbf{A}_3)$, $\Delta_B:=\Delta(\mathbf{B}_1, \mathbf{B}_2,\mathbf{B}_3)$ we have  
\begin{equation}\label{eq:compose_projections}
P_{\Delta_A} \circ P_{\Delta_B} = P_{P_{\Delta_A}(\Delta_B)}. 
\end{equation}
Indeed, for every $i \in \{1,2,3\}$, let $(\beta^{(i)}_1, \beta^{(i)}_2, \beta^{(i)}_3) \in H$ denote the homogeneous coordinates of $\mathbf{B}_i$  (with respect to $\mathbf{P}_1$, $\mathbf{P}_2$, $\mathbf{P}_3$).
Then for every $\mathbf{x} \in T_1$ with homogeneous coordinates $(\alpha_1, \alpha_2, \alpha_3) \in H$ we have
\[P_{\Delta_B}(\mathbf{x}) = 
\sum_{i=1}^3 \alpha_i\mathbf{B}_i = \sum_{i=1}^3 \sum_{j=1}^3 \alpha_i\beta_j^{(i)}\mathbf{P}_j,\]
hence
$\left(\sum_{i=1}^3 \alpha_i\beta_1^{(i)}, \sum_{i=1}^3 \alpha_i\beta_2^{(i)}, \sum_{i=1}^3 \alpha_i\beta_3^{(i)}\right) \in H$
are the homogeneous coordinates of $P_{\Delta_B}(\mathbf{x})$. Thus,
\begin{multline*}
P_{\Delta_A} \circ P_{\Delta_B}(\mathbf{x}) = 
\sum_{j=1}^3 \sum_{i=1}^3 \alpha_i\beta_j^{(i)}\mathbf{A}_j
=\sum_{i=1}^3 \alpha_i \left(\sum_{j=1}^3\beta_j^{(i)}\mathbf{A}_j\right) \\
= P_{\Delta(P_{\Delta_A}(\mathbf{B}_1),P_{\Delta_A}(\mathbf{B}_2),P_{\Delta_A}(\mathbf{B}_3))}(\mathbf{x})
=P_{P_{\Delta_A}(\Delta_B)}(\mathbf{x}).
\end{multline*}
Form \eqref{eq:compose_projections} and the above formul{\ae} for $\W_n$ we obtain
\begin{align*}
\W_n=&\left\{P_{\Delta_1}\circ \cdots \circ P_{\Delta_{n-1}}\circ P_{T_1}(\Delta_n): (\Delta_{j})_{j=1}^{n} \in \HH^{n}(\Omega) \text{ is a path that terminates in } \Omega \right\} \\
= &\left\{P_{\Delta_1}\circ \cdots \circ P_{\Delta_{n-1}}\circ P_{\Delta_n}(T_1): (\Delta_{j})_{j=1}^{n} \in \HH^{n}(\Omega) \text{ is a path that terminates in } \Omega \right\} \\
= &\left\{P_{\Delta_1,\ldots,\Delta_n}(T_1): (\Delta_{j})_{j=1}^{n} \in \HH^{n}(\Omega) \text{ is a path that terminates in } \Omega \right\}.
\end{align*}
\end{proof}

The above lemma and \eqref{eq:compose_projections} immediately yield the following result.
\begin{corollary}
Let $n$ be a positive integer. Then
\[
L^{\w}_n = \bigcup_{(\Delta_{j})_{j=1}^{n} \in \HH^n(\Omega)} P_{\Delta_1, \ldots, \Delta_n}(T_1),
\qquad
L^{\y}_n = \bigcup_{(\Delta_{j})_{j=1}^{n} \in \HH^n(\Gamma)} P_{\Delta_1, \ldots, \Delta_n}(T_1).
\]
For integers $n, k$  with $0 < k < n$ we have
\begin{align*}
\W_n:=&\{P_{W_{n-k}}(W_k): W_k \in \W_k, W_{n-k} \in \W_{n-k}\} \cup \{P_{W'_{n-k}}(Y'_k): Y'_k \in \Y'_k, W'_{n-k} \in \W'_{n-k}\}, \\
\W'_n:=&\{P_{W_{n-k}}(W'_k): W'_k \in \W'_k, W_{n-k} \in \W_{n-k}\} \cup \{P_{W'_{n-k}}(Y_k): Y_k \in \Y_k, W'_{n-k} \in \W'_{n-k}\}, \\
\Y_n:=&\{P_{Y_{n-k}}(Y_k): Y_k \in \Y_k, Y_{n-k} \in \Y_{n-k}\} \cup \{P_{Y'_{n-k}}(W'_k): W'_k \in \W'_k, Y'_{n-k} \in \Y'_{n-k}\}, \\
\Y'_n:=&\{P_{Y_{n-k}}(Y'_k): Y'_k \in \Y'_k, Y_{n-k} \in \Y_{n-k}\} \cup \{P_{Y'_{n-k}}(W_k): W_k \in \W_k, Y'_{n-k} \in \Y'_{n-k}\}.
\end{align*}
\end{corollary}

\begin{proposition}\label{prop:GDIFS}
Let $\W'$ and $\Y'$ be non-empty.
Then the sets $\linfw$ and $\linfy$ have a graph-directed  structure
and satisfy
\begin{align*}
\linfw & =\bigcup_{\Delta \in \W}P_\Delta(\linfw) \cup \bigcup_{\Delta' \in \W'}P_{\Delta'}(\linfy), \\
\linfy & =\bigcup_{\Delta \in \Y}P_\Delta(\linfy) \cup \bigcup_{\Delta' \in \Y'}P_{\Delta'}(\linfw).
\end{align*}
The Hausdorff dimension $\dim_H$ and the box-dimension $\dim_B$ of the two fractals coincide and are given by
\[\dim_H(\linfw) = \dim_H(\linfy) = \dim_B(\linfw) = \dim_B(\linfy) = \frac{\log{\lambda}}{\log m},\]
where
\[\lambda = \frac{\abs{\W} + \abs{\Y} + \sqrt{(\abs{\W} - \abs{\Y})^2 + 4\abs{\W'}\abs{\Y'}}}{2}.\]
\end{proposition}
\begin{proof}
Since $\W'$ and $\Y'$ are non-empty, it follows that $\HH$ is strongly connected. For every edge $\Delta \in \HH^1$ we have $\Delta \in \T_m \cup \T'_m$, thus the map $P_\Delta$ is a similarity of ratio $\nicefrac{1}{m}$. Therefore, our setting fits into the graph directed constructions framework discussed in \cite{MW}.
Note that 
\[\bigcup_{\Delta \in \W}P_\Delta({\rm int}(T_1)) \cup \bigcup_{\Delta' \in \W'}P_{\Delta'}({\rm int}(T_1)) \subset {\rm int}(T_1) \supset \bigcup_{\Delta \in \Y}P_\Delta({\rm int}(T_1)) \cup \bigcup_{\Delta' \in \Y'}P_{\Delta'}({\rm int}(T_1)),\]
where the unions are, by construction, disjoint, and ${\rm int}(T_1)$ denotes the interior of the triangle $T_1$. Hence, the open set condition is fulfilled.
The assertion concerning the dimensions follows from \cite[Theorem~3]{MW}.
\end{proof}
The set equation can be easily extended to paths of arbitrary length in $\HH$. One can see that Lemma~\ref{lemma:pathsinH} immediately yields the following corollary.
\begin{corollary}\label{cor:fractal_as_finite_union}
For all $n\ge 1$ the self-similar sets $\linfw$ and $\linfy$ satisfy 
\begin{align*}
\linfw & =\bigcup_{\Delta \in \W_n}P_\Delta(\linfw) \cup \bigcup_{\Delta' \in \W'_n}P_{\Delta'}(\linfy), \\
\linfy & =\bigcup_{\Delta \in \Y_n}P_\Delta(\linfy) \cup \bigcup_{\Delta' \in \Y'_n}P_{\Delta'}(\linfw).
\end{align*}
\end{corollary}

\section{Topological properties of triangular labyrinth fractals}\label{sec:topological_prop}

For $\W \subset \T_m$ and $\W' \subset \T'_m $, let $\B:=\T_m\setminus \W $ and 
$\B':=\T'_m\setminus \W' $. Analogously, for $\Y \subset \T_m$ and $\Y' \subset \T'_m $, let $\D:=\T_m\setminus \Y $ and 
$\D':=\T'_m\setminus \Y' $. We call $\B \cup \B'$ the set of w-black triangles and $\D \cup \D'$ the set of y-black triangles. We define $\G(\B \cup \B')$ as the graph with vertex set $\B \cup \B'$, such that two distinct vertices $B_1,B_2\in \V(\G(\B\cup \B'))$ are connected by an edge in the graph if one of the following three conditions is satisfied:
\begin{enumerate}
\item the black triangles $B_1$ and $B_2$ intersect along a common side;

\item $B_1$ and $B_2$ intersect in exactly one point which does not lie on the boundary of $T_1$, i.e.,
\[B_1 \cap  B_2 = \{\alpha_1\mathbf{P}_1 + \alpha_2\mathbf{P}_2 +\alpha_3\mathbf{P}_3\} \text{ where } (\alpha_1, \alpha_2, \alpha_3) \in H, \text{ and } \alpha_1\alpha_2\alpha_3 \not=0;\]

\item 
$B_1$ and $B_2$ are border triangles that intersect in exactly one point, which lies on the side $\varrho_i$ of the triangle $T_1$, 
%which is opposite to its vertex $P_i$, 
for some $i\in\{ 1,2,3\}$, and the corner triangle $T_m(k_1, k_2, k_3)$, where $k_i=m-1$ and $k_j=0$ for $j \in \{ 1,2,3\}\setminus \{i\}$, is not an element of in $\Y$. 
\end{enumerate}

$\G(\D \cup \D')$ is defined analogously, we just have to replace $\Y$ by $\W$ in the third condition occurring in the above definition.

Let us remark that two distinct vertices in $\G(\B\cup \B')$ are connected by an edge (and thus are neighbours) in the graph $\G(\B \cup \B')$ if and only if the corresponding  black triangles are neighbours (as defined in Section \ref{sec:construction}), except in the case when they intersect on the boundary $\partial T_1$ and the above condition (3) is violated. The analogon holds for $\G(\D\cup \D')$.

In the following we also use the notations $\B_n$ and $\B'_n$, $\D_n$ and $\D'_n$ for the corresponding elements of level $n$, and 
 $\G(\B_n \cup \B'_n)$ and $\G(\D_n \cup \D'_n)$ for the corresponding graphs, respectively.

\begin{lemma}\label{lemma:steinhaus}
If $\W \cup \W'$ is a set of white triangles such that the graph $\G(\W \cup \W')$ is a tree, then from every (black) triangle  in $\B \cup \B'$ there is a path in $\G(\B \cup \B')$ to a border triangle. An analogous result holds for the set of yellow triangles $\Y \cup \Y'$ with respect to $\D \cup \D'$.
\end{lemma}

\begin{proof} 
We give an indirect proof. Therefore, we assume that there is a black triangle $B$ in $\V(\G( \B \cup \B'))$, such that the connected component $\C_B$ of $B$ in $\G( \B \cup \B')$ contains no border triangle.  
By the lemma's assumption, $\G(\W \cup \W')$ is a tree, therefore there exists a triangle $W \in \W \cup \W'$, such that $W$ has only one neighbour in $\G(\W \cup \W')$. 
We consider, w.l.o.g, that $W \in \W$. Otherwise, we can rotate the whole planar object to get $W \in \W$. Let us assume, without loss of generality, that this neighbour lies below $W$, i.e., is the $3$-neighbour of the triangle $W$. Now we introduce new graphs $\G'(\W \cup \W')$ and $\G'(\B \cup \B')$, which are obtained from $\G(\W \cup \W')$ and $\G(\B \cup \B')$ by colouring $W$ in black. The next step is to show that our indirect assumption on $B$ still holds in $\G'(\B \cup \B')$. Therefore, we assume the contrary, i.e., that the new connected component of $B$ in $\G'(\B \cup \B')$ contains a border triangle. Then it follows that $W$ has a black neighbour $B_1$ in  $\G'(\B \cup \B')$ such that $B_1 \in\C_B$. Here, two  cases may occur. In the first of these cases, $W$ is a border triangle, and in the second case $W$ is not a border triangle, but has a neighbour $B_2$ from which there exists a path in $\G'(\B \cup \B')$ to a black border triangle.
In the first case, we may assume w.l.o.g. that $W$ lies in the border strip  of type $2$ (i.e., its side of type $2$ lies on $\varrho_2$). Then the triangle $W$ has a black $1$-neighbour. It follows that $B_1$ either lies in the border strip of type $2$, or it is a neighbour of a black triangle of $\G(\B \cup \B')$ that lies in the border strip of type $2$.
No matter where  $B_1$ (the black neighbour of $W$ in  $\G'(\B \cup \B')$) lies, there is a path in $\G(\B \cup \B')$ from $B_1$ to a black border triangle, i.e., $\C_B$ contains a black border triangle, which leads to a contradiction.

If $W$ is not a border triangle, then the $2$- and $1$-neighbours of $W$ 
are black, which implies that all neighbour triangles of $W$ in $\G'(\B \cup \B')$ are in the same connected component in $\G'(\B \cup \B')$. Since $B_1$ and $B_2$ have to be in distinct connected components of $\G'(\B \cup \B')$, this is a contradiction.

Now, we may proceed with colouring white triangles with at most one neighbour until no such triangle is left. It is not possible that the new graph $\G^{new}(\W \cup \W')$ that arises at the end of  this colouring procedure is empty, because in that case the connected component of $B$ would contain a border triangle, which contradicts our assumption.

The other possibility would be that all white triangles in $\G^{new}(\W \cup \W')$ have at least two neighbours in $\G^{new}(\W \cup \W')$, which implies that $\G(\W \cup \W')$ is not a tree, in contradiction with the lemma's assumption.
\end{proof}

Let us remark that the obove result is a Steinhaus chessboard type theorem dealing with a triangular chessboard. 

\begin{lemma}\label{lemma:arc_in_the_complement}
Let $n\ge 1$. If $x$ is a point in $T_1 \setminus L^{\w}_n$, then there is an arc $a$ in $T_1 \setminus L^{\w}_{n+1 }$ that connects $x$ and a point in the boundary $\fr(T_1)$.  An analogous result holds for $x \in T_1 \setminus L^{\y}_{n+1}$.
\end{lemma}
\begin{proof} 
Let $x$ be a point in  $T_1 \setminus L^{\w}_n$. Thus $x$ lies in a black triangle $B_0 \in \G(\B_n \cup \B'_n )$. By Lemma \ref{lemma:steinhaus} applied to $\W_n \cup \W'_n$, there exists a path $\{B_0, B_1, \dots,B_k\}$, with $k\ge 0$, in $\G(\B_n \cup \B'_n)$ from $B_0$ to a black border triangle $B_k$. Let us first construct the arc $ a'\subset T_1 \setminus L^{\w}_{n +1} $ as the union of arcs between centers of consecutive black triangles in the path $\{B_0, B_1, \dots,B_k\}$. Therefore, we proceed as follows.
When $B_j$ and $B_{j+1}$, with $0\le j\le k-1$, are consecutive black triangles in the above path, their intersection is either a common side or  only a point. 
In the first case, we can construct an arc that connects their centres as an arc passing through the mid-point of the segment (the side) that is their intersection.
If the intersection of $B_j$ and $B_{j+1}$ is only one point (their common vertex), which we denote for convenience by $P_{j,j+1}$, then  for this point we have either $P_{j,j+1}\in L^{\w}_n$ or $P_{j,j+1}\in T_1\setminus L^{\w}_n$. 

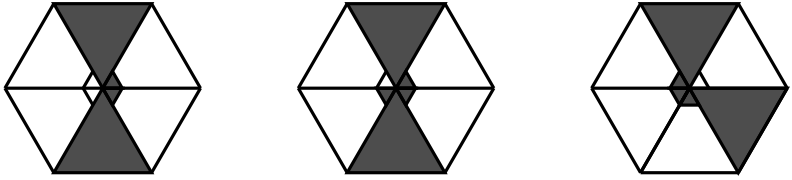
\begin{figure}[h]
\begin{center}
 \begin{tikzpicture}
[scale=1.30]
\draw[line width=1.2pt, draw= black] (-1, 0) -- (-0.5,0.865);
\draw[line width=1.2pt, draw= black] (1, 0) -- (0.5,0.865);
\draw[line width=1.2pt, draw= black] (-1, 0) -- (1,0);
\draw[line width=1.2pt, draw= black] (-1, 0) -- (-0.5,-0.865);
\draw[line width=1.2pt, draw= black] (1, 0) -- (0.5,-0.865);
\draw[line width=1.2pt, draw=black, fill=black, fill opacity=0.7]
       (0,0) -- (0.5,0.865) -- (-0.5,0.865) -- cycle;
\draw[line width=1.2pt, draw=black, fill=black, fill opacity=0.7]
       (0,0) -- (0.5,-0.865) -- (-0.5,-0.865) -- cycle;
%jetzt noch die kloanen dreieckerl dazu
\draw[line width=1.2pt, draw= black] (-0.2, 0) -- (-0.1,0.173);
\draw[line width=1.2pt, draw= black] (-0.2, 0) -- (-0.1,-0.173);
\draw[line width=1.2pt, draw=black, fill=black, fill opacity=0.7]
       (0,0) -- (0.2, 0) -- (0.1,0.173) -- cycle;
\draw[line width=1.2pt, draw=black, fill=black, fill opacity=0.7]
       (0,0) -- (0.2, 0) -- (0.1,-0.173) -- cycle;

% das zweite sechseck
\draw[line width=1.2pt, draw= black] (2, 0) -- (2.5,0.865);
\draw[line width=1.2pt, draw= black] (4, 0) -- (3.5,0.865);
\draw[line width=1.2pt, draw= black] (2, 0) -- (4,0);
\draw[line width=1.2pt, draw= black] (2, 0) -- (2.5,-0.865);
\draw[line width=1.2pt, draw= black] (4, 0) -- (3.5,-0.865);
\draw[line width=1.2pt, draw=black, fill=black, fill opacity=0.7]
       (3,0) -- (3.5,0.865) -- (2.5,0.865) -- cycle;
\draw[line width=1.2pt, draw=black, fill=black, fill opacity=0.7]
       (3,0) -- (3.5,-0.865) -- (2.5,-0.865) -- cycle;

%jetzt noch die kloanen dreieckerl dazu
\draw[line width=1.2pt, draw= black] (2.8, 0) -- (2.9,0.173);
\draw[line width=1.2pt, draw=black, fill=black, fill opacity=0.7]
       (3,0) -- (3.2, 0) -- (3.1,0.173) -- cycle;
\draw[line width=1.2pt, draw=black, fill=black, fill opacity=0.7]
       (3,0) -- (3.2, 0) -- (3.1,-0.173) -- cycle;
\draw[line width=1.2pt, draw=black, fill=black, fill opacity=0.7]
       (3,0) -- (2.8, 0) -- (2.9,-0.173) -- cycle;

% das dritte sechseck
\draw[line width=1.2pt, draw= black] (5, 0) -- (5.5,0.865);
\draw[line width=1.2pt, draw= black] (6, 0) -- (5.5,-0.865);
\draw[line width=1.2pt, draw= black] (5, 0) -- (7,0);
\draw[line width=1.2pt, draw= black] (5, 0) -- (5.5,-0.865);
\draw[line width=1.2pt, draw= black] (7, 0) -- (6.5,-0.865);
\draw[line width=1.2pt, draw= black] (6.5, 0.865) -- (7,0);
\draw[line width=1.2pt, draw= black] (5.5,-0.865) -- (6.5,-0.865);
\draw[line width=1.2pt, draw= black] (5.5,-0.865) -- (6,0);
\draw[line width=1.2pt, draw=black, fill=black, fill opacity=0.7]
       (6,0) -- (6.5,0.865) -- (5.5,0.865) -- cycle;
\draw[line width=1.2pt, draw=black, fill=black, fill opacity=0.7]
       (6,0) -- (6.5,-0.865) -- (7,0) -- cycle;
% die kloanen dreieckerl dazu
\draw[line width=1.2pt, draw=black, fill=black, fill opacity=0.7]
       (6,0) -- (5.8, 0) -- (5.9,-0.173) -- cycle;
\draw[line width=1.2pt, draw=black, fill=black, fill opacity=0.7]
       (6,0) -- (5.8, 0) -- (5.9,0.173) -- cycle;
\draw[line width=1.2pt, draw=black, fill=black, fill opacity=0.7]
       (6,0) -- (5.9, -0.173) -- (6.1,-0.173) -- cycle;
\draw[line width=1.2pt, draw= black] (6.2, 0) -- (6.1,0.173);

\end{tikzpicture}
\end{center}
\caption{Positions of black triangles  $\B_{n} \cup \B'_{n}$ that intersect in one point, which does not lie on the boundary of $T_1$, and the black triangles from $\B_{n+1} \cup \B'_{n+1}$ that play a role in the construction described in the proof of Lemma \ref{lemma:arc_in_the_complement}. Here black triangles are coloured in grey for better visibility in the figure.}
\label{fig:black_triangles_1}
\end{figure}

 Let us first analyse the case when $P_{j,j+1}\in T_1 \setminus L^{\w}_n$. In this case, it is either the common vertex of six black triangles from $\B_n$, such that neighbouring black triangles share a common edge (i.e., $P_{j,j+1}$ is the centre of the regular hexagon which is the union of the six black triangles,) as in Figure \ref{fig:black_triangles_1}, but having all triangles or order $n$ coloured in black), or it is the common vertex of three black triangles from $\B_n$ such that $P_{j,j+1}$ lies on one of the sides of $T_1$ and the three black triangles of level $n$ are like the triangles shown in Figure \ref{fig:black_triangles_2}, but all coloured in black. It is easy to see that in these cases where the point $P_{j,j+1}$ is ``surrounded'' by black triangles there is no difficulty to construct a curve from any centre or midpoint of a side of one of these triangles  to any centre or midpoint of a side of an other among these black triangles, such that this curve passes through $P_{j,j+1}$ and is contained in $T_1\setminus L^\w_n$.

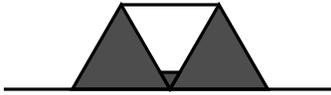
\begin{figure}[h]
\begin{center}
 \begin{tikzpicture}
[scale=1.30]

\draw[line width=1.2pt, draw= black] (-1.7, 0) -- (1.7,0);
\draw[line width=1.2pt, draw= black] (-0.5,0.865) -- (0.5,0.865);
\draw[line width=1.2pt, draw=black, fill=black,fill opacity=0.7]
       (0,0) -- (-0.5,0.865) -- (-1,0) -- cycle;
\draw[line width=1.2pt, draw=black, fill=black,fill opacity=0.7]
       (0,0) -- (0.5,0.865) -- (1,0) -- cycle;
%
%%jetzt noch die kloanen dreieckerl dazu
\draw[line width=1.2pt, draw=black, fill=black,fill opacity=0.7]
       (0,0) -- (0.1,0.173) -- (-0.1,0.173) -- cycle;
\end{tikzpicture}
\end{center}
\caption{Black triangles from $\B_{n} \cup \B'_{n}$ that intersect in one point, which lies on the boundary of $T_1$ or of a white triangle from some $\W_k\cup \W'_k$, $k<n$, that contains them.}
\label{fig:black_triangles_2}
\end{figure}

For the case when $P_{j,j+1}\in L^{\w}_n$ the argumentation works as follows.
 By the corner property and the definition of $\G(\B \cup \B')$ in all situations that can occur in this case one can construct an arc that starts at the centre of $B_j$, ends at the centre of $B_{j+1}$, and goes through the centre(s) of one, two or three (neighbouring) black triangles from $\B_{n+1}\cup \B'_{n+1}$, such there if there is more than one such small triangle, then the triangles share pairwise a side, and the first and last of them has a side lying on a side of $B_j$ and $B_{j+1}$, respectively. In other words, when $B_j$ and $B_{j+1}$ intersect in only one point, we can construct a curve in  $T_1 \setminus L^{\w}_{n+1}$ that connects their centres and passes ``close'' to $P_{j,j+1}$, but makes a detour around this point, passing through black triangles that belong to $\B_{n+1}\cup \B'_{n+1}$. 
Figures \ref{fig:black_triangles_1}  and \ref{fig:black_triangles_2} show, up to symmetry and rotation, what cases can occur when $B_j$ and $B_{j+1}$ are black triangles whose intersection is one point. 

\begin{figure}[h]
\begin{center}
 \begin{tikzpicture}
[scale=1.30]
\draw[line width=1.2pt, draw= black] (-1.8, 0) -- (1.3,0);

\draw[line width=1.2pt, draw=black, fill=black,fill opacity=0.7]
       (0,0) -- (-0.5,0.865) -- (-1,0) -- cycle;

jetzt noch die kloanen dreieckerl dazu
\draw[line width=1.2pt, draw= black] (0, 0) -- (0.5,0.865);
\draw[line width=1.2pt, draw= black] (-0.5, 0.856) -- (0.5,0.865);
\draw[line width=1.2pt, dashed, draw=black]
       (-1,0) -- (-0.5,0.865) -- (-1.5,0.865) -- cycle;

% das zweite sechseck
\draw[line width=1.2pt, draw= black] (2, 0) -- (2.5,0.865);
\draw[line width=1.2pt, draw= black] (4, 0) -- (3.5,0.865);
\draw[line width=1.2pt, draw= black] (1.7, 0) -- (4.3,0);

\draw[line width=1.2pt, draw=black, fill=black, fill opacity=0.7]
       (3,0) -- (3.5,0.865) -- (2.5,0.865) -- cycle;

\draw[line width=1.2pt, draw= black] (2.8, 0) -- (2.9,0.173);
\draw[line width=1.2pt, draw=black, fill=black, fill opacity=0.7]
       (3,0) -- (3.2, 0) -- (3.1,0.173); -- cycle;

\end{tikzpicture}
\end{center}
\caption{ Black border triangles from $\B_{n} \cup \B'_{n}$. In the left example the dashed triangle can be white or black}
\label{fig:black_triangles_3}
\end{figure}
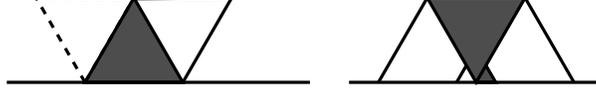

Finally, we construct the arc $a \subset T_1 \setminus L^{\w}_n$ as the union of $a'$ with an arc from $x$ to the centre of $B_0$ and with a further arc which connects the centre of $B_k$ to a point on the boundary $\fr(T_1)$ that lies in $T_1 \setminus L^{\w}_{n+1}$.
If $B_k$ intersects the boundary $\fr(T_1)$ along one of its sides, then we simply connect the centre of $B_k$ with the mid-point of that side. If $B_k$ intersects the boundary $\fr(T_1)$ only in one point, then there exists a triangle $B_{k+1}$ in $\B_{n+1}\cup \B'_{n+1}$ such that one of its sides lies on a side of $B_k$, and another of its sides lies on the boundary $\fr(T_1)$,  and we can construct, e.g., a curve that starts at the centre of $B_k$, passes through  the mid-point of the side of $B_{k+1}$ that lies on one side of $B_{k}$ and then ends at the midpoint of the side of $B_{k+1}$ that lies on boundary $\fr(T_1)$, see also Figure \ref{fig:black_triangles_3}. This completes the proof.

The proof for the yellow fractal $\linfy$ works analogously. 
\end{proof}

The above lemma and the construction of triangular labyrinth fractals immediately yield the following result.

\begin{proposition}\label{prop:arc_in_the_complement}
If $x$ is a point in $T_1 \setminus \linfw$ then there exists an arc $a \subset T_1 \setminus \linfw$ that connects  $x$ and a point in $\fr(T_1)$. An analogous result holds for $x \in T_1 \setminus \linfy$.
\end{proposition}

\begin{theorem}\label{theo:dendrite}
$\linfw$ and $\linfy$ are dendrites.

\end{theorem}
\begin{proof}
We only give the proof for $\linfw$ since the proof for the yellow fractal works analogously.
The set $\linfw$ is the intersection of compact connected sets and thus is connected. From the construction it immediately follows that for any $\varepsilon> 0$ there exists  an integer $n\ge 1,$ such that the diameter of every triangle in $\W_n \cup \W'_n$ is strictly less than $\varepsilon$. In other words, for any $\varepsilon > 0$, the set $\linfw$ is the finite union of connected sets  of diameter less than $\varepsilon$, by Corollary~\ref{cor:fractal_as_finite_union}. We obtain, by applying the Hahn-Mazurkiewicz-Sierpi\'nski Theorem \cite[Theorem 2, p 256]{Kuratowski}, that $\linfw$ is locally connected. In order to prove that $\linfw$ is a dendrite we give an indirect proof for the fact that $\linfw$ contains no simple closed curve. Therefore, we assume the there is  a simple closed curve $c$ in $\linfw$. The Jordan Curve Theorem then implies that the bounded region $J$ inside $c$ is non-empty and open. Therefrom it follows that there exists an integer $n\ge 2$, such that there is a black triangle $B$ of level $n$ that is contained in $J$. Then, by Proposition \ref{prop:arc_in_the_complement} we obtain that there is an arc in $T_1 \setminus \linfw$ from a point $x_1 \in B$ to a point $x_2$ that lies in $\fr T_1$, which comes in contradiction with the Jordan Curve Theorem, since $x_1$ is inside the curve $c$, and $x_2$ lies outside $c$.

\end{proof}
\end{section}

In the rest of the article we  often formulate the problems that we want to solve on the white fractal $\linfw$, since everything works analogously on the yellow fractal $\linfy$.

\section{Paths in graphs of trianular labyrinth patterns and sets} \label{sec:paths_TLG}
Let us introduce the following notation. For $n\ge 1$ and two distinct triangles $W_1, W_2\in \V(\G(\W_n\cup \W'_n))$, let $p_n(W_1,W_2)$ denote the path in $ \in \G(\W_n \cup \W'_n)$ that connects $W_1$ and $W_2$. The existence and uniqueness of this path is provided by the fact that $\G(\W_n \cup \W'_n)$ is a tree.

For any $m$-triangular labyrinth patterns system  $(\W\cup \W', \Y \cup \Y')$ we introduce the following types of paths in the associated graphs.
For convenience we let $\P$ denote the set of unordered pairs of $\{1,2,3\}$, i.e.,
\[\P:=\{\{1,2\}, \{1,3\}, \{2,3\}\}.\]
For $\{i,j\} \in \P$  we call 
the path $p(e_i^\w, e_j^\w)$ in $\G(\W \cup  \W')$ that connects the exits $e_i^{\w}$ and $e_j^{\w}$ the \emph{ white path of type} $\{i,j\}$. Analogously, we call the  path in $\G(\Y \cup  \Y')$ that connects the exits $e_i^{\y}$ and $e_j^{\y}$ \emph{ the yellow path of type} $\{i,j\}$. By the tree property of $\G(\W \cup  \W')$ and $\G(\Y \cup  \Y')$ these paths are unique. 
Of course,  one can associate to each path between exits of  $\G(\W_n \cup  \W'_n)$ the same types as above, for all $n \geq 2$.

We define the length of any of these paths as the number of triangles in the path. 
For $n\ge 1$  and $1\le i<j\le3$ we denote the length of the white path of type  $\{i,j\}$ in $\G(\W_n \cup  \W'_n)$ by $\ell_{\{i,j\}}(n)$, and the length of the yellow path of type  $\{i,j\}$ in $\G(\Y_n \cup  \Y'_n)$ by 
$\ell'_{\{i,j\}}(n)$.
Observe that all exits are contained in $\W_n$ ($\Y_n$, respectively) and $\G(\W_n \cup  \W'_n)$ ($\G(\Y_n \cup  \Y'_n)$, respectively) is bipartite. Therefore, the length of the white and yellow paths is always an odd integer.
Together with Proposition~\ref{prop:no_double_exits} this shows that $\ell_{\{i,j\}}(n) \geq 3$ and $\ell'_{\{i,j\}}(n) \geq 3$ hold for all $\{i, j\} \in \P$ and $n \geq 1$.
\begin{proposition}\label{prop:path_matrix} With the above notations, let 
$\ell_{\{i,j\}}(0)=\ell'_{\{i,j\}}(0)=1$, for all $\{i,j\} \in \P$.
There are non-negative $3 \times 3$ integer matrices $M_\w$ and $M_\y$ such that for all $n\ge 1$ we have
\begin{equation}\label{eq:matrices_iteration}
\left(
 \begin{array}{l}
\ell_{\{1,2\}}(n) \\  
\ell_{\{1,3\}}(n) \\  
\ell_{\{2,3\}}(n) \\  
\ell'_{\{1,2\}}(n) \\  
\ell'_{\{1,3\}}(n) \\  
\ell'_{\{2,3\}}(n)
\end{array}
\right)
= \left(\begin{array}{cc} M_\w & \tilde M_\w \\ \tilde M_\y& M_\y \end{array}\right)
\left(
 \begin{array}{l}
\ell_{\{1,2\}}(n-1) \\  
\ell_{\{1,3\}}(n-1) \\  
\ell_{\{2,3\}}(n-1) \\  
\ell'_{\{1,2\}}(n-1) \\  
\ell'_{\{1,3\}}(n-1) \\  
\ell'_{\{2,3\}}(n-1)
\end{array}
\right) ,
\end{equation}
with $\tilde M_\w=M_\w - I_3$, $\tilde M_\y=M_\y - I_3$, and $I_3$ is the $3\times 3$ identity matrix.
\end{proposition}

\begin{proof}
As a first step of the proof, let us describe the construction of paths between exits in $\G(\W_n \cup \W'_n)$, for $n \ge 1$. The construction for  $\G(\Y_n \cup \Y'_n)$ works analogously. Note that the method used here is inspired by the construction of paths in the graphs of  (square) labyrinth sets \cite{laby_4x4}.

Let us, without loss of generality, start with the path of type $\{1,2\}$ in $\G(\W \cup \W')$, i.e., the path between the exits $e_1^\w(1)$ and $e_2^\w(1)$, in $\G(\W_1 \cup \W'_1)$. 
Observe that it passes  an odd number of vertices 
$W_1, \ldots, W_{2k+1}$ where $W_1=e_1^\w(1)$ and $W_{2k+1} = e_2^\w(1)$.
We assign to each of these triangles a \emph{type of triangle}  $\gamma \in \P$, which is $\{1,2\}$, $\{1,3\}$, or $\{2,3\}$, according to the neighbours of this triangle in the path.
Let $h \in \{2, \ldots, k\}$. Then $W_{2h-1} \in \W_1$ and we assign to it the type 
$\{i,j\}$ where $i$ and $j$ correspond to the labels of the edges between $W_{2h-2}$ and $W_{2h-1}$ and 
$W_{2h-1}$ and $W_{2h}$. 
We have $W_{2h} \in  \W'_1$ for each $h \in \{1, \ldots, k\}$ and assign to this triangle, analogously, the type $\{i,j\}$, where $i$ and $j$ are the labels of the connecting edges.

Concerning the exits themselves, we define $W_1=e_1^\w(1)$ to be of type $\{1,j\}$,  if $W_1$ and $W_2$ are $j$-neighbours, i.e., the edge that connects $W_1$ and $W_2$ in $\G(\W \cup \W')$ is labelled by $j$. Observe that $j \not=1$ since $e_1^\w(1)$ is a border triangle having its side of type $1$ on $\varrho_1$. Analogously, $W_{2k+1}=e_2^\w(1)$ is of type $\{j,2\}$, depending on the edge between $W_{2k}$ and $W_{2k+1}$. 

We emphasise that the type $\gamma \in \P $ of a triangle is always defined with respect to the path (in the triangular labyrinth pattern or triangular labyrinth set of some level $n$) that we consider. (In other words, a triangle can change its type when analysing it from the perspective of different paths, corresponding to different pairs of exits in the pattern.)

After we have assigned to each triangle in the path one of the above 3 types, we proceed to the next step. In order to obtain the path of type $\{1,2\}$ in $\G(\W_2 \cup \W'_2)$, we replace all triangles in the considered path in the following way: if $W\in \W_1$ is of type $\{i,j\}$, then we replace it by the path of type $\{i,j\}$ in $\G(\W \cup \W')$, and if  $W\in \W'_1$ is of type $\{i,j\}$, the we replace it by the path of type $\{i,j\}$ in $\G(\Y \cup \Y')$. 

More generally,  we proceed as follows: for every $n \ge 1$  and every pair of exits in $\W_n \cup \W'_n$  we replace each triangle  $W \in \W_n$ of type $\{i,j \}$ in the path in $\G(\W_n \cup \W'_n)$ between these exits by the  path of type $\{i,j \}$ in $\G(\W \cup  \W')$ and  each triangle $W\in \W'_n$ of type $\{i,j \}$ by the path of type $\{i,j \}$ in $\G(\Y \cup  \Y')$. 

For $\pi, \gamma \in \P$ we introduce the following notation.\\
Let $m^\w_{\pi,\gamma}$ be the number of triangles from $\W$ of type $\gamma$ in the path of type $\pi$ in $\G(\W \cup \W')$,  
$\tilde m^\w_{\pi,\gamma}$ be the number of triangles from $\W'$ of type $\gamma$ in the path of type $\pi$ in $\G(\W \cup \W')$,  
$m^\y_{\pi,\gamma}$ be the number of triangles from $\Y$ of type $\gamma$ in the path of type $\pi$ in $\G(\Y \cup \Y')$,  
and
$\tilde m^\y_{\pi,\gamma}$ be the number of triangles from $\Y'$ of type $\gamma$ in the path of type $\pi$ in $\G(\Y \cup \Y')$.  

By the above described construction of the path between exits in the graph when passing from the path in $\G(\W_{n-1} \cup \W'_{n-1})$ to the path between the same pair of exits in the graph $\G(\W_n \cup \W'_n)$ of the white labyrinth set obtained at the next iteration, both for the white and the yellow paths, we obtain, for all $n\ge 1$:
 \[ \ell_{\pi}(n)=\sum_{\gamma \in \P} m^\w_{\pi,\gamma} \cdot \ell_{\gamma}(n-1)+\sum_{\gamma \in \P}  \tilde m^\w_{\pi,\gamma} \cdot \ell'_{\gamma}(n-1), \] 

\[ \ell'_{\pi}(n)=\sum_{\gamma \in \P} \tilde m^\y_{\pi,\gamma} \cdot \ell_{\gamma}(n-1)+\sum_{\gamma \in \P} m^\y_{\pi,\gamma}\cdot \ell'_{\gamma}(n-1). \] 

At this point, we introduce the non-negative matrices 
\[
M_\w:=\left( m^\w_{\pi,\gamma}\right)_{(\pi,\gamma)\in \P \times \P } \text{ and } 
\tilde M_\w:=\left( \tilde m^\w_{\pi,\gamma}\right)_{(\pi,\gamma)\in \P \times \P },
\] 
associated with $\G(\W \cup \W')$, and
\[
M_\y:=\left( m^\y_{\pi,\gamma}\right)_{(\pi,\gamma)\in \P \times \P } \text{ and } 
\tilde M_\y:=\left( \tilde m^\y_{\pi,\gamma}\right)_{(\pi,\gamma)\in \P \times \P },
\] 
associated with $\G(\Y \cup \Y')$.
With these notations, the above equations immediately yield the formula \eqref{eq:matrices_iteration}.

Now, let us show that $\tilde M_\w=M_\w - I_3$, $\tilde M_\y=M_\y - I_3$.
In order to prove  the first relation, for the ``white'' matrices, let us consider, w.l.o.g.,  the path from the exit $e_1^\w(1)$ to $e_2^\w(1)$ in $\G(\W_1\cup \W'_1)$.
The exit $e_1^\w(1)$ is a triangle of type $\{1,j_1\}$ with $j_1 \in \{2,3\}$.   Its neighbour in the path is of type $\{j_1,j_2\}$. The next triangle is of type $\{j_2,j_3\}$,  and so on. Finally, the exit $e_2^\w(1)$ is of type $\{j_r,2\}$, where $r$ is odd.

Therefore, for every $h \in \{1, \ldots, r\}$ we have $W_{2h}$ is of type $\{j_{2h-1}, j_{2h}\}$ and
for every $h \in \{1, \ldots, r-1\}$ we have $W_{2h+1}$ is of type $\{j_{2h}, j_{2h+1}\}$. Additionally, 
$W_{1}$ is of type $\{1,j_1\}$ and $W_{2r+1}$ is of type $\{j_{2r},2\}$.
We see that ``$3$'' occurs in the type of the even indexed triangles as often as it occurs in the odd indexed ones while the number of occurrences of  $1$ and $2$, respectively, in the type of the even indexed triangles is one less than the number of occurrences of  $1$ and $2$, respectively, in the type of the odd indexed ones.
From this observation we deduce the equations
\begin{align*}
m^\w_{\pi,\{1,2\}}+m^\w_{\pi, \{1,3\}} = & \tilde m^\w_{\pi, \{1,2\}}+\tilde m^\w_{\pi,\{1,3\}}+1, \\
m^\w_{\pi,\{1,2\}}+m^\w_{\pi,\{2,3\}} = & \tilde m^\w_{\pi, \{1,2\}}+\tilde m^\w_{\pi,\{2,3\}}+1, \\
m^\w_{\pi,\{1,3\}}+m^\w_{\pi, \{2,3\}} = & \tilde m^\w_{\pi,\{1,3\}}+\tilde m^\w_{\pi,\{2,3\}}.
\end{align*}

One easily verifies that the above equations yield
$$m^\w_{\pi,\{1,2\}}=\tilde m^\w_{\pi,\{1,2\}}+1, \;m^\w_{\pi,\{1,3\}}=\tilde m^\w_{\pi,\{1,3\}} \mbox{ and } m^\w_{\pi, \{2,3\}}=\tilde m^\w_{\pi,\{2,3\}}.$$ 

In the analogous way we obtain  $m^\w_{\pi,\gamma}=\tilde m^\w_{\pi,\gamma}+\delta_{\pi,\gamma},$ for $\pi \in \P \setminus \{ \{1,2\} \}, \gamma \in \P$, and $m^\y_{\pi,\gamma}=\tilde m^\y_{\pi,\gamma}+\delta_{\pi,\gamma},$ where $\delta_{\pi,\gamma}$  is the Kronecker symbol, for $\pi, \gamma \in \P$.
This completes the proof.
\end{proof}

With the above notations, we call the $6 \times 6$-matrix $$M=\left(\begin{array}{cc} M_\w & \tilde M_\w\\ \tilde M_\y& M_\y \end{array}\right)$$ the \emph{global path matrix}, $M_\w$ the \emph{white path matrix} and $M_\y$ the \emph{yellow path matrix} of the triangular labyrinth patterns system, respectively.
Let $M(n)$ denote the global matrix of the system $(\W_n\cup \W'_n, \Y_n \cup \Y'_n)$, for $n \ge 1$.
\\
The above proposition immediately yields the following result. 
\begin{corollary}\label{cor:iteration_matrix} With the above notations we have
\begin{equation} \label{eq:iteration_matrix}
\left(
 \begin{array}{l}
\ell_{\{1,2\}}(n) \\  
\ell_{\{1,3\}}(n) \\  
\ell_{\{2,3\}}(n) \\  
\ell'_{\{1,2\}}(n) \\  
\ell'_{\{1,3\}}(n) \\  
\ell'_{\{2,3\}}(n) \\  
\end{array}
\right)
=M^n \cdot 
\left(
\begin{array}{l}
1 \\  
1 \\  
1 \\  
1 \\  
1 \\  
1 \\  
\end{array}\right),
\end{equation} 
for all $n \ge 1,$
where $M=\left(\begin{array}{cc} M_\w & M_\w-I_3\\ M_\y-I_3& M_\y \end{array}\right)$  is the global path matrix.\\
In other words, $M(n)=M^n$, for all $n\ge 1$.

\end{corollary}

\begin{example} 
Again, we consider the triangular labyrinth patterns system from Example~\ref{Beispiel1}. In this case we have
\[M= \left(\begin{array}{cccccc}
3 & 1 & 0 & 2 & 1 & 0 \\ 1 & 1 & 1 & 1 & 0 & 1 \\ 1 & 1 & 2 & 1 & 1 & 1 \\
0 & 0 & 1 & 1 & 0 & 1 \\ 0 & 2 & 0 & 0 & 3 & 0 \\ 0 & 2 & 1 & 0 & 2 & 2
\end{array}\right),\]
where 
\[M_\w= \left(\begin{array}{ccc}
3 & 1 & 0 \\ 1 & 1 & 1 \\ 1 & 1 & 2 
\end{array}\right), 
\mbox{ and  }
M_\y=
\left(\begin{array}{ccc}
1 & 0 & 1 \\ 0 & 3 & 0 \\ 0 & 2 & 2 
\end{array}\right).
\]
\end{example}

In Section~\ref{sec:blocked} we deal again with the matrix $M$ . There we  see that 
$M$ is a primitive matrix provided that the 
triangle labyrinth patterns system satisfies certain conditions (named in Section \ref{sec:blocked} as the property of being``blocked'') which ensure, roughly speaking, that the labyrinth sets obtained in the resulting iterations (and thus also the resulting labyrinth fractals) are ``sufficiently winding".

\section{Arcs in triangular labyrinth fractals}\label{sec:arcs_TLF}
 In this section we deal with aspects of the structure, topology, and fractal geometric properties of the arcs in triangular labyrinth fractals.
By a non-trivial arc in a labyrinth fractal we always mean an arc that connects two distinct points of the fractal.

The following lemma, whose proof is based on  a theorem from the book of Kuratowski \cite[Theorem 3, par 47, V, p.181]{Kuratowski}, provides a construction method for the  unique arc between any two distinct points in the dendrite $\linfw$. The proof is based on the same idea as for the (square) self-similar labyrinth fractals \cite{laby_4x4}.

\begin{theorem}[\cite{Kuratowski}, Theorem 3, par 47, V, p.181]\label{theo:arc_Kuratowski}
In a compact metric space let $C_n$, $n=1,2,\dots$ be connected sets containing the points $\mathbf{A}$ and $\mathbf{B}$, such that to every point $\mathbf{X} \in \bigcap_{n\geq 1} C_n$ there corresponds a decomposition $C_n = A_n \cup B_n$ which satisfies the following conditions
\begin{enumerate}
\item[(i)] $\mathbf{A} \in A_n$, $\mathbf{B}\in B_n$, $\mathbf{X} \in A_n \cap B_n,$
\item[(ii)] The diameter of $A_n\cap B_n$ tends to zero, as $n\to \infty$,
\item[(iii)] $\overline{A}_{n+1} \subseteq A_n$, $\overline{B}_{n+1} \subseteq B_n$,
\end{enumerate}
then the intersection $\cap_n C_n$ is an arc between $a$ and $b$.
\end{theorem}

\begin{lemma}[Arc Construction]\label{lemma:arc_construction} 
Let $\mathbf{A},\mathbf{B} \in \linfw$, such that $\mathbf{A} \ne \mathbf{B}$. For all $n \ge 1$, there exist $W_n(\mathbf{A}), W_n(\mathbf{B}) \in \W_n \cup \W'_n $ such that 
\begin{enumerate}
\item[(a)] $W_1(\mathbf{A})\supseteq W_2(\mathbf{A})\supseteq\dots$,
\item[(b)] $W_1(\mathbf{B})\supseteq W_2(\mathbf{B})\supseteq\dots$,
\item[(c)] $\{\mathbf{A}\}=\bigcap_{n=1}^{\infty}W_n(\mathbf{A})$,
\item[(d)] $\{\mathbf{B}\}=\bigcap_{n=1}^{\infty}W_n(\mathbf{B})$,
\item[(e)] The set $\displaystyle \bigcap_{n=1}^{\infty} \left( \bigcup _{W\in p_n(W_n(\mathbf{A}),W_n(\mathbf{B}))}W \right)$ is an arc that connects $\mathbf{A}$ and $\mathbf{B}$.
\end{enumerate}
\end{lemma}
\begin{proof}
We denote by $\W(\mathbf{A})$ the set of all white triangles in $\bigcup_{n\ge 1} (\W_n \cup \W'_n)$ that contain $\mathbf{A}$. Let $W_1(\mathbf{A})$ be a white triangle in $\W_1 \cup \W'_1$, with the property that infinitely many white triangles of ${\mathcal W}(\mathbf{A})$ are subsets of $W_1(\mathbf{A})$. For $n \ge 2$, let $W_n(\mathbf{A})$ be a white triangle in $\W_n \cup \W'_n$, such that $W_{n-1}(\mathbf{A} )\supseteq W_n(\mathbf{A})$ and  infinitely many white triangles of $\W(\mathbf{A})$ are contained in $W_n(\mathbf{A})$. We define $W_n(\mathbf{B})$, for all $n\ge 1$, analogously. Herewith the  assertions (a) to (d) in the above lemma are proven.

In order to prove (e) we apply Theorem \ref{theo:arc_Kuratowski}.
For a triangle $W \in \W_n \cup \W'_n$ we define $\tilde{W}:=W \cap L^{\w}_{n+1}$.
By the tree property and the exits property it follows that the set $C_n:=\bigcup_{W \in p_n(W_n(\mathbf{A}), W_n(\mathbf{B}))}\tilde{W}$ is connected for all $n\ge 1$. Let now $C_{\infty}=\bigcap_{n=1}^{\infty} C_n$ and observe that $C_\infty \not= \emptyset$. As a consequence of the fact that $C_{\infty}=\bigcap_{n=1} \big( \bigcup_{W \in p_n(W_n(\mathbf{A}), W_n(\mathbf{B}))}{W} \big)$, it is now sufficient to prove that $C_{\infty}$ is an arc between $\mathbf{A}$ and $\mathbf{B}$. The idea of taking the set $\tilde W$ is to avoid the situation when triangles of such a path could intersect at a common vertex they share, which would create problems in proving that condition (ii) in the mentioned theorem is satisfied. Therefore, 
for $\mathbf{X} \in C_{\infty}$ let $\W(\mathbf{X})$ denote the set of all triangles in $\bigcup_{n=1}^{\infty}p_n(W_n(\mathbf{A}),W_n(\mathbf{B}))$ that contain the point $\mathbf{X}$. Let $W_1(\mathbf{X})$ be a triangle in $p_1(W_1(\mathbf{A}), W_1(\mathbf{B}))$ that contains infinitely many triangles of  $\W(\mathbf{X})$ as subsets. Now, for $n\ge 2,$ let $W_n(\mathbf{X})$  be a triangle occurring in the path $p_n(W_n(\mathbf{A}), W_n(\mathbf{B}))$, such that $W_{n-1}\supseteq W_n$, and $W_n(\mathbf{X})$  contains infinitely many triangles of $\W(\mathbf{X})$ as subsets. Let $A_n=\bigcup_{W\in p_n(W_n(\mathbf{A}),W_n(\mathbf{X})}\tilde W$ and $B_n=\bigcup_{W\in p_n(W_n(\mathbf{X}),W_n(\mathbf{B}))}\tilde W$. The set $A_n\cap B_n$ coincides with $\tilde W_n(\mathbf{X})$, based on the fact that if $W_1, W_2 \in p_n(W_n(\mathbf{A}), W_n(\mathbf{B}))$ are distinct triangles then  the corner property yields that $\tilde W_1 \cap \tilde W_2$ is non-empty if and only if $W_1$ and $W_2$ are neighbours in $\G(\W \cup \W')$. We have $\overline{A}_{n+1}\subseteq A_n$, since $\tilde{W}_{n+1}(\mathbf{A})\subset \tilde{W}_{n}(\mathbf{A})$ and $\tilde{W}_{n+1}(\mathbf{X})\subset \tilde{W}_{n}(\mathbf{X})$, and the path between triangles in $\G(\W_k \cup \W'_k)$ is obtained as described in Proposition \ref{prop:path_matrix}, for all $k\ge 1$.
Analogously we obtain $\overline{B}_{n+1}\subseteq B_n$, and herewith all three conditions in Theorem \ref{theo:arc_Kuratowski} are fulfilled, which implies that $C_{\infty}$ is an arc that connects $\mathbf{A}$ and $\mathbf{B}$ in $\linfw$.
\end{proof}

In the following, for any two points $\mathbf{X} \ne \mathbf{Y}$ in a dendrite we denote by $a(\mathbf{X},\mathbf{Y})$ the arc  that connects them in the dendrite. 

With the notation of Section \ref{sec:labyrinth_patterns_systems} we introduce the exits of the white  and yellow fractal as follows.

\begin{definition}\label{def:exits_fractal}
For every $i \in \{1,2,3\}$, \emph{the exit of type} $i$ in $\linfw$ is the point $\mathbf{E}_i^\w$, defined by $\displaystyle \{\mathbf{E}_i^\w\}=\bigcap_{n=1}^{\infty}e_i^\w(n)$, where  $e_i^\w(n)$ is the exit of  type $i$ in $\W_n$, and the point $\mathbf{E}_i^\y $, with $\displaystyle\{\mathbf{E}_i^\y\}=\bigcap_{n=1}^{\infty}e_i^\y(n)$, is  \emph{the exit of type} $i$ in $\linfy$, where $e_i^\y(n)$ is the exit of  type $i$ in $\Y_n$.
\end{definition}

Since we frequently use the arcs that connect the exists of the labyrinth fractals we introduce a special notation for them. In particular, for any pair $\{i,j\} \in \P $ let 
\[a_{\{i,j\}}^\w:=a(\mathbf{E}_i^\w, \mathbf{E}_j^\w) \text{ and }
a_{\{i,j\}}^\y:=a(\mathbf{E}_i^\y, \mathbf{E}_j^\y).\]

The proposition below gives information about a topological (shape) feature of the
dendrites $\linfw$ and $\linfy$.

\begin{proposition}\label{prop:exists_C}
Let $\linfw$ be a white triangular labyrinth fractal, and $\mathbf{E}_1^\w,\mathbf{E}_2^\w,\mathbf{E}_3^\w$ its exits. There exists a point $\mathbf{C}^\w \in \linfw$ such that 
\[\{\mathbf{C}^\w  \}=a_{\{1,2\}}^\w \cap a_{\{2,3\}}^\w \cap a_{\{1,3\}}^\w.\]
Similarly, there exists a point $\mathbf{C}^\y \in \linfy$ with the analogous property.
\end{proposition}
\begin{proof} We give a proof for the white fractal.
Let us first consider the special case when there exists an $i \in \{1,2,3\}$ such that $\mathbf{E}_i^\w \in a_{\{j,k\}}^\w$, where $\{i,j,k \}=\{1,2,3 \}$. In this case we have $\{ \mathbf{E}_i^\w  \}=a_{\{1,2\}}^\w \cap a_{\{2,3\}}^\w \cap a_{\{1,3\}}^\w$, i.e., $\mathbf{C}^\w=\mathbf{E}_i^\w,$
as shown in the left example in Figure \ref{fig:lemma_exists_C}.
Assume now that there exists no $i \in \{1,2,3\}$ as in the special case above. Since $\linfw$ is a dendrite (and thus, cycle-free), it immediately follows that  distinct
non-disjoint arcs in the dendrite share either a point or a common subarc. Let thus $\mathbf{C}^\w\in \linfw$ be such that $a_{\{2,3\}}^\w \cap a_{\{1,3\}}^\w=a(\mathbf{E}_3^\w,\mathbf{C}^\w)$. Then $a(\mathbf{E}_1^\w,\mathbf{C}^\w)\cup a(\mathbf{C}^\w,\mathbf{E}_2^\w)$ is an arc that connects $\mathbf{E}_1^\w$ and $\mathbf{E}_2^\w$ in $\linfw$, see also Figure \ref{fig:lemma_exists_C}. From the uniqueness of the arc that connects given distinct points in a dendrite we get  $a(\mathbf{E}_1^\w,\mathbf{C}^\w)\cup a(\mathbf{C^\w},\mathbf{E}_2^\w)=a_{\{1,2\}}^\w$, which yields the desired result.
\end{proof}

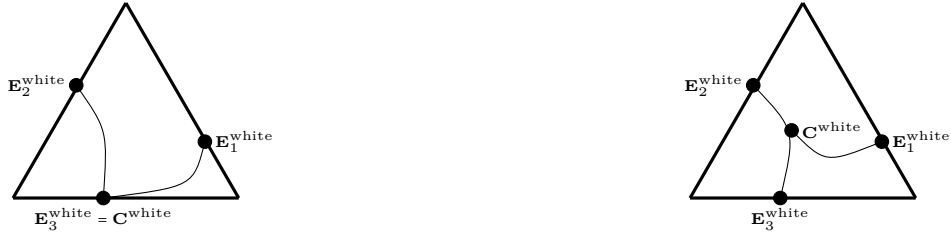
\begin{figure}[h]
\begin{center}
 \begin{tikzpicture}
[scale=3.0]
% das linke dreieck
\draw[line width=1.2pt, draw= black] (-2, 0) -- (-1,0);
\draw[line width=1.2pt, draw= black] (-2, 0) -- (-1.5,0.865);
\draw[line width=1.2pt, draw= black] (-1,0) -- (-1.5,0.865);

\draw[fill] (-1.6,0) circle [radius=0.03];
\coordinate[label=below:\tiny{$\mathbf{E}_3^\w=\mathbf{C}^\w$}] (E3) at (-1.6,0);

\draw[fill] (-1.72,0.5) circle [radius=0.03];
\coordinate[label=left:\tiny{$\mathbf{E}_2^\w$}] (E2) at (-1.72,0.5);

\path[draw](-1.6,0) .. controls (-1.58,0.3) .. (-1.72,0.5);

\draw[fill] (-1.15,0.25) circle [radius=0.03];
\coordinate[label=right:\tiny{$\mathbf{E}_1^\w$}] (E1) at (-1.15,0.25);

\path[draw](-1.15,0.25) .. controls (-1.2,0.05) .. (-1.6,0);

% das rechte dreieck
\draw[line width=1.2pt, draw= black] (1, 0) -- (2,0);
\draw[line width=1.2pt, draw= black] (2, 0) -- (1.5,0.865);
\draw[line width=1.2pt, draw= black] (1,0) -- (1.5,0.865);

\draw[fill] (1.4,0) circle [radius=0.03];
\coordinate[label=below:\tiny{$\mathbf{E}_3^\w$}] (E3) at (1.4,0);

\draw[fill] (1.28,0.5) circle [radius=0.03];
\coordinate[label=left:\tiny{$\mathbf{E}_2^\w$}] (E2) at (1.28,0.5);

\path[draw](1.4,0) .. controls (1.47,0.30) .. (1.28,0.5);

\draw[fill] (1.85,0.25) circle [radius=0.03];
\coordinate[label=right:\tiny{$\mathbf{E}_1^\w$}] (E1) at (1.85,0.25);

\draw[fill] (1.45,0.30) circle [radius=0.03];
\coordinate[label=right:\tiny{$\mathbf{C}^\w$}] (C) at (1.45,0.30);

\path[draw](1.85,0.25) .. controls (1.6,0.15) .. (1.45,0.30);

\end{tikzpicture}
\end{center}
\caption{The three arcs that connect the exits of a (white) triangular labyrinth fractal intersect in one point $\mathbf{C}$. The left triangle shows the case  $\mathbf{E}_3^\w=\mathbf{C}^\w$.}
\label{fig:lemma_exists_C}
\end{figure}

With the notations in the above  Lemma, we define the core of the fractals, als follows.

\begin{definition}\label{def:core_fractals}
The core of the white triangular labyrinth fractal is the set
\[ \mathrm{core}(\linfw)=\bigcup_{\pi \in \P} a_{\pi}^\w=\bigcup_{i=1}^3 a(\mathbf{E}_i^\w,\mathbf{C}^\w),  \]
and the core of the yellow triangular labyrinth fractal is 
\[ \mathrm{core}(\linfy)=\bigcup_{\pi \in \P} a_{\pi}^\y=\bigcup_{i=1}^3 a(\mathbf{E}_i^\y,\mathbf{C}^\y).\]
\end{definition}

\begin{lemma}\label{lemma:exits_are_fixpoints}
With the above notations we have for every $i \in \{1,2,3\}$:
\begin{enumerate} 
\item[(a)] For all $n \geq 1$ the triangle $e_i^\w(n)$ is the only element of $\W_n \cup \W'_n$ that contains the exit $\mathbf{E}_i^\w  \in \linfw$.
\item[(b)] If $n\ge 1$ then $\mathbf{E}^\w_i$ is a fixed point of the projection map $P_{e_i^\w(n)}$. 
\item[(c)] Let $\mathbf{Z} \in T_1$. Then for all $n\ge 1$ the vectors $P_{e_i^\w(n)}(\mathbf{Z})-\mathbf{E}^\w_i$ and $\mathbf{Z}-\mathbf{E}^\w_i$ are linearly dependent.
\end{enumerate}
\end{lemma}

\begin{proof}
(a)\quad We give an indirect proof. W.l.o.g., let $i =3$, i.e., we consider the exit $\mathbf{E}_3^\w $.  Assume that for some $n\ge 1$, $\mathbf{E}_3^\w$ lies in two distinct triangles $W_a, W_b \in \W_n$. Then one of these two triangles is $e_3^\w(n)$, the exit of type $3$ in $\W_n\cup \W'_n$, otherwise we would have a contradiction with the definition of the exit of type $i$ in $\linfw$. Assume, w.l.o.g., $W_a= e_3^\w(n)$. 
If $W_a \in \T_{m^n}$, $W_a=T_{m^n}(k_1,k_2,k_3)$ satisfies $k_1=0$ or $k_2=0$, i.e., $W_a$ lies in one of the other two border strips of level  $n$ of the triangle, then $\mathbf{E}_3^\w\in \{\mathbf{P}_1,\mathbf{P}_2\}$, which is in contradiction with the fact that $\mathbf{E}_3^\w$ lies in two distinct triangles $W_a, W_b \in \W_n$. Thus, $W_a$ can not be a corner triangle in $\W_n$. This implies that $P_{W_a}(\linfw)$ does not contain a corner vertex of the triangle $W_a$, which is in contradiction with the fact $\mathbf{E}_3^\w \in W_a \cap W_b$. This completes the proof.

(b)\quad 
With the above notations, the definition of the exits in $\linfw$  and the GIFS approach from Section \ref{sec:TLF} we have $\{ \mathbf{E}_i^\w \}=\lim_{n\to \infty}P_{\Delta_1, \dots,\Delta_n}(T_1)$, where $\Delta_k=W_1(\mathbf{E}_i^\w)=e_i^\w(1)$, for all $k=1,2\ldots$. This implies $P_{e_i^\w(n)}(\mathbf{E}_i^\w)=\mathbf{E}_i^\w$.

(c)\quad The linear dependence of the vectors $P_{e_i^\w(n)}(\mathbf{Z})-\mathbf{E}_i^\w$ and $\mathbf{Z}-\mathbf{E}_i^\w$ immediately follows from $P_{e_i^\w(n)}(\mathbf{E}_i^\w)=\mathbf{E}_i^\w$ and the fact that the vectors $\mathbf{Z}-\mathbf{E}_i^\w$ and $P_{e_i^\w(n)}(\mathbf{Z}-\mathbf{E}_i^\w)=P_{e_i^\w(n)}(\mathbf{Z})-P_{e_i^\w(n)}(\mathbf{E}_i^\w)$ are linearly dependent, due to the fact that the map $P_{e_i^\w}$ is a similarity (of factor $m^{-1}$) and fixed point $\mathbf{E}_i^\w$, more precisely it is a homothetic map with fixed point $\mathbf{E}_i^\w$ and similarity factor $m^{-1}$. 
\end{proof}
An analogous lemma holds for the yellow fractal $\linfy$.

\begin{remark}
In Section~\ref{sec:TLF} we have mentioned that
for each $T \in \T_m \cup \T'_m$ the map $P_T$ is a contraction. Therefore, there exists a uniquely determined fixed point $\mathbf{X}$ that satisfies $P_T(\mathbf{X})=\mathbf{X}$.
With this observation the affirmation $(b)$ in the above lemma gives an alternative (and equivalent) definition for the exits of the white  and yellow fractal.
\end{remark}

The following result describes a characteristic property of the graph of a triangular labyrinth pattern, and provides the discrete analogon of the result in Proposition \ref{prop:exists_C}.

\begin{proposition}\label{prop:exists_W_star} 
Let $(\W \cup \W', \Y \cup \Y')$ be an $m$-triangular labyrinth patterns system. Then there exists a unique triangle $W^*\in \W \cup \W'$ that is contained in the path $p\,(e_i^\w,e_j^\w)$ that connects $e_i^\w$ and $e_j^\w$ in $\G(\W \cup \W')$ for all $\{i,j\} \in \P$.
The analogous assertion holds for $Y^*\in \Y \cup \Y'$. 
\end{proposition}
\begin{proof}
Let $W^{(1)}_1, \ldots, W^{(1)}_{r_1}$ and $W^{(2)}_1, \ldots, W^{(2)}_{r_2}$ denote the paths $p\,(e_1^\w,e_2^\w)$ and $p\,(e_1^\w,e_3^\w)$, respectively, with $W^{(1)}_1 =W^{(2)}_1 = e_1^\w$, $W^{(1)}_{r_1} = e_2^\w$, and $W^{(2)}_{r_2} = e_3^\w$. Let $s \leq \min\{{r_1}, {r_2}\}$ be the largest index such that $W^{(1)}_h =W^{(2)}_h$ holds for all $h \in \{1, \ldots, s\}$. 
 Define $\W(1):=\{W^{(1)}_h: 1\leq h < s\}$, $\W(2):=\{W^{(1)}_h: s< h \leq r_1\}$, $\W(3):=\{W^{(2)}_h: s< h \leq r_2\}$ and $W^*:=W^{(1)}_s = W^{(2)}_s$. We claim that $\W(1)$, $\W(2)$, $\W(3)$ and $\{W^*\}$ are pairwise disjoint sets. 
Indeed, except for the disjointness of $\W(2)$ and $\W(3)$, this follows immediately by the definition of a path. 
We show this remaining case indirectly and assume that $\W(2) \cap \W(3) \not=\emptyset$.  
We can  exclude that $\W(2)$ or $\W(3)$ is the empty set (that is $s < \min\{r_1, r_2\}$). Thus, there are indices $h_1$ and $h_2$ with $s < h_1 \leq r_1$ and $s < h_2  \leq r_2$ such that 
$W^{(1)}_{h_1} = W^{(2)}_{h_2}$. W.l.o.g. we may assume that 
$\{W^{(1)}_h: s< h <h_1\}$ and $\{W^{(2)}_h: s< h <h_2\}$ are disjoint sets. Observe that either $h_1>s+1$ or $h_2>s+1$, otherwise the maximality of $s$ would be violated. This yields that, $W^{(1)}_{h_1-1}, \ldots, W^{(1)}_{s}, W^{(2)}_{s+1}, \ldots, W^{(2)}_{h_2}=W^{(1)}_{h_1}$ is a cycle in $\G(\W \cup \W')$ which contradicts the tree property.

The pairwise disjointness of $\W(2)$, $\W(3)$ and $\{W^*\}$ shows that
$W^{(1)}_{r_1}, \ldots, W^{(1)}_{s+1}, W^*, W^{(2)}_{s+1}, \ldots, W^{(2)}_{r_2}$ is a path in 
$\G(\W \cup \W')$ from $W^{(1)}_{r_1}= e_2^\w$ to $W^{(2)}_{r_2}= e_3^\w$ and by the tree property it must coincide with $p\,(e_2^\w,e_3^\w)$. Therefore, $W^*$ is contained in $p\,(e_i^\w,e_j^\w)$ for all $\{i,j\} \in \P$ and since $\W(1)$, $\W(2)$ and $\W(3)$ are pairwise disjoint. By a pigeonhole principle argument, we also see that $W^*$ is the unique triangle with this property.

Observe that the triangle $W^*$ can also coincide with one of the exits of $\W$. Indeed, 
if $s=1$ then $W^* = e_1^\w$. Note that in this case we have $\W(1)=\emptyset$. If $s=\max\{r_1,r_2\}=r_1$ then 
$W^* = e_2^\w$ (and $\W(2)=\emptyset$) and if $s=\max\{r_1,r_2\}=r_2$ then we have $W^* = e_3^\w$ (and $\W(3)=\emptyset$).
\end{proof}
With the  notations from the proof,  we call the set of triangles 
\[\W^0:=\W(1) \cup \W(2)\cup \W(3) \cup \{W^*\}\]
the core of the (white) labyrinth pattern. The core $\Y^0$ of the yellow pattern in the system is defined in the obvious analogous way. The core of the white and yellow labyrinth set of level $n\ge 1$ are obtained in the analogous way.

\begin{remark} 
Since the pair $(\W_n\cup \W'_n, \Y_n \cup \Y'_n)$ of labyrinth sets obtained at the $n$-th iteration, for some $n\ge 1$, can be viewed as an $m^n$-triangular labyrinth patterns system, the above proposition can be applied to it. We denote by $W_n^*$ and $Y_n^*$ the corresponding vertices in the graphs $\G(\W_n \cup \W'_n)$ and $\G(\Y_n \cup \Y'_n)$, respectively. In the patterns shown in Figure \ref{fig:W_star_and_Y_star} the triangles $W^*$ and $Y^*$ are marked by a grey triangle inside each of them. On the other hand, these patterns coincide with the triangular labyrinth sets in the triangular labyrinth system of level $2$ shown in Figure~\ref{fig:Bsp1W23} and Figure~\ref{fig:Bsp1Y23}, respectively, and therefore the marked triangles are the triangles $W^*_2$ and $Y^*_2$ obtained from the triangular patterns system shown in Figure \ref{fig:Bsp1WY1}.
\end{remark}
\begin{figure}[h]
\begin{center}
\includegraphics[width=0.30\textwidth]{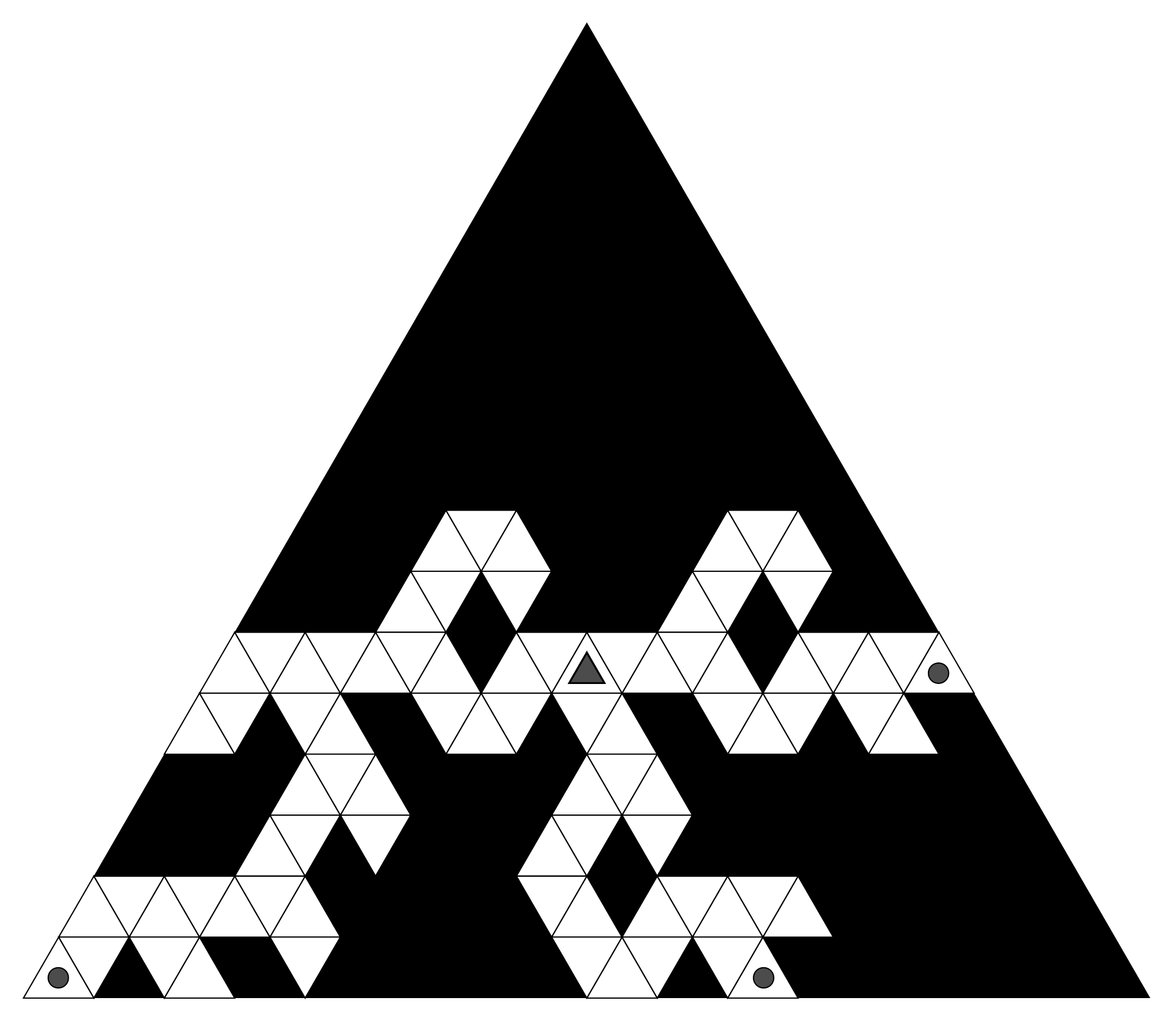}
\hspace{2.00cm}
\includegraphics[width=0.30\textwidth]{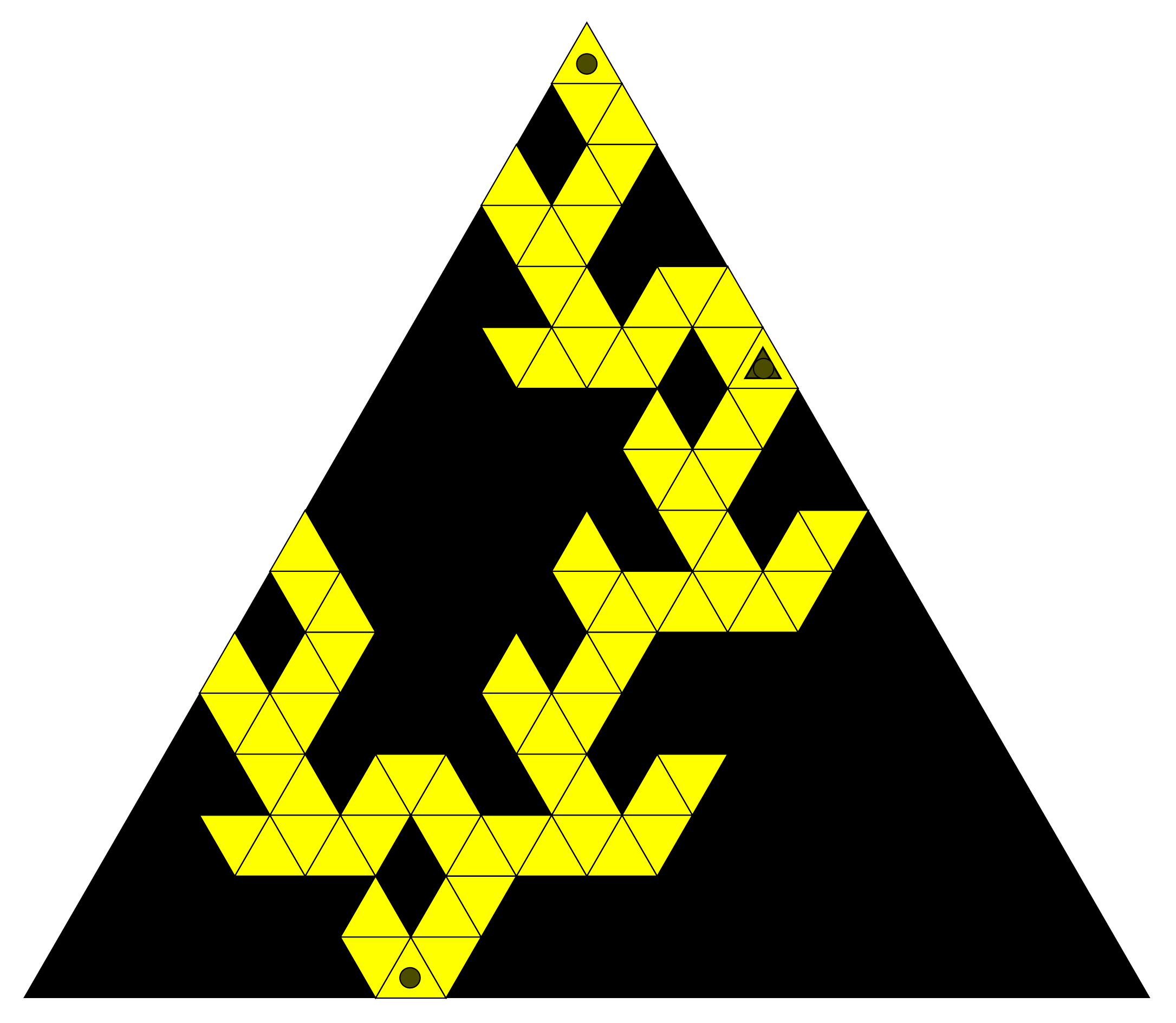}
\caption{The exits and the special triangles $W^*$ and $Y^*$ in a triangular labyrinth patterns system}
\label{fig:W_star_and_Y_star}
\end{center}
\end{figure}
\begin{corollary}\label{cor:C_intersection_of W_star}
$ \{C^\w\}=\bigcap_{n\ge 1} W_n^*$ and $ \{C^\y\}=\bigcap_{n\ge 1} Y_n^*$. In other words, the points $C^\w \in \linfw$ and $C^\y \in \linfy$ in Proposition \ref{prop:exists_C} are uniquely determined.
\end{corollary}
\begin{proof} This easily follows from  Proposition \ref{prop:exists_C} and  Proposition \ref{prop:exists_W_star} and the construction of arcs in the fractal given in Lemma \ref{lemma:arc_construction}.

\end{proof}

\begin{lemma} \label{lemma:exits_of_neighbours_coincide}
Let $\linfw$, $\linfy$ be the white and yellow triangular labyrinth fractal, respectively, generated by an $m$-triangular labyrinth patterns system. For arbitrarily fixed $n\ge 1$ and $i\in \{1,2,3 \}$, let $W\in \W_n$ and $W'\in \W'_n$, such that $W$ and $W'$ are $i$-neighbours. Then $P_W(\mathbf{E}_i^\w)=P_{W'}(\mathbf{E}_i^\y)$.
\end{lemma}

\begin{proof}
Without loss of generality, let us assume $i=1$, since the proof for $i=2$ or $i=3$ works in the same way. Let us first remark that since $\mathbf{E}_1^\w$ and $\mathbf{E}_1^\y$ are exits of the fractals, they can be represented by homogeneous coordinates 

$(0, \alpha, 1-\alpha), (0, \alpha', 1-\alpha') \in H$ such that
\[\mathbf{E}_1^\w = \alpha \mathbf{P}_2 + (1-\alpha) \mathbf{P}_3 \text{ and }
\mathbf{E}_1^\y = \alpha' \mathbf{P}_2 + (1-\alpha') \mathbf{P}_3.\]
From the definition of the exits in the fractal (obtained as limits of nested sequences of the triangles that are exits in labyrinth sets of level $n\ge 1$), and the fact that the exits of type $1$ at the $n$-th iteration are triangles of the form  $T_{m^n}(0,k_1(n),m^n-1-k_1(n))\in \W_n$ and $T_{m^n}(0,m^n-1-k_1(n),k_1(n))\in \Y_n$, respectively, with $0\le k_1(n)<m$, it follows that $\alpha'=1-\alpha$.

Now, for $n\ge 1$, let $W\in \W_n$ and $W' \in \W'_n$ be $1$-neighbours. By using the notations from \eqref{formel_Tm}, there exist non-negative integers $k_1,k_2,k_3$ with $k_1+k_2+k_3=m-1$ such that $W=T_m(k_1,k_2,k_3)=\Delta(\mathbf A_1,\mathbf A_2, \mathbf A_3)$ and $W'=T'_m(k_1-1,k_2,k_3)=\Delta(\mathbf A'_1, \mathbf A'_2,\mathbf A'_3   )$ and we obtain:
\begin{align*}
P_W(\mathbf E_1^\w) &=0\cdot \mathbf A_1+ \alpha \cdot \mathbf A_2 +(1-\alpha)\cdot \mathbf  A_3)\\
&= \frac{1}{m}(\alpha k_1 \mathbf P_1+\alpha (k_2+1) \mathbf P_2 +\alpha k_3 \mathbf P_3)\\
&+ \frac{1}{m}((1-\alpha k_1) \mathbf P_1+(1-\alpha) k_2 \mathbf P_2 +(1-\alpha)( k_3+1) \mathbf P_3)\\
&=\frac{1}{m}(k_1 \mathbf P_1+ (\alpha+k_2) \mathbf P_2 +(k_3+1-\alpha)\mathbf P_3)
\end{align*}
and, analogously, 
\begin{align*}
P_W(\mathbf E_1^\y) &=0\cdot \mathbf A'_1+ (1-\alpha) \cdot \mathbf A'_2 +\alpha\cdot \mathbf  A'_3\\
&=\frac{1}{m}(k_1 \mathbf P_1+ (\alpha+k_2) \mathbf P_2 +(k_3+1-\alpha)\mathbf P_3)
\end{align*}

\end{proof}

For  $i\in \{1,2,3\}$, let $\mathbf{E}_i^\w$ and $\mathbf{E}_i^\y$ be the exits of type $i$ in $\linfw$ and $\linfy$, respectively, and $n\ge 1$. For $W \in \W_n$, we call $P_W(\mathbf{E}_i^\w)$ the \emph{exit of type $i$ of} $W$, and for $W' \in \W'_n$, we call $P_{W'}(\mathbf{E}_i^\y)$ the \emph{exit of type $i$ of} $W'$. Analogously, for $Y \in \Y_n$, we call $P_Y(\mathbf{E}_i^\y)$ the \emph{exit of type $i$ of} $Y$, and for $Y' \in \Y'_n$, we call $P_{Y'}(\mathbf{E}_i^\w)$ the \emph{exit of type $i$ of} $Y'$.

\begin{remark}  From the above definition and the construction of triangular labyrinth fractals one can see that
the exit of type $i$ of a triangle $W \in \W_n \cup \W'_n$ is  a point $\mathbf{E}_i(W)$  that lies on the intersection of $\linfw$ with the side of type $i$ of $W$ such that  if $U$ is a neighbour of type $i$ of $W$ there exist points $\mathbf{P} \in {\rm int}(W)\cap \linfw$ and $\mathbf{Q} \in {\rm int}(U) \cap \linfw$ such that  $\mathbf{E}_i(W)$ lies on the arc that connects $\mathbf{P}$ and $\mathbf{Q}$ in $\linfw$. 
The same holds for the analogous definitions for triangles from $\Y_n \cup \Y'_n$ in $\linfy$. In other words, the exits of a triangle are the points where the fractal crosses its sides to ``pass'' to (the interior of) one of its neighbour triangles.

Moreover,  the structure of labyrinth fractals and the definition of the exits imply that if ${\mathbf X}\in \linfw$ is an exit of a triangle, say ${\mathbf X}={\mathbf E}_i(W)$ for some $W \in \W_{n_0}\cup \W'_{n_0}$, then for all $n\ge n_0$ the point ${\mathbf X}$  is the exit of type $i$ of any triangle $W' \in \W_{n}\cup \W'_{n}$, $W' \subseteq W$, that contains ${\mathbf X}$ and whose side of type $i$ lies on the side of type $i$ of $W$, i.e., ${\mathbf X}={\mathbf E}_i(W')$. 
\end{remark}

Thus, we have defined exits for three different kinds of objects: for triangular labyrinth sets of  order $n\ge 1$, for the triangular labyrinth fractals $\linfw$ and $\linfy$, and for triangles in $\W_n  \cup \W'_n$ or $\Y_n \cup \Y'_n$.

\begin{lemma}\label{lemma:arc_exits_triangle}
Let $\{i,j\} \in \P$, $p$ be the path in $\G(\W_n\cup \W'_n)$ that connects $e_i^\w(n)$ and $e_j^\w(n)$, and $W \in \W_n \cup \W'_n$ a triangle of type $\{i',j'\}$ with respect to $p$. Then $W \cap a_{\{i,j\}}^\w$ is an arc that connects the exits of type $i'$ and $j'$ of $W$. An analogous assertion holds for the yellow fractal $\linfy$.
\end{lemma}

\begin{proof} We give a proof for the white fractal and $\{i',j'\}=\{1,2\}$. By using the path construction described in the proof of Proposition \ref{prop:path_matrix} and the arc construction method given in Lemma \ref{lemma:arc_construction} we construct the arc $a_{\{1,2\}}^\w$ from a sequence of paths $\{p_k\}_{k\ge 1}$.
Now, we separately consider  the cases $W \in \W_n$, and $W \in \W'_n$, respectively. 

First, let $W \in \W_n$ be a triangle of type $\{1,2\}$ in the path $p$. Thus, for $k=n+h$,  $h\ge 1$ the path $p_k$ ``restricted to'' $W$, i.e., restricted to the subgraph of $\G(\W_{n+h} \cup \W'_{n+h})$ induced by the triangles from $\W_{n+h}$ and $\W'_{n+h}$ that are contained in $W$ is a path from $P_{W}(e_1^\w(h))$ to $P_{W}(e_2^\w(h))$ in $\G(\W_h \cup \W'_h)$.
From the continuity of the projection $P_{W}$ and the definition of the exits of $W$ it follows that $W \cap a$ is the arc in $\linfw$ that connects the exits of type $1$ and $2$ of $W$.

In the case when $W \in \W'_n$, for all $k=n+h>n$ the path $p_h$ restricted to $W$ is a path from $P_{W}(\mathbf{E}_1^\y)$ to $P_{W}(\mathbf{E}_2^\y)$ in $\G(\Y_h \cup \Y'_h)$. Again, the continuity of $P_{W}$ and the definition of the exits of $W$ yield that $W \cap a$ is the arc in $\linfw$ that connects the exits of type $1$ and $2$ of $W$.
\end{proof}

\begin{lemma} \label{lemma:projections_of_arcs_between_exits}
\begin{enumerate} With the above notations, the following assertions hold for all $\{i,j\} \in \P$:
\item[(a)] If $W \in \W_n$ then $P_W(a_{\{i,j\}}^\w)$ is the arc in $\linfw$  that connects the exits of type $i$ and $j$ of $W$.
\item[(b)] If  $Y' \in \Y'_n$ then $P_{Y'}(a_{\{i,j\}}^\w)$ is the arc in $\linfy$  that connects the exits of type $i$ and $j$ of $Y'$.
\item[(c)] If  $W' \in \W'_n$ then $P_{W'}(a_{\{i,j\}}^\y)$ is the arc in $\linfw$ that connects the exits of type $i$ and $j$ of $W'$.
\item[(d)] If $Y \in \Y_n$ then $P_{Y}(a_{\{i,j\}}^\y)$ is the arc $\linfy$ that connects the exits of type $i$ and $j$ of $Y$.
\end{enumerate}
\end{lemma}
\begin{proof}
By the construction of  triangular labyrinth fractals, and the construction of the arcs in fractals, we obtain, for any arcs $a \subset \linfw $ and $a'\subset \linfy$, that $P_W(a), P_{W'}(a')\subset \linfw$ and $P_Y(a'), P_{Y'}(a)\subset \linfy$, and these images are also arcs.  The uniqueness  of the arc between given distinct points in each of the dendrites $\linfw$ and $\linfy$, and the definition of the exits of white or yellow triangles  imply the facts of the present lemma.
\end{proof}

Let $W_1, \ldots, W_{2h_w+1}$ be the path in $\G(\W_1 \cup \W'_1)$ from $W_1 = e_i^\w$ to $W_{2h_w+1} = e_j^\w$ and denote for each $n \in \{1, \ldots, 2h_w+1\}$ by $\{i_n,j_n\}$ the type of $W_n$ (according to the proof of Proposition~\ref{prop:path_matrix}). 
Then we have 
\[a_{\{i,j\}}^\w = \bigcup_{r=0}^{h_w} P_{W_{2r+1}}(a^\w_{\{i_r,j_r\}}) \cup \bigcup_{r=1}^{h_w} P_{W_{2r}}(a^\y_{\{i_r,j_r\}}),\]
where $P_{W_n}$ is a contraction with ratio $m^{-1}$ for each $r \in \{1, \ldots, 2h_w+1\}$.

Therefore, the six arcs between exits of the two fractals (three in $\linfw$ and three in $\linfy$) are the invariant sets of a GIFS. 
We denote the corresponding directed multigraph by $\HH'=(\V(\HH'),\E(\HH'))$. The set of its vertices
$\V(\HH')$ consists of the six arcs, that is
\[\V(\HH'):=\{a_{\pi}^\w: \pi \in \P\} \cup \{a_{\pi}^\y: \pi \in \P\},\]
where $\P=\{\{1,2\},\{1,3\},\{2,3\}\}$. 
The set of edges $\E(\HH')$ can be obtained as outlined above,  i.e.,
for each vertex the outgoing edges
are given by the triangles (including their types) that appear in the corresponding path in $\G(\W_1 \cup \W'_1)$ and $\G(\Y_1 \cup \Y'_1)$, respectively.
We see that the incidence matrix is the global path matrix $M$ from Proposition~\ref{prop:path_matrix}. This approach of the arcs in triangular labyrinth fractals with GIFS is used repeatedly in Sections \ref{sec:blocked} to \ref{sec:partially_blocked} in order to study properties like, e.g., the Hausdorff dimension of arcs between exits in triangular labyrinth fractals. 

\section{Blocked triangular labyrinth patterns systems}\label{sec:blocked}

Recall that a strip of type $1$ is a collection of triangles
\begin{multline*}
\{T_m(K,k_2,k_3): k_2, k_3 \in \NN, K+k_2+k_3=m-1\} \cup \\
\{T'_m(K,k_2,k_3): k_2, k_3 \in \NN, K+k_2+k_3=m-2\} \subset \T_{m} \cup \T'_{m}
\end{multline*}
for some fixed $K \in \{0, \ldots, m-1\}$.
We say that the white path (yellow path, respectively) of type $\{2,3\}$ is \emph{blocked} if
it contains two triangles $W_1, W_2$ that are not contained in the same strip of type $1$. 
Analogously we define  blocked (white and yellow) paths of type $\{1,3\}$  and of type $\{1,2\}$.

For  $\pi \in \P=\{\{1,2\}, \{1,3\}, \{2,3\}\}$,
we say that the triangular labyrinth patterns system $(\W \cup \W', \Y \cup \Y')$ is {\it $\pi$-blocked} if the yellow or the white path of type $\pi$ is blocked.
 A triangular labyrinth patterns system $(\W \cup \W', \Y \cup \Y')$ is { \it globally blocked} if it is {\it $\pi$-blocked} for all $\pi \in \P=\{\{1,2\}, \{1,3\}, \{2,3\}\}$.
We call the fractals defined by globally blocked triangular labyrinth patterns systems \emph{globally blocked} triangular labyrinth fractals.

Most results of this section concern globally blocked triangular labyrinth patterns systems and the corresponding fractals.

For globally blocked labyrinth systems all arcs in $\linfw$ and $\linfy$ have a fractal shape and a Hausdorff dimension larger than $1$. In order to prove this, let us recall that the arcs between the exits are obtained as attractors of a GIFS whose incidence matrix is 
\[M = \left(\begin{array}{cc} M_\w & \tilde M_\w \\ \tilde M_\y & M_\y \end{array}\right),\]
which we introduced at the end of Section~\ref{sec:arcs_TLF}. Moreover, $M_\w= \tilde M_\w + I_3$ and
$M_\y= \tilde M_\y + I_3$, where $I_3$ is the $3 \times 3$ identity matrix.
We are going to analyse this matrix $M$ in more detail. More precisely, we show that $M$ is a primitive matrix, i.e. there exists a power $n \ge 1$ such that all entries of $M^n$ are strictly positive.

\begin{theorem}\label{theorem:matrix_primitive}
Let $(\W \cup \W', \Y \cup \Y')$  be a globally blocked triangular labyrinth patterns system. Then its global path matrix $M$ is a primitive matrix with a dominant eigenvalue strictly larger than $m$.
\end{theorem}
For the proof of Theorem \ref{theorem:matrix_primitive} we need two lemmas. The first one shows us that, in fact, the matrix $M_\w + M_\y - I_3$ contains all necessary information concerning the eigenvalues of $M$.

\begin{lemma}\label{eigenvalues}
Every eigenvalue of $M_\w + M_\y - I_3$ is also an eigenvalue of
the global path matrix $M$.
Additionally, $M$ has a three dimensional eigenspace with respect to the eigenvalue $1$.
\end{lemma}
\begin{proof}
Let $\lambda$ be an eigenvalue of $M_\w + M_\y - I_3$ and $\mathbf{v} \in \RR^3$ a corresponding right eigenvector. By concatenation of the vectors $(M_\w  - I_3)\mathbf{v}$ and $(M_\y - I_3)\mathbf{v}$ we obtain a right eigenvector of $M$ with respect to $\lambda$ since
\[
\left(\begin{array}{cc}
M_\w & M_\w - I_3 \\ M_\y-I_3 & M_\y
\end{array}\right) \cdot \left(\begin{array}{c}(M_\w  - I_3)\mathbf{v}\\(M_\y - I_3)\mathbf{v}\end{array}\right)  = \lambda\left(\begin{array}{c}(M_\w  - I_3)\mathbf{v}\\(M_\y - I_3)\mathbf{v}\end{array}\right).\]
Furthermore, for any arbitrary vector $\mathbf{v} \in \RR^3$ a similar calculation yields
\[M\cdot\left(\begin{array}{c}\mathbf{v}\\-\mathbf{v}\end{array}\right) = \left(\begin{array}{c}\mathbf{v}\\-\mathbf{v}\end{array}\right).\]
This shows that $1$ is an additional eigenvalue of $M$ and the induced eigenspace is of dimension three.
\end{proof}

The next lemma implies that $M_\w + M_\y - I_3$ is a primitive matrix if and only if the triangular labyrinth patterns system $(\W \cup \W', \Y \cup \Y')$ is  globally blocked, and it includes an estimation for the row sum norm.

\begin{lemma}
\label{lemma:row_positive}
Let $\pi \in \{\{1,2\},\{1,3\},\{2,3\}\}$. If the triangular labyrinth patterns system $(\W \cup \W', \Y \cup \Y')$ is  $\pi$-blocked then all entries of the corresponding row of the matrix $M_\w + M_\y - I_3$ are strictly positive and their sum is strictly larger than $m$. Otherwise, the
diagonal entry in the row indexed by $\pi$ equals $m$ and the other two entries are zero.
\end{lemma}
\begin{proof}
We prove the lemma for $\pi=\{1,2\}$ and, correspondingly, the first row of $M_\w + M_\y - I_3$. For other choices of $\pi$ the proof runs analogously.
To start with, we need some notation and preliminary considerations.

Let $m_1, m_2, m_3$ denote the entries of the first row 
of $M_\w + M_\y - I_3$.
Observe
that from Proposition~\ref{prop:path_matrix} we know 
that both $\tilde M_\w=M_\w -I_3$ and $\tilde M_\y = M_\y - I_3$ are  non-negative matrices, hence  $M_\w + M_\y - I_3$ is a non-negative matrix with  strictly positive main diagonal, which means that $m_1 \geq 1$. 
Furthermore,  the first row of
$M_\w$ ($M_\y$, respectively) is
completely determined by the white (yellow, respectively) path of type $\{1,2\}$.
Let $k_1, k_2 \in \{0, \ldots, m-1\}$ such that $e^\w_1=T_m(0, k_1, m-1-k_1)$ and
$e^\w_2=T_m(k_2, 0, m-1-k_2)$.
Denote by  $W_1, \ldots, W_{2r_w+1}$ the white path of type $\{1,2\}$ in $\G(\W \cup \W')$, hence
 $W_1=e^\w_1$ and $W_{2r_w+1}=e^\w_2$. We have $W_{2h+1} \in \W$ for $h \in \{0, \ldots, r_w\}$ and
$W'_{2h} \in \W'$ for $h \in \{1, \ldots, r_w\}$.

Denote by $w^{(1)}_h, w^{(2)}_h, w^{(3)}_h$ the integers such that $W_h=T_m(w^{(1)}_h, w^{(2)}_h, w^{(3)}_h)$ for all $h \in \{1, \ldots, 2r_w+1\}$.
For any  $h  \in \{1, \ldots, 2r_w\}$ let $i_{h}$ be the label of the edge that connects $W_h$ with $W_{h+1}$ in $\G(\W \cup \W')$. By the definition of  these edge labels, $W_{h}$ and $W_{h+1}$ are $i_h$-neighbours and 
$W_{h}$ is of type $\{i_{h-1},i_{h}\}$. In particular, $i_h \not= i_{h+1}$.
Therefore, for every $h  \in \{1, \ldots, r_w\}$ we have
\begin{equation}\label{equat1}
\begin{split}
(w^{(1)}_{2h-1}, w^{(2)}_{2h-1}, w^{(3)}_{2h-1}) - (w^{(1)}_{2h}, w^{(2)}_{2h}, w^{(3)}_{2h}) =& (\delta_{1i_{2h-1}},\delta_{2i_{2h-1}},\delta_{3i_{2h-1}}),\\
(w^{(1)}_{2h+1}, w^{(2)}_{2h+1}, w^{(3)}_{2h+1}) - (w^{(1)}_{2h}, w^{(2)}_{2h}, w^{(3)}_{2h}) =& (\delta_{1i_{2h}},\delta_{2i_{2h}},\delta_{3i_{2h}}).
\end{split}
\end{equation}

In the spirit of  the proof of Proposition~\ref{prop:path_matrix} we define 
for each ordered pair $(i,j) \in \{1,2,3\}^2$ with $i \not=j$
\[\tilde m^\w_{(i,j)} := \#\{h \in \{1, \ldots, r_w\}: (i_{2h-1},j_{2h}) = (i,j)\}.\]
Clearly, the triangle $W_{2h}$ is of  type $\{i_{2h-1},i_{2h}\}$. 
Therefore, the entries of the first row  of $\tilde M_\w = M_\w - I_3$ 
are given by the three integers
\begin{align*}
\tilde m^\w_{\{1,2\},\{1,2\}} = &  \tilde m^\w_{(1,2)} + \tilde m^\w_{(2,1)},\\
\tilde m^\w_{\{1,2\},\{1,3\}} = & \tilde m^\w_{(1,3)} + \tilde m^\w_{(3,1)}, \\
\tilde m^\w_{\{1,2\},\{2,3\}} = & \tilde m^\w_{(2,3)} + \tilde m^\w_{(3,2)}.
\end{align*}
But differently from the type  (which is an unordered pair) of a triangle with respect to a path, the ordered pair $(i_{2h-1},i_{2h})$ contains information concerning the ``direction'' along which  the triangle $W_{2h}$ is passed when we go from the exit $e^\w_1=W_1$ to the exit  $e^\w_2=W_{2r_w+1}$. In particular, by using \eqref{equat1} we obtain
\begin{equation}\label{equat2}
\left(\begin{array}{c} k_2\\ -k_1 \\ k_1-k_2 \end{array}\right) = 
\left(\begin{array}{c} w^{(1)}_{2r_w+1} \\ w^{(2)}_{2r_w+1} \\ w^{(3)}_{2r_w+1} \end{array}\right)-
\left(\begin{array}{c} w^{(1)}_{1} \\ w^{(2)}_{1} \\ w^{(3)}_{1} \end{array}\right) =
\left(\begin{array}{c}
\tilde m^\w_{(2,1)} + \tilde m^\w_{(3,1)} - \tilde m^\w_{(1,2)} - \tilde m^\w_{(1,3)} \\
\tilde m^\w_{(1,2)} + \tilde m^\w_{(3,2)} - \tilde m^\w_{(2,1)} - \tilde m^\w_{(2,3)} \\
\tilde m^\w_{(1,3)} + \tilde m^\w_{(2,3)} - \tilde m^\w_{(3,1)} - \tilde m^\w_{(3,2)}
\end{array}\right).
\end{equation}

Now we turn our attention to $\G(\Y \cup \Y')$. Let $Y_1, \ldots, Y_{2r_y+1}$ be the yellow path of type $\{1,2\}$ 
from $Y_1=e^\y_1=(0,m-1-k_1,k_1)$ to $Y_{2r_y+1}=e^\y_2=(m-1-k_2,0,k_2)$. Analogously,  for all  $h  \in \{1, \ldots, 2r_y\}$ let $i_{h}$ be the label of the edge that connects $Y_h$ with $Y_{h+1}$ and set for each ordered pair $(i,j) \in \{1,2,3\}^2$ with $i \not=j$
\[\tilde m^\y_{(i,j)} := \#\{h \in \{1, \ldots, r_y\}: (i_{2h-1},i_{2h}) = (i,j)\}.\]
Similarly, the entries in the first row of  $\tilde M_\y = M_\y - I_3$ are given by
\begin{align*}
\tilde m^\y_{\{1,2\},\{1,2\}} = &  \tilde m^\y_{(1,2)} + \tilde m^\y_{(2,1)},\\
\tilde m^\y_{\{1,2\},\{1,3\}} = & \tilde m^\y_{(1,3)} + \tilde m^\y_{(3,1)}, \\
\tilde m^\y_{\{1,2\},\{2,3\}} = & \tilde m^\y_{(2,3)} + \tilde m^\y_{(3,2)}.
\end{align*}
and we have
\[
\left(\begin{array}{c} m-1-k_2 \\ k_1-m+1 \\ k_2-k_1 \end{array}\right) = 
\left(\begin{array}{c}
\tilde m^\y_{(2,1)} + \tilde m^\y_{(3,1)} - \tilde m^\y_{(1,2)} - \tilde m^\y_{(1,3)} \\
\tilde m^\y_{(1,2)} + \tilde m^\y_{(3,2)} - \tilde m^\y_{(2,1)} - \tilde m^\y_{(2,3)} \\
\tilde m^\y_{(1,3)} + \tilde m^\y_{(2,3)} - \tilde m^\y_{(3,1)} - \tilde m^\y_{(3,2)}
\end{array}\right).
\]

Combining these results yields
\begin{equation}\label{equat3}
\begin{split}
m_1 = & \tilde m^\w_{(1,2)} + \tilde m^\y_{(1,2)} + \tilde m^\w_{(2,1)}+ \tilde m^\y_{(2,1)} +1,\\
m_2 = & \tilde m^\w_{(1,3)} +\tilde m^\y_{(1,3)} + \tilde m^\w_{(3,1)}+\tilde m^\y_{(3,1)}, \\ 
m_3 = & \tilde m^\w_{(2,3)}+ \tilde m^\y_{(2,3)} + \tilde m^\w_{(3,2)}+\tilde m^\y_{(3,2)},
\end{split}
\end{equation}
and furthermore the equations
\begin{equation}\label{equat4}\begin{split}
m-1 = \tilde m^\w_{(2,1)} +\tilde m^\y_{(2,1)} + \tilde m^\w_{(3,1)}+\tilde m^\y_{(3,1)} - \tilde m^\w_{(1,2)}- \tilde m^\y_{(1,2)}  - \tilde m^\w_{(1,3)} - \tilde m^\y_{(1,3)}, \\
1-m = \tilde m^\w_{(1,2)} +\tilde m^\y_{(1,2)} + \tilde m^\w_{(3,2)}+\tilde m^\y_{(3,2)} - \tilde m^\w_{(2,1)}- \tilde m^\y_{(2,1)}  - \tilde m^\w_{(2,3)} - \tilde m^\y_{(2,3)}, \\
0 = \tilde m^\w_{(1,3)} +\tilde m^\y_{(1,3)} + \tilde m^\w_{(2,3)}+\tilde m^\y_{(2,3)} - \tilde m^\w_{(3,1)}- \tilde m^\y_{(3,1)}  - \tilde m^\w_{(3,2)} - \tilde m^\y_{(3,2)}.
\end{split}\end{equation}

Now, let us start with the concrete proof of the two implications stated in the lemma.
First, we assume that our system is not $\{1,2\}$-blocked. Then neither the white path of type $\{1,2\}$ nor the yellow path of type $\{1,2\}$ is blocked.
For the white path this means that $w^{(3)}_h = K$ for all $h \in \{1, \ldots, 2r_w+1\}$ where $K=m-1-k_1 = m-1-k_2$, thus in this case $k_1$ and $k_2$ must coincide. 
Then \eqref{equat1} and the fact that two consecutive labels are distinct yields that
$\tilde m^\w_{(1,2)} = k_1$ and
$\tilde m^\w_{(2,1)}=\tilde m^\w_{(1,3)}=\tilde m^\w_{(3,1)}= \tilde m^\w_{(2,3)}=\tilde m^\w_{(3,2)}=0$.
Analogously, $\tilde m^\y_{(1,2)} = m-1-k_1$ and
$\tilde m^\y_{(2,1)}=\tilde m^\y_{(1,3)}=\tilde m^\y_{(3,1)}= \tilde m^\y_{(2,3)}=\tilde m^\y_{(3,2)}=0$.
Then, from the formul{\ae} \eqref{equat3} we obtain that $m_1= m$ and $m_2=m_3=0$.

Now we assume that our system is $\{1,2\}$-blocked. 
We indirectly prove that this implies $m_2>0$ and $m_3>0$. W.l.o.g., we may  assume  that $m_3=0$, hence, neither the white nor the yellow path of type $\{1,2\}$ contain a triangle of type $\{2,3\}$. 
This means that for all $h \in \{1, \ldots, 2r_w-1\}$ we have $(i_{h}, i_{h+1}) \not\in \{(2,3), (3,2)\}$ and also  $(j_{h}, j_{h+1}) \not\in \{(2,3), (3,2)\},$ for all $h \in \{1, \ldots, 2r_y-1\}$ and 
\begin{equation}\label{equat5}
\tilde m^\w_{(2,3)} = \tilde m^\y_{(2,3)} = \tilde m^\w_{(3,2)} = \tilde m^\y_{(3,2)}=0.
\end{equation}

The assumption that our system is $\{1,2\}$-blocked implies that at least one of the two paths, the white path of type $\{1,2\}$  or the yellow path of type $\{1,2\}$, is blocked. If the white path of type $\{1,2\}$ is blocked then there must be at least one index $h \in \{2, \ldots, 2r_w+1\}$ such that $w^{(3)}_{h} \not= 
w^{(3)}_{1} = m-1-k_1,$   since the path has to ``leave'' the strip of type $3$ containing the exit $W_1=e_1^\w$, and  the formul{\ae} in \eqref{equat1} imply that there exists at least one $h \in \{1, \ldots, 2r_w\}$ with $i_{h} = 3$. 
Analogously, if the yellow path of type $\{1,2\}$ is blocked then
there exists at least one index $h \in \{1, \ldots, 2r_y\}$ with $j_{h} = 3$. 
Due to \eqref{equat5} we conclude that $\tilde m^\w_{(1,3)} +\tilde m^\y_{(1,3)} + \tilde m^\w_{(3,1)}+\tilde m^\y_{(3,1)} >0$.
Now, \eqref{equat4} and \eqref{equat5} imply that 
 either $\tilde m^\w_{(1,3)}>0$ or $\tilde m^\y_{(1,3)}>0$. Assume that 
$\tilde m^\w_{(1,3)}>0$.
Then there exists an $h_0 \in \{1, \ldots, r_w\}$ such that
$i_{2h_0-1}=1$ and $i_{2h_0}=3$. Now observe that in this case necessarily $i_{2h_0+1}=1$ (provided that $h_0<r_w$). Indeed,
$i_{2h_0+1}=3$ is not possible since $i_{2h_0} \not= i_{2h_0+1}$. We also can exclude that $i_{2h_0+1}=2,$ since in that case $W_{2h_0+1}$ would be of type $\{2,3\}$.
For the next edge we have $i_{2h_0+2} \not=1$ and thus  $i_{2h_0+2}=2$ or $i_{2h_0+2}=3$. By successively applying this  argumentation we conclude that for every $h \in \{h_0, \ldots, r_w\}$ we have $(i_{2h-1}, i_{2h}) \in \{(1,2), (1,3)\}$.

Similarly, we have $i_{2h_0-2} = 2$ or $i_{2h_0-2} = 3$ and since, 
due to \eqref{equat5},
 $(i_{2h_0-3},i_{2h_0-2}) \not\in \{(2,2),(2,3),(3,2),(3,3)\}$, we necessarily have $i_{2h_0-3}=1$. Therefore, $(i_{2h-1}, i_{2h}) \in \{(1,2), (1,3)\}$ holds for all $h \in \{1, \ldots, r_w\},$ showing that $\tilde m^\w_{(1,2)}+\tilde m^\w_{(1,3)}>0$ while 
$\tilde m^\w_{(2,1)}+\tilde m^\w_{(3,1)}=0$. This is a contradiction to \eqref{equat2}, since $k_2 \geq 0$. For $\tilde m^\y_{(3,1)}>0$ the argumentation works analogously.

Finally we show that $m_1+m_2+m_3>m$. Indeed, by using the formul{\ae}  \eqref{equat3}, \eqref{equat4} and the fact that $m_3 \geq 1$ we get
\[m_1+m_2+m_3 \geq 1+\tilde m^\w_{(2,1)} + \tilde m^\y_{(2,1)} + \tilde m^\w_{(3,1)}+\tilde m^\y_{(3,1)} +  m_3
\geq 1+ m-1 +1\geq m+1.\]
\end{proof}

\begin{corollary}
\label{cor:matrix_primitive}
The matrix $M_\w + M_\y - I_3$ is a  matrix with all entries positive and row sum norm strictly larger than $m$ if and only if the triangular labyrinth patterns system $(\W \cup \W', \Y \cup \Y')$ is  globally blocked.
\end{corollary}

\begin{proof}[Proof of Theorem~\ref{theorem:matrix_primitive}]
We first prove that $M$ is primitive. We consider the square of the matrix $M$ and obtain
\[M^2=  
\left(\begin{array}{cc}
M_\w & M_\w - I_3 \\ M_\y - I_3 & M_\y
\end{array}\right)^2 = 
\left(\begin{array}{cc}
M^2_{11} & M^2_{12} \\ M^2_{21} & M^2_{22}
\end{array}\right),\]
where
\begin{align*}
M^2_{11} & = (M_\w - I_3)(M_\w+M_\y-I_3) + M_\w, \\
M^2_{12} & = (M_\w - I_3)(M_\w + M_\y), \\ 
M^2_{21} & = (M_\y - I_3)(M_\w + M_\y), \\
M^2_{22} & = (M_\y - I_3)(M_\w+M_\y-I_3) + M_\y.
\end{align*}

From Lemma~\ref{lemma:row_positive} we conclude that $M_\w+M_\y-I_3$ is a matrix with strictly positive entries. Thus, if we show that every row of the (non-negative) matrices $M_\w - I_3$ and $M_\y - I_3$, respectively, has at least one positive entry we are done. We only show this for the first row of $M_\w - I_3$; for the other five rows the argumentation runs analogously. 
Indeed, by taking into account Corollary~\ref{cor:iteration_matrix} (and the shape of the matrix $M$), we see that the row sum of the first row of $2M_\w - I_3$ gives $\ell_{\{1,2\}}(1)$, the length of the
path in $\G(\W \cup \W')$ from the exit $e^\w_1$ to the exit $e^\w_2$ which is always an odd number, and by Proposition~\ref{prop:no_double_exits} we have $\ell_{\{1,2\}}(1)>1$ (since $\ell_{\{1,2\}}(1)=1$ would mean that $e^\w_1 \not= e^\w_2$). Thus,  all entries of $M^2$ are positive.

In order to estimate the dominant eigenvalue, we observe that by Lemma~\ref{lemma:row_positive} the matrix $M_\w+M_\y-I_3$ is primitive and, thus, it has a dominant eigenvalue $\theta$. Furthermore,
the row sum norm of
 $M_\w+M_\y-I_3$ is strictly larger than $m$, hence we have  $\theta>m$ and from Lemma~\ref{eigenvalues} we conclude that $\theta$ is also the dominant eigenvalue of $M$.
\end{proof}

\begin{theorem}\label{theorem:dim_arcs}
Let  $(\W \cup \W', \Y \cup \Y')$ be a globally blocked triangular labyrinth patterns system.
Then the arc between any two distinct exits $\mathbf{E}_i^\w$ and $\mathbf{E}_j^\w$ in $\linfw$ (and also the arc between any two distinct exits $\mathbf{E}_i^\y$ and $\mathbf{E}_j^\y$ in $\linfy$) has Hausdorff dimension  $\log_{m}(\theta)>1$, where $\theta$ is the dominant eigenvalue  of $M_\w + M_\y - I_3$. 
\end{theorem}
\begin{proof}
At the end of Section~\ref{sec:arcs_TLF} we have mentioned that the arcs are given by a GIFS, where the incidence matrix of the associated multigraph $\HH'$ is 
\[M=\left(\begin{array}{cc} M_\w & \tilde M_\w \\ \tilde M_\y & M_\y \end{array}\right).\]
The similarity ratios of the maps in this GIFS are all equal $m^{-1}$.

From Theorem~\ref{theorem:matrix_primitive} we know that $M$ is a primitive matrix with dominant eigenvalue $\theta >m$. The primitivity of $M$ implies that $\HH'$ is strongly connected. Therefore,  by applying  \cite[Theorem~3]{MW}, we obtain
\[\dim_H(a_{\{i,j\}}^\w)=\dim_H(a_{\{i,j\}}^\y) = \log_m(\theta) > 1.\]
\end{proof}

\section{On arc length in globally blocked triangular labyrinth fractals}\label{sec:arc_length}
Throughout this section we study properties of globally blocked triangular labyrinth fractals, i.e. we assume the triangular labyrinth fractals studied in this section to be defined by globally blocked triangular labyrinth patterns systems. Moreover, we use the $1$-dimensional Lebesgue measure, which we denote by $\lambda$, for the length of arcs in the fractals.

\begin{corollary}\label{cor:infinite_arcs_btw_exits}
The globally blocked fractals $\linfw$ and $\linfy$ have the property that the length of any arc in the fractal between any two exits of the fractal is infinite.
\end{corollary}
\begin{proof}
The result easily follows from Theorem \ref{theorem:dim_arcs}, based on the relationship between the  Hausdorff dimension, the Hausdorff measure and the Lebesgue measure (see, e.g. \cite{falconerbook}). 
\end{proof}

\begin{lemma}\label{lemma:infinite_arcs_in_triangles}
Let $n\ge 1$, $W \in \W_n \cup \W'_n$, and $\{i,j\} \in \P$. Then the arc in $\linfw$ that connects the exits of type $i$ and $j$ of the triangle $W$ has infinite length. The analogous assertion holds for $Y \in \Y_n \cup \Y'_n$. 
\end{lemma}

\begin{proof}
By Lemma~\ref{lemma:projections_of_arcs_between_exits} the arc that connects the exits of type $i$ and $j$ of the triangle $W$ is given by $P_W(a_{\{i,j \}}^\w)$ or $P_W(a_{\{i,j \}}^\y)$. By Corollary~\ref{cor:infinite_arcs_btw_exits}, $a_{\{i,j \}}^\w$ and $a_{\{i,j \}}^\y$ have infinite length and since  $P_W$ is a similarity, it follows that $P_W(a_{\{i,j \}}^\w)$ also infinite length, too.
\end{proof}

\begin{theorem}\label{theorem:infinite_arcs_in_fractal}
Let $\mathbf X, \mathbf Y \in \linfw$, and $\mathbf X \ne \mathbf Y $. Then the length of the arc in $\linfw$ that connects the points $\mathbf X$ and $\mathbf Y$ is infinite, i.e.,   $\lambda \left(a (\mathbf X, \mathbf Y )\right)=\infty$.
\end{theorem}

\begin{proof} For $n\ge 1$ let $W_n(\mathbf X)$ and  $W_n(\mathbf Y)$ be the triangles in $\W_n \cup \W'_n$ that contain $\mathbf X$ and $\mathbf Y$, respectively (chosen as in the proof of Lemma \ref{lemma:arc_construction}).
Let $p_n$ be the path in $\G(\W_n \cup \W'_n) $ that connects these triangles (vertices of the graph). 

Since $\mathbf X \ne \mathbf Y$, one can choose $n$ large enough such that the path $p_n$ has length at least $3$, i.e., $p_n$ consists of at least three triangles. Let $W \in p_n$, $W\ne  W_n(\mathbf X), W_n( \mathbf Y)$. By the construction of the arc $a(\mathbf X, \mathbf Y)$ according to Lemma \ref{lemma:arc_construction} and Lemma \ref{lemma:arc_exits_triangle} it follows that $a \cap W$ is an arc between two exits of the triangle W. Lemma \ref{lemma:infinite_arcs_in_triangles} implies that the length of the arc $a\cap W$ is infinite, which immediately yields $\ell \left(a (\mathbf X, \mathbf Y )\right)=\infty$.
\end{proof}

Before dealing with tangents to arcs in triangular labyrinth fractals let us recall the definition from the book by Tricot \cite{Tricot_book}.

\begin{definition}\cite[Sect. 7.2., p.73]{Tricot_book}
We say that there exists a tangent $\tau$ at a point $\mathbf{X}_0$ to an arc $a$, if for every positive angle $\phi$ there exists an $\varepsilon$ such that for all $\mathbf X \in a$ that satisfy $\abs {\mathbf{X}_0 - \mathbf X }< \varepsilon$, the angle between the straight line $\tau$ and the straight line through $\mathbf X$ and $\mathbf{X}_0$ is not greater than $\phi$.
\end{definition}

\begin{lemma}\label{lemma:no_tangent_exits} 
Let $\linfw$ be the white fractal generated by a globally blocked triangular labyrinth patterns system. For any two distinct exits ${\mathbf E}_i^\w$ and ${\mathbf E}_j^\w$ of $\linfw$, there exists no tangent  to the arc $a_{\{i,j,\}}^\w$ at ${\mathbf E}_i^\w$ and ${\mathbf E}_j^\w$. The analogous assertion holds for distinct exits of the fractal $\linfy$.
\end{lemma}

\begin{proof}
We prove the above assertion for one of the exits, $\mathbf E_i^\w$. We assume without loss of generality that $i=1$.
By Lemma \ref{lemma:exits_are_fixpoints}, for any $n \ge 1$, $\mathbf E_1^\w$ and $\mathbf E_j^\w$ lie each in exactly one white triangle of level $n$, more precisely, $\mathbf E_1^\w\in W_n(\mathbf E_1^\w)=e_1^\w(n)$, and $\mathbf E_j^\w\in W_n(\mathbf E_j^\w)= e_j^\w(n)$. Let $p_n$ be the path in $\G(\W_n \cup \W'_n)$ from $e_1^\w(n)$ to $ e_j^\w(n)$. The triangle $e_1^\w(n)$ can be of type $\{1,2\}$ or $\{1,3\}$ with respect to $p_n$, for all $n\ge 1$. This implies that at least one of the types $\{1,2\}$ and $\{1,3\}$ (with respect to $p_n$) occurs infinitely many times in the sequence $\{W_n(\mathbf E_1^\w)\}_{n\ge 1}$. W.l.o.g. we assume there is a sequence $\{n_k\}_{k\ge 1}$ such that $W_{n_k}(\mathbf E_1^\w)$ is a triangle of type $\{1,2\}$ with respect to $p_n$. By Lemma \ref{lemma:arc_exits_triangle} it then follows that $\W_{n_1}\cap a_{\{i,j,\}}^\w$ is an arc between the exits of type $1$ and $2$ of the triangle $W_{n_1}(\mathbf E_1 ^{\w})$.

 By Theorem \ref{theorem:matrix_primitive}, the path matrix of the globally blocked triangular labyrinth patterns system is primitive. This implies that there exists an $n\ge 1$ such that 
images through projection maps $P_W$, with $W \in \W_n \cup \W'_n$, of all three exits of $\linfw$ and all three exits of $\linfy$ 
must be contained in any arc in $\linfw$ that connects two distinct exits of $\linfw$. It follows that an arc in $\linfw$ that connects two distinct exits of a triangle cannot be a straight line segment. Therefore, one can choose two distinct points $\mathbf Q_1$ and $\mathbf Q_2$ in $W_{n_1}(\mathbf E_1^\w) \cap a_{\{i,j,\}}^\w$, with $\mathbf Q_1,\mathbf Q_2 \ne \mathbf E_1^\w$, having the property that the vectors $\mathbf Q_1 - \mathbf E_1^\w$ and   $\mathbf Q_2 - \mathbf E_1^\w$ are linearly independent. 

It is easy to see that the sequences of points $\{ \mathbf A_{k} \}_{k\ge 1}$, with $\displaystyle \mathbf A_k=P_{W_{n_k}(\mathbf E_1^\w)}(\mathbf Q_1)$, and $\{ \mathbf B_{k} \}_{k\ge 1}$, with $\displaystyle \mathbf B_k=P_{W_{n_k}(\mathbf E_1^\w)}(\mathbf Q_2)$, converge to $\mathbf E_1^\w$, and $\mathbf A_{k},\mathbf B_{k}\ne \mathbf E_1^\w$, for all $k\ge 2$.
By Lemmas \ref{lemma:arc_exits_triangle} and \ref{lemma:projections_of_arcs_between_exits} we obtain $\mathbf A_{k},\mathbf B_{k}\in a_{\{i,j,\}}^\w$, for all $k\ge 1$.

On the other hand, by assertion $(c)$ in Lemma \ref{lemma:exits_are_fixpoints}, all $\mathbf A_{k}$, with $k\ge 1$, are collinear. Let $a_1$ be the straight line that contains these points. Likewise, all $\mathbf B_k$, for $k\ge 1$, lie on a straight line which we denote by $a_2$. Since the vectors  $\mathbf Q_1 - \mathbf E_1^\w$ and   $\mathbf Q_2 - \mathbf E_1^\w$ are linearly independent, it follows that $a_1 \ne a_2.$ This implies that there exists no tangent to the arc $a_{\{i,j,\}}^\w$ at the point $\mathbf E_1^\w$.
\end{proof}

\begin{corollary}\label{cor:points_notangent_dense}
If $a\subset \linfw$ (or $a\subset \linfy$) is an arc that connects two points in the globally blocked triangular labyrinth fractal, then the set of all points, at which there exists no tangent to $a$, is dense in $a$. 
\end{corollary}

\begin{proof}
We give a proof for $a=a(\mathbf X, \mathbf Y) \subset \linfw$.  For $n\ge 1$ let $p_n$ be the path that connects $W_n(\mathbf X)$ and $W_n(\mathbf Y)$ in $\G(\W_n \cup \W'_n)$, and $p_n=\{W_1,\dots, W_{2k+1}  \}$,  constructed as in Lemma \ref{lemma:arc_construction}. Let $h\in \{1, \dots, 2k+1\}$. From Lemmas \ref{lemma:arc_exits_triangle} and \ref{lemma:projections_of_arcs_between_exits} it follows that $W_h \cap a$ is an arc that connects the exit of type $i_1$ and that of type $i_2$, with $1\le i_1< i_2 \le 3$ of $W_h$. Then, the set 
\[ a \cap \left(\bigcup_{n=1}^{\infty}\bigcup_{W\in \W_n}P_W\Big(\left\{\mathbf E_1^\w,\mathbf E_2^\w, \mathbf E_3^\w   \right\}\Big) \cup \bigcup_{W'\in \W'_n}P_{W'}\Big(\left\{\mathbf E_1^\y,\mathbf E_2^\y, \mathbf E_3^\y   \right\}\Big) \right)\]
 is dense in $a$. On the other hand, from Lemma \ref{lemma:no_tangent_exits} we then conclude that the above set is a subset of all points at which no tangent to a exists.
\end{proof}

The following theorem is the main result obtained for totally blocked triangular labyrinth fractals, that also summarises some of the results already proven in this paper.
\begin{theorem}\label{theorem:main_result_blocked}
Let $\linfw$ and $\linfy$ be the fractals generated by a globally blocked triangular labyrinth patterns system with global path matrix $M,$ let $\theta$ be the largest eigenvalue of $M.$ Then the following assertions hold:
\begin{itemize}
\item[(a)]Between any two points in $\linfw$ there is a unique arc  $a\in \linfw$ that connects them.
\item[(b)]The length of the arc $a$ is infinite. 
\item[(c)] $\dim_H(a)=\frac{\log \theta}{\log m}  > 1$ .
\item[(d)]The set of points of the arc $a$, where there exists no tangent to the arc, is dense in $a$. 
\end{itemize}
The analogous affirmations hold for the yellow fractal  $\linfy$.
 \end{theorem}

\begin{proof}
The assertion (a) is implied by the fact that $\linfw$ is a dendrite. Theorem \ref{theorem:infinite_arcs_in_fractal}  yields (b), and (d) is due to Corollary \ref{cor:points_notangent_dense}. The  dimension of the arcs follows from Theorem \ref{theorem:dim_arcs} in the previous section, the GIFS construction of the fractals, the structure of the arcs (the arc between two exits of the fractal is a finite union of arcs that connect exits of triangles), and the self-similarity properties of the fractal.
\end{proof}

\section{On partially blocked triangular labyrinth patterns systems and fractals}\label{sec:partially_blocked}

First, let us note that it is impossible to have a triangular labyrinth patterns system that is not $\pi$-blocked, for all $\pi \in \P$.
This is due to the fact that triangular labyrinth fractals are dendrites, or, even more basically, this would contradict the tree property of labyrinth patterns. Therefore, triangular labyrinth patterns systems are either globally blocked or partially blocked, i.e., $\pi$-blocked for some bot not all three pairs $\pi$. In this section we illustrate what is meant by a ``partially blocked triangular labyrinth patterns system''.
We call the triangular labyrinth fractals obtained from such partially blocked patterns systems \emph{partially blocked triangular labyrinth fractals}. 

First, let us analyse the case when a triangular labyrinth patterns system is blocked in only one direction, i.e., there exists a unique $\pi \in \P$ such that the labyrinth patterns system is $\pi$-blocked. Therefore, let us assume, w.l.o.g., that the triangular labyrinth patterns system is $\{2,3\}$-blocked.
As a consequence of the fact that the system is not $\{1,2\}$-blocked, we get :

\begin{align*}
\mathbf E_1^\w=\alpha\mathbf{P}_2 + (1-\alpha)\mathbf{P}_3 & \text{\;  and \;} \mathbf E_1^\y=(1-\alpha)\mathbf{P}_2+ \alpha\mathbf{P}_3, \text { and}\\
\mathbf E_2^\w=\alpha\mathbf{P}_1 + (1-\alpha)\mathbf{P}_3 & \text{\;  and \;} \mathbf E_2^\y=(1-\alpha)\mathbf{P}_1 + \alpha\mathbf{P}_3, 
\end{align*}
for some $\alpha \in (0,1)$. The fact that the system is, in addition, not $\{1,3\}$-blocked implies
that the third exit of the fractals generated by the system satisfy 
\begin{align*}
\mathbf E_3^\w=\alpha\mathbf{P}_1 + (1-\alpha)\mathbf{P}_2 & \text{\;  and \;} \mathbf E_3^\y=(1-\alpha)\mathbf{P}_1+ \alpha\mathbf{P}_2,
\end{align*}
where the value of $\alpha$ is the same as in the first two lines of fromul{\ae} above.

Moreover, since $\linfw$ and $\linfy$ are dendrites, and taking into account Proposition \ref{prop:exists_C},
we conclude that $a_{\{2,3\}}^\w = a_{\{1,2\}}^\w \cup a_{\{1,3\}}^\w$ and
$a_{\{2,3\}}^\y = a_{\{1,2\}}^\y \cup a_{\{1,3\}}^\y$.
%}

\begin{remark}
\label{remark:cores}
It easily follows that the cores $\W^0$ and $\Y^0$ (see Proposition \ref{prop:exists_W_star}), respectively, of the white and yellow patterns of a $m$-triangular labyrinth patterns system that is $\{2,3\}$-blocked, but is neither $\{1,2\}$- nor $\{1,3\}$-blocked, have in essence the same shape, namely the white pattern consisting only of the $k$-th strip of type $2$ and the $(m-1-k)$-th strip of type $3$, and   the corresponding yellow pattern consisting of  of the $(m-1-k)$-th strip of type $2$ and the $k$-th strip of type $3$, with $k \in \{1,\dots,m-2 \}$. 
\end{remark}

\begin{example}\label{Beispiel2}
Let $m=4$ and consider the $4$-triangular labyrinth patterns system
\begin{align*}
\W_{1}:= & \{T_{4}({\bf k}): {\bf k} \in \{(0, 0, 3),(0, 2, 1),(1, 0, 2), (1, 1, 1), (1, 2, 0), (2, 0, 1)\}\}, \\
\W'_{1} := & \{T'_{4}({\bf k}): {\bf k} \in \{(0, 0, 2),(0, 1, 1), (0, 2, 0), (1, 0, 1)\}\}, \\
\Y_{1}:= & \{T_{4}({\bf k}): {\bf k} \in \{(0, 1, 2), (1, 0, 2), (1, 1, 1), (2, 1, 0)\}\}, \\
\Y'_{1} := & \{T'_{4}({\bf k}): {\bf k} \in \{(0, 0, 2), (0, 1, 1), (1, 1, 0), (2, 0, 0)\}\}.
\end{align*}
Figure~\ref{fig:Bsp2WY1} shows the white and yellow triangles of level $1$.
\begin{figure}[h]
\begin{center}
\includegraphics[width=0.2\textwidth]{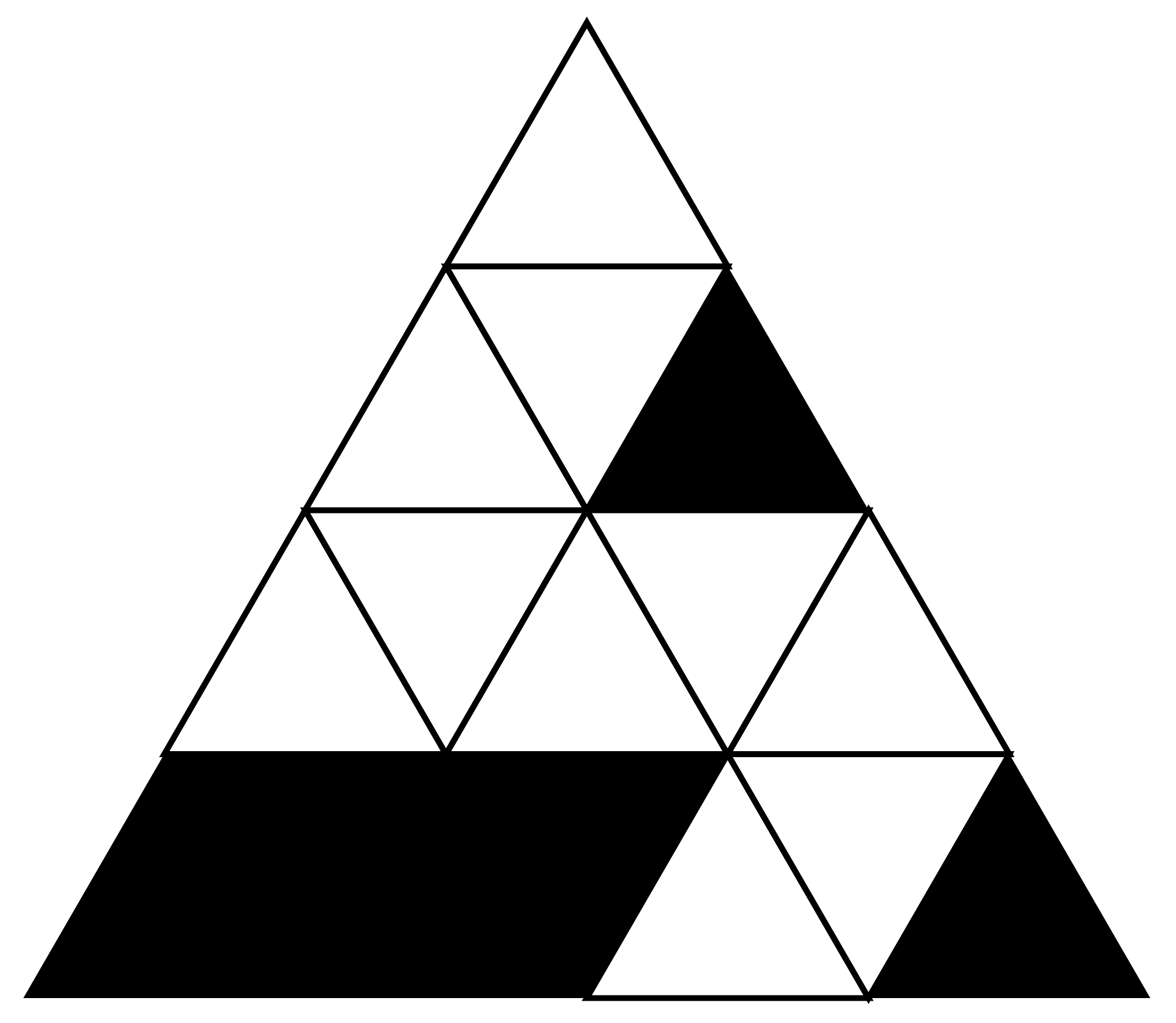}
\hspace{3cm}
\includegraphics[width=0.2\textwidth]{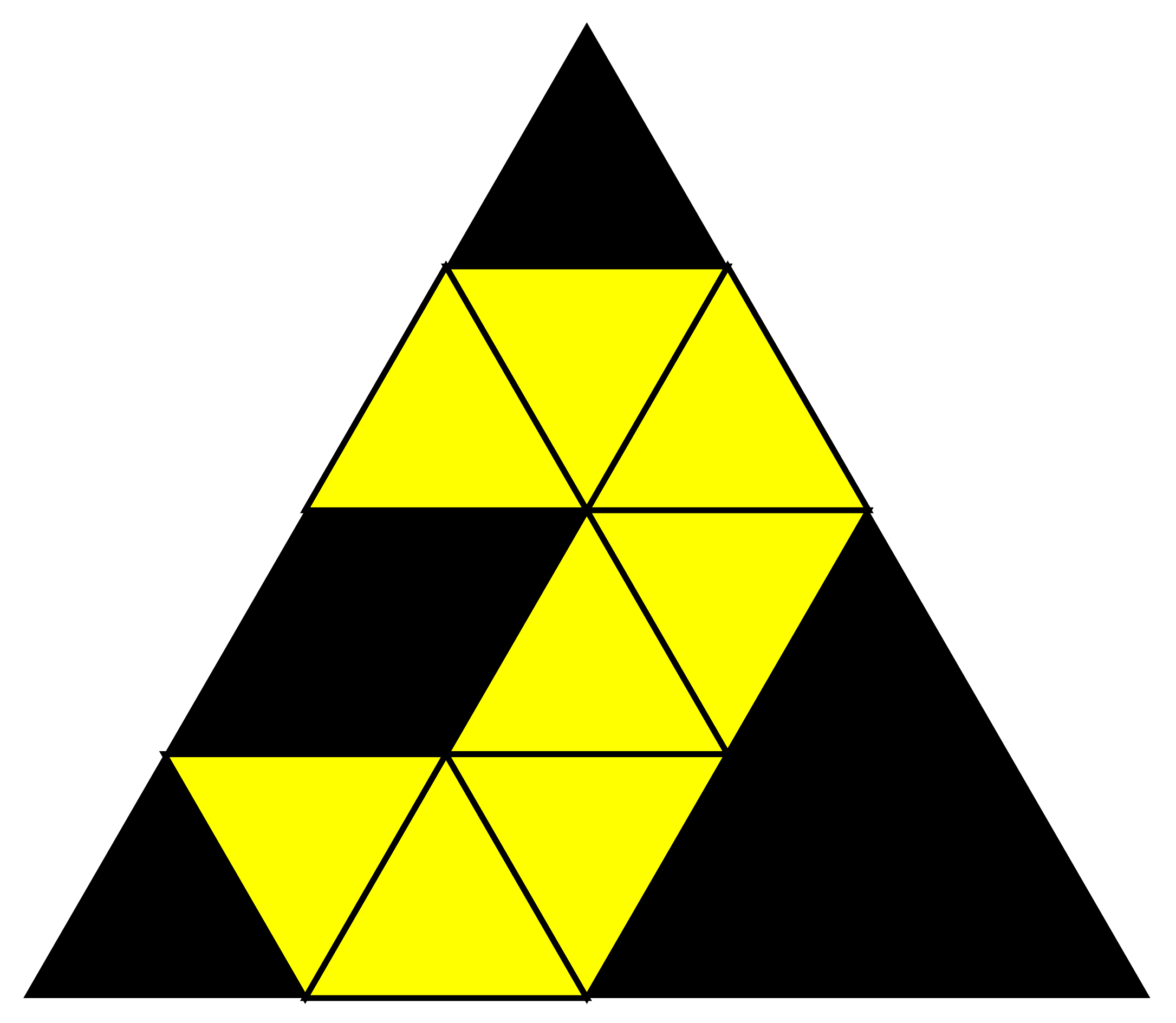}
\caption{On the left we see $\W_{1} \cup \W'_1$ and on the right we have $\Y_{1} \cup \Y'_1$ 
discussed in Example~\ref{Beispiel2}.}
\label{fig:Bsp2WY1}
\end{center}
\end{figure}
We see that $e_1^\w=T_4(0,2,1)$, $e_2^\w=T_4(2,0,1)$,  $e_3^\w=T_4(1,2,0)$ and 
$e_1^\y=T_4(0,1,2)$, $e_2^\y=T_4(1,0,2)$,  $e_3^\y=T_4(2,1,0)$.
Observe that $\W_{1}$ contains $T_4(1, 1, 1)$ and $\W'_{1}$ contains $T'_4(0, 1, 1)$ as well as
$T'_4(1, 0, 1)$,  i.e., $\W_{1} \cup \W'_{1}$ contains a strip of type $3$. Similarly, 
$\Y_{1} \cup \Y'_{1}$ contain a strip of of type $3$, too since $T'_{4}(0, 1, 1)$ is included in 
$\Y'_{1}$. In particular, our triangular labyrinth patterns system is not $\{1,2\}$-blocked. Without difficulties we also see that it is not $\{1,3\}$-blocked. Thus, it is only $\{2,3\}$-blocked.
The white and yellow triangles of level $2$ and $4$ are shown in 
Figure~\ref{fig:Bsp2WY23}.
\begin{figure}
\begin{center}
\includegraphics[width=0.3\textwidth]{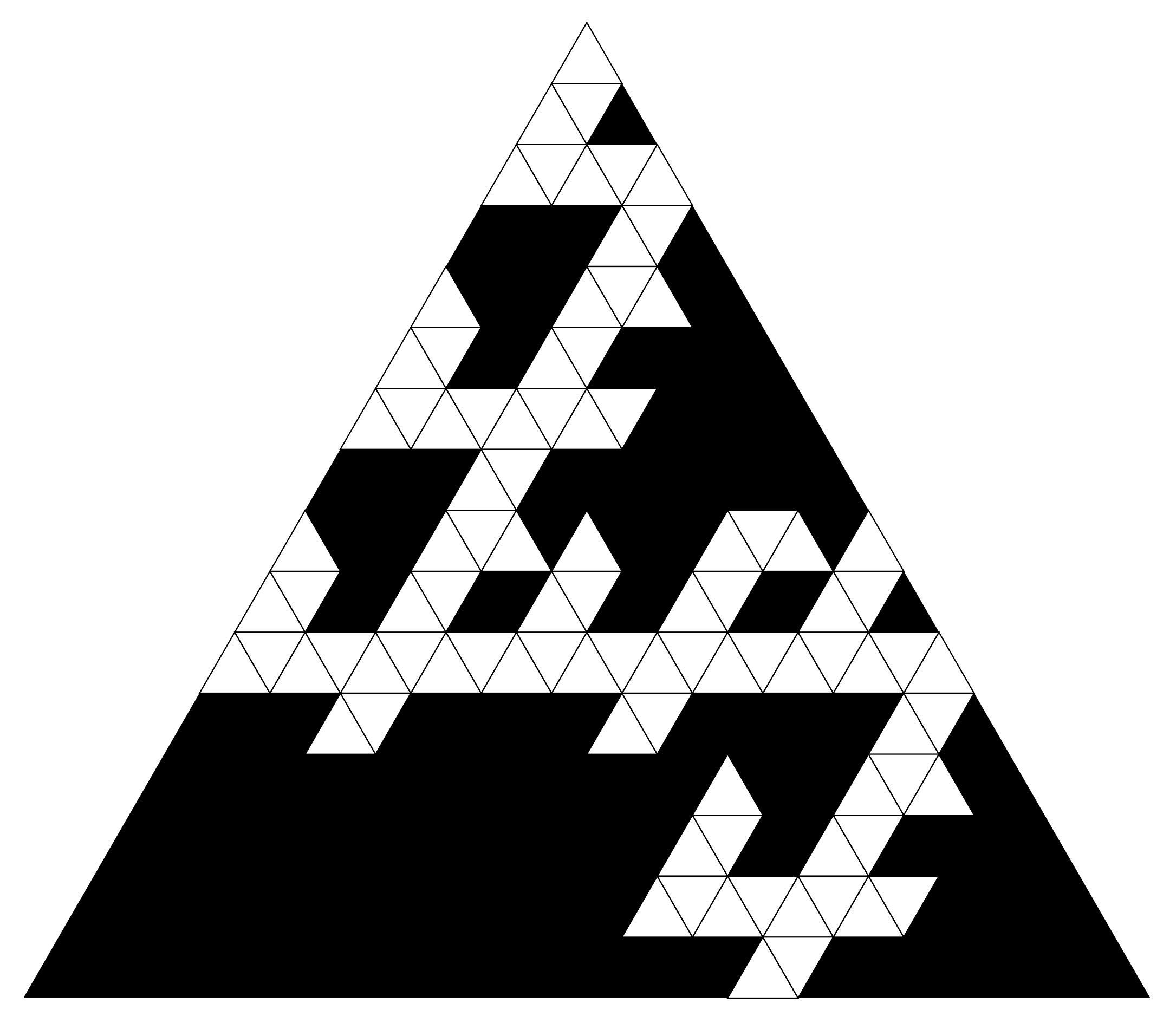}
\hspace{2cm}
\includegraphics[width=0.3\textwidth]{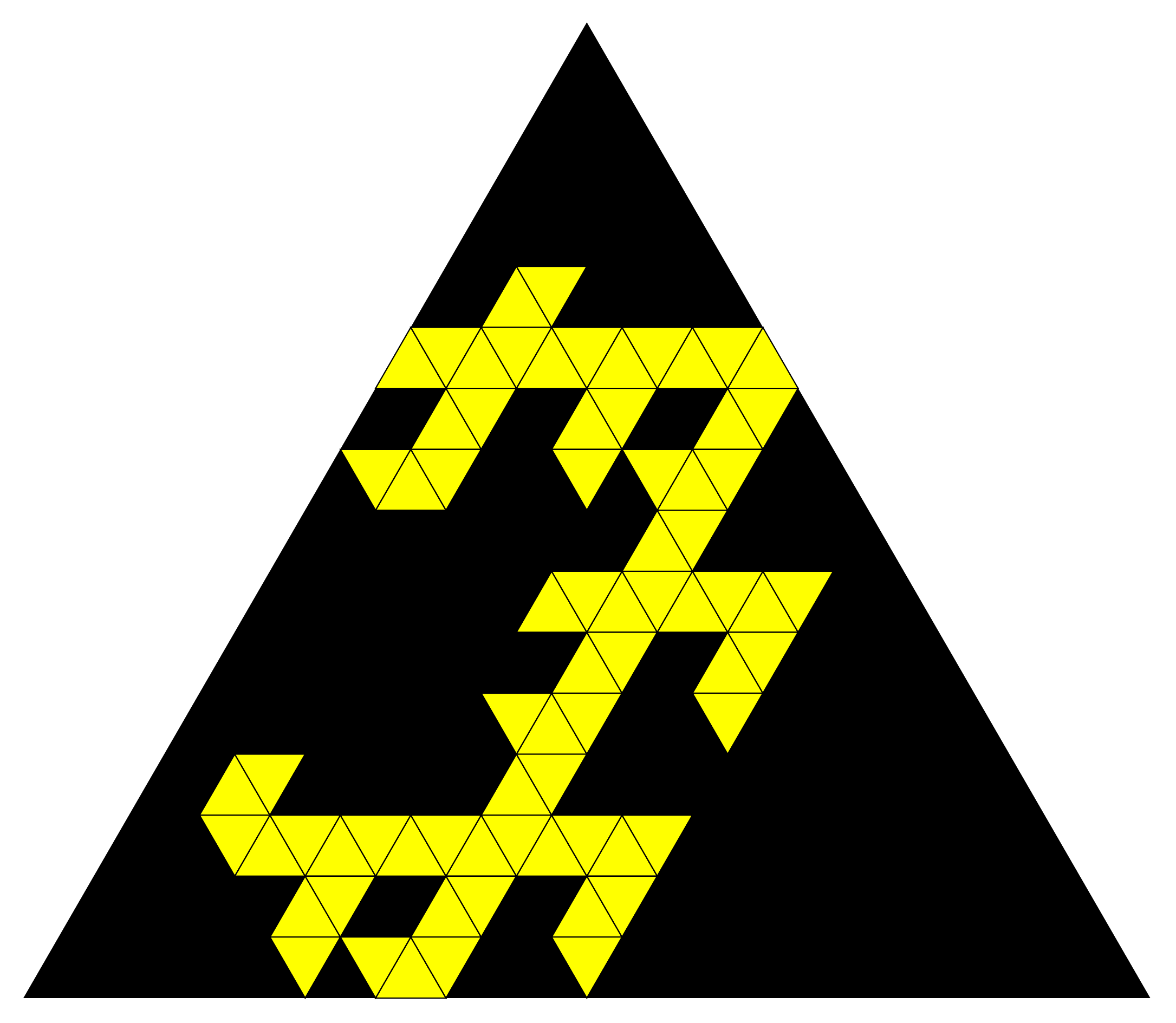}
\includegraphics[width=0.3\textwidth]{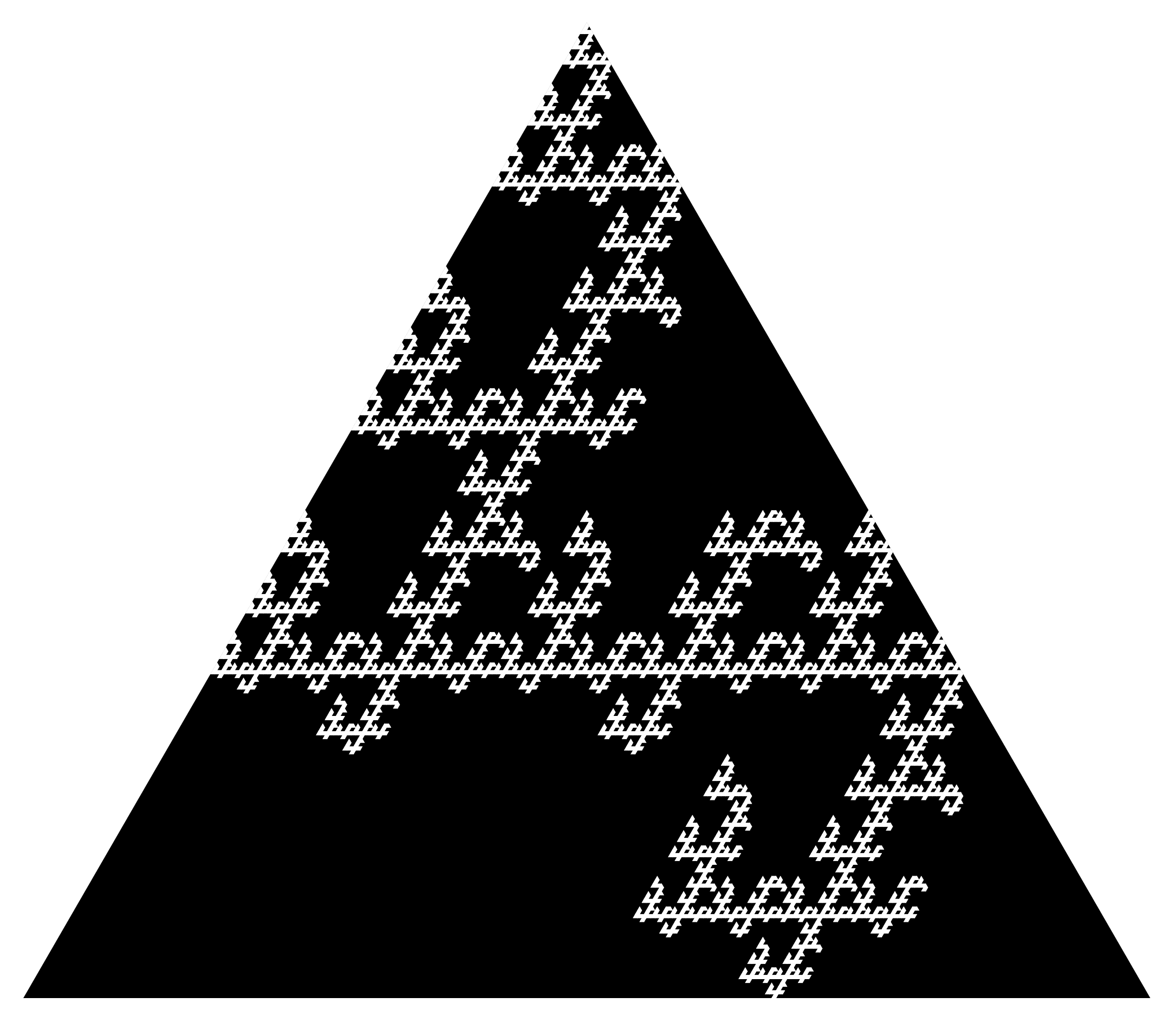}
\hspace{2cm}
\includegraphics[width=0.3\textwidth]{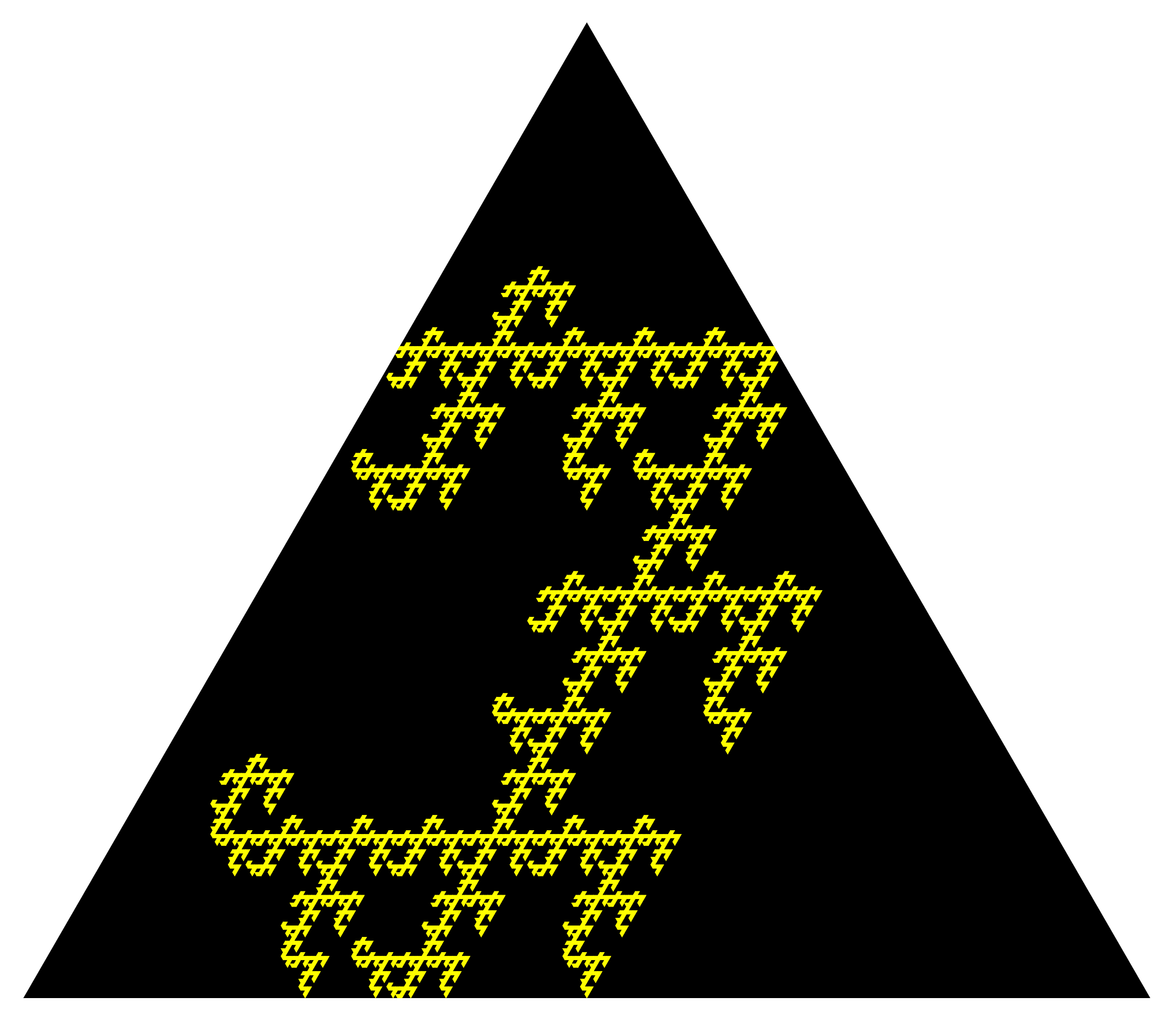}
\caption{The triangular labyrinth patterns system from Example~\ref{Beispiel2}: $\W_{2} \cup \W'_2$ (top left), $\Y_{2} \cup \Y'_2$ (top right), $\W_{4} \cup \W'_4$ (bottom left), $\Y_{4} \cup \Y'_4$ (bottom right).}\label{fig:Bsp2WY23}
\end{center}
\end{figure}
\end{example}

\begin{lemma}\label{lemma:2x_ungeblockt_paths_exits}
Let $\pi=\{i,j\}$ be the unique $\pi \in \P$ such that the triangular labyrinth patterns system is $\pi$-blocked. Then all arcs between exits of $\linfw$ or $\linfy$ have finite length. Moreover, the Hausdorff dimension of any arc between exits of $\linfw$ or $\linfy$ is one.
\end{lemma}
\begin{proof} 
We sketch a constructive proof. W.l.o.g, let $\pi=\{2,3\} $ (as illustrated in Figure \ref{fig:Bsp2WY1}). Based on Remark \ref{remark:cores}, the arc construction given in Lemma \ref{lemma:arc_construction}, and the definition of the fractals' exits, it immediately follows that the arc in $\linfw$ between $\mathbf E_1^\w$ and $\mathbf E_2^\w$ is just a straight line segment included in a parallel to $\rho_3$, and the arc in $\linfw$ between $\mathbf E_1^\w$ and $E_3^\w$ is a straight line segment which is included in a parallel to $\rho_2$. Moreover, due to Proposition \ref{prop:exists_C}, it follows that the arc  in $\linfw$ that connects $\mathbf E_2^\w$ and $\mathbf E_3^\w$ is the union of the two straight line segments defined by $\mathbf E_1^\w$ and $\mathbf E_2^\w$, and  by $\mathbf E_1^\w$ and $\mathbf E_3^\w$, respectively. The definition of Hausdorff dimension and the relation between Hausdorff measure and $1$-dimensional Lebesgue measure immediately yield that the Hausdorff dimension of arcs between exits of the fractal equals one. This completes the proof.
\end{proof}

\begin{proposition}\label{prop:2x_ungeblockt_finite_arcs}
Let $\pi=\{i,j\}$ be the unique $\pi \in \P$ such that the triangular labyrinth patterns system is $\pi$-blocked. Then all arcs in $\linfw$ and $\linfy$ have  finite length. 
\end{proposition}

\begin{proof} ({\it Sketch.}) W.l.o.g., we prove the above assertion for $\linfw$, the proof for $\linfy$ works in the analogous way. Let ${\mathbf X},{\mathbf Y} \in \linfw $,  ${\mathbf X}\ne {\mathbf Y}$. We consider several cases corresponding to the possible positions of the points ${\mathbf X}$ and $ {\mathbf Y}$. Due to the structure of the triangular labyrinth fractals it is sufficient to consider the following cases.
\\{\it Case 1.} The case when ${\mathbf X}$ and $ {\mathbf Y}$ are exits of $\linfw$ is covered by Lemma \ref{lemma:2x_ungeblockt_paths_exits}.
\\{\it Case 2.} If ${\mathbf X}$ and $ {\mathbf Y}$ are exits of triangles of level  $n$ in $\W_n \cup \W'_n$, then one can easily obtain that the arc connecting them is the finite union of arcs in $\linfw$ between exits of triangles of level $n$, which all have finite length, by the self similar structure of the fractal and Lemma \ref{lemma:2x_ungeblockt_paths_exits}.
\\{\it Case 3.} One of the points is an exit of $\linfw$, say, w.l.o.g, ${\mathbf X}={\mathbf E}_i^\w$ and the other is a point in the core of the fractal dendrite, $\mathrm{core}(\linfw)$, which is not an exit of any triangle of any level $n\ge 1$. 
In particular, in this case of a triangular labyrinth patterns system that is blocked only with respect to one direction, say $\pi=\{2,3\}$, the core is just $a_{\{1,2\}}^\w \cup a_{\{1,3\}}^\w $, i.e., the core is here the union of two straight line segments. In this case, one can find, based on the structure of the triangular labyrinth fractals, two sequences $\{ {\mathbf Z}_n^{(1)}\}_{n\ge 1}$ and $\{ {\mathbf Z}_n^{(2)}\}_{n\ge 1}$, of  points that are exits of triangles of level $n$,  such that ${\mathbf Y}$ lies, for every $n \ge 1$, on the line segment defined by 
${\mathbf Z}_n^{(1)}$ and ${\mathbf Z}_n^{(2)}$ (which is contained in the fractal) and 
\[\lambda \big(a({\mathbf X}, {\mathbf Z}_n^{(1)} )\big)\le \lambda \big(a({\mathbf X},{\mathbf Y})\big) \le  
\lambda \big(a({\mathbf X}, {\mathbf Z}_n^{(2)} )\big). \] Thus, by the result for case 2, the arc that connects ${\mathbf X}$ and ${\mathbf Y}$ in $\linfw$ has finite length.
\\ {\it Case 4.} ${\mathbf X}\in \linfw \setminus \mathrm{core}(\linfw)$, i.e., ${\mathbf X}$ lies on a ``branch'' of the dendrite outside its core,  and ${\mathbf Y}$ is one of the fractal's exits. Since the branch containing ${\mathbf X}$ starts from a ramification point that lies either on the arc (straight line segment) in $\linfw$ that connects ${\mathbf E}_1^\w$ and ${\mathbf E}_2^\w$, or on the one that connects  ${\mathbf E}_1^\w$ and ${\mathbf E}_3^\w$, let us assume, w.l.o.g., that the branch that contains ${\mathbf X}$ starts on the straight line that connects ${\mathbf E}_1^\w$ and ${\mathbf E}_2^\w$ and ${\mathbf Y}={\mathbf E}_1^\w$. An other choice for the position of the branch or of the exit ${\mathbf Y}$ would not essentially change  the proof. From the structure of $\mathrm{core}(\linfw)$  we see in this case that the only possible  ramification  points  in the fractal dendrite are among the exits of type $1$ of triangles of level $n$, for $n\ge 1$. If ${\mathbf X}$ is some exit of some triangle of level $n\ge 1$, then we are done. If this is not the case, there exists $n_0 \ge 1$ such that ${\mathbf X}$ lies in the interior of some triangle $W\in \W_{n_0} \cup \W'_{n_0}$, in the core of the fractal dendrite and at the same time scaled triangular labyrinth fractal $\linfw \cap W$, i.e. in $\mathrm {core}(\linfw \cap W)$. This, together with the result in case 3,  yields the desired result.
\end{proof}

In the following we give a more general result, which summarises the properties of the arcs in laybrinth fractals generated by a patterns systen blocked in only one direction.

\begin{theorem}\label{theorem:2x_unblocked_mainresult} Let $\pi=\{i,j\}$ be the unique $\pi \in \P$ such that the triangular labyrinth patterns system is $\pi$-blocked. Then both for $\linfw$ and for $\linfy$ the following hold:
\begin{enumerate}
\item[(a)] Every non-trivial arc in the fractal has finite length.
\item[(b)] All non-trivial arcs of the fractal have Hausdorff dimension one.
\item[(c)] For any non-trivial arc in the fractal there exists the tangent to the arc at almost all (with respect to Lebesgue measure) points of the arc.
\end{enumerate}
\end{theorem}
\begin{proof}
Proposition \ref{prop:2x_ungeblockt_finite_arcs}  yields the assertion (a), which then leads to (b), based on the relationship between  Hausdorff and Lebesgue measure, and the definition of Hausdorff dimension (see, e.g., \cite{falconerbook}).
For any arc in the labyrinth fractal that connects two distinct points of the fractal, the existence of the tangent easily follows from the finite length of the arc and a theorem in the book of Tricot \cite[p.73, Chapter 7.1]{Tricot_book}, which shows (c).
\end{proof}

We herefrom see that the case when there exists a unique $\pi \in \P $ such that the triangular labyrinth patterns system is $\pi$-blocked looks rather  less interesting for research.

Now, we consider the  case when there exists a unique $\pi \in \P$ such that the triangular labyrinth patterns system is not $\pi$-blocked.
As in the case of globally blocked triangular
labyrinth fractals we are interested in the  GIFS that generates the arcs (see the end of Section~\ref{sec:arcs_TLF}). But now,
Theorem~\ref{theorem:matrix_primitive} does not hold anymore, which is equivalent to the fact that  the 
multigraph $\HH'$  is not strongly connected. 
In order to study such systems we have to analyse  the structure of the strongly connected components.

\begin{proposition}\label{prop:1x_ungeblockt}
Let $\pi$ be the unique element of $\P$ such that the triangular labyrinth patterns system is not $\pi$-blocked. Then $a_{\pi}^\w$ and $a_{\pi}^\y$ have finite length
while $a_{\pi'}^\w$ and $a_{\pi'}^\y$ have infinite length 
%\ll{ALT:} for $\pi' \not=\pi$  
for $\pi' \in \P\setminus \{ \pi \} $.

\end{proposition}
\begin{proof} 
Assume that $\pi=\{1,2\}$ (the other cases can be dealt with analogously). We want to apply \cite[Theorem 5]{MW}, thus we are interested in the GIFS that determines the arcs.
We have already noted that the global path matrix $M$
is the incidence matrix of the corresponding  multigraph $\HH'$ and that  $M$ is not primitive any more.
Due to Lemma~\ref{lemma:row_positive} and Proposition~\ref{prop:path_matrix} we know the shape of $M$.
In particular, in this setting, there exists a $k \in \{1,\ldots, m-2\}$ and non-negative $2 \times 2$ integer matrices 
$M'_\w$ and $M'_\y$ such that
\[M=\left(\begin{array}{c|c|c|c}
k+1 & \begin{array}{cc} 0 \quad & \quad  0 \end{array} & k & \begin{array}{cc} 0 \quad & \quad  0 \end{array} \\
\hline
\begin{array}{c} * \\ * \end{array} &   M'_\w   & \begin{array}{c} * \\ * \end{array} & M'_\w - I_2 \\
\hline
m-k-1 & \begin{array}{cc} 0\quad & \quad 0 \end{array} & m-k & \begin{array}{cc} 0 \quad & \quad  0 \end{array} \\
\hline
\begin{array}{c} * \\ * \end{array} & M'_\y- I_2 & \begin{array}{c} * \\ * \end{array} & M'_\y  
\end{array}\right),\]
where $I_2$ is the two-dimensional identity matrix and each $*$ represents a non-negative entry.
Since our system is $\{1,3\}$-blocked and $\{2,3\}$-blocked we know by Lemma~\ref{lemma:row_positive} that the sum of the second and the fifth row as well as the sum of the third and the sixth row of $M$ are strictly positive.
From this we conclude that $\HH'$ consists of two strongly connected components $\HH'(1), \HH'(2)$.
The first one, $\HH'(1)$, contains the first vertex $a_{\{1,2\}}^\w$ and the fourth vertex $a_{\{1,2\}}^\y$, and the second one, $\HH'(2)$, consists of the other four vertices $a_{\{1,3\}}^\w, a_{\{1,3\}}^\y, a_{\{2,3\}}^\w, a_{\{2,3\}}^\y$.
We also see that  the only edges between the strongly connected components start in $\HH'(2)$ and terminate in $\HH'(1)$.

Each strongly connected component corresponds to certain eigenvalues of $M$. 
For $\HH'(1)$ these are $m$ and $1$, since $(1,0,0,1,0,0)$ is a left eigenvector of $M$ with respect to $m$ and $(1,0,0,-1,0,0)^T$ is a right eigenvector of $M$ with respect to the eigenvalue $1$.  Therefore, the Hausdorff dimension of the corresponding two arcs is $1$ and they are of finite length. (This part can also be proven without using GIFS.) 
In order to estimate the eigenvalues for $\HH'(2),$ we study the matrix 
\[M'=\left(\begin{array}{cc}
 M'_\w   & M'_\w - I_2 \\
 M'_\y- I_2  & M'_\y  
\end{array}\right),\]
which encodes  the adjacencies within $\HH'(2)$.
Analogously to Lemma~\ref{eigenvalues} we obtain that the eigenvalues of $M'$ are given by the eigenvalues of $M'_\w+M'_\y-I_2$ (and additionally $1$ with multiplicity $2$).
We have already observed that $M'_\w+M'_\y-I_2$ is a strictly positive matrix and 
from the proof of Lemma~\ref{lemma:row_positive} we also see that the diagonal elements of $M'_\w+M'_\y-I_2$  are larger than or equal to $m-1$.
 This shows that the row sum norm of $M'_\w+M'_\y-I_2$ and, hence, its dominant eigenvalue is at least $m$. 
This implies that the largest  among the eigenvalues corresponding to $\HH'(2)$ is at least $m$.
If this eigenvalue is larger than $m$ then the Hausdorff dimension of the four arcs (which are the vertices of the multigraph $\HH'(2)$) is strictly larger than $1$ and if the 
largest eigenvalue of $M'$ is $m$ then the four arcs have Hausdorff dimension $1$. However, in both cases 
the one-dimensional Lebesgue measure is infinite by \cite[Theorem 5]{MW}.
\end{proof}

\begin{example}\label{Beispiel3}
We consider the $4$-triangular labyrinth patterns system given by
\begin{align*}
\W_{1}:= & \{T_{4}({\bf k}): {\bf k} \in \{(0, 2, 1),(1, 0, 2), (1, 1, 1), (1, 2, 0)\}\}, \\
\W'_{1} := & \{T'_{4}({\bf k}): {\bf k} \in \{(0, 2, 0), (1, 0, 1), (1, 1, 0)\}\}, \\
\Y_{1}:= & \{T_{4}({\bf k}): {\bf k} \in \{(0, 1, 2), (1, 0, 2), (1, 1, 1),(1, 2, 0), (2, 0, 1), (2, 1, 0)\}\}, \\
\Y'_{1} := & \{T'_{4}({\bf k}): {\bf k} \in \{ (0, 1, 1), (1, 0, 1), (1, 1, 0)\}\}
\end{align*}
(see Figure~\ref{fig:Bsp3WY1}). 
\begin{figure}[h]
\begin{center}
\includegraphics[width=0.2\textwidth]{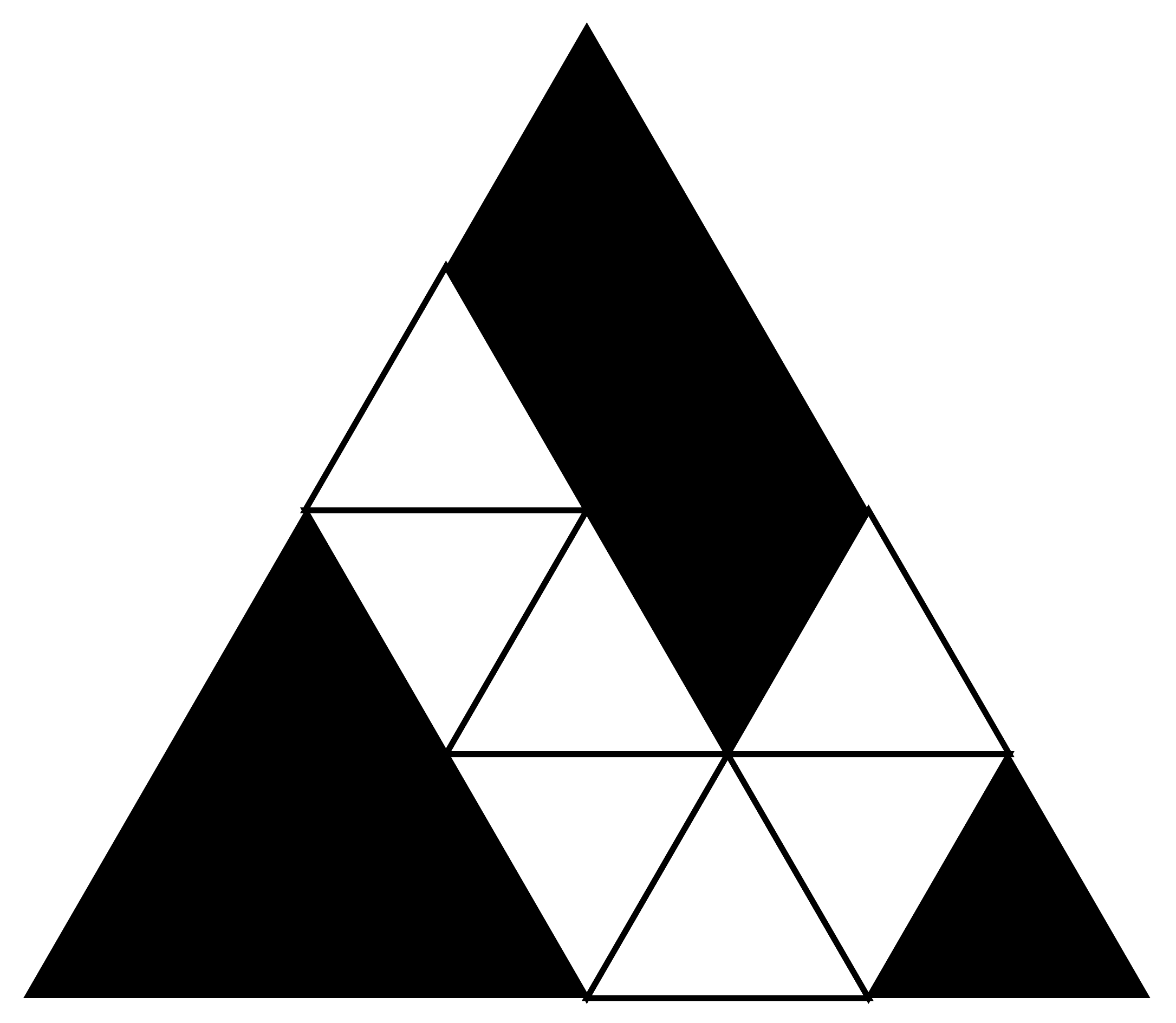}
\hspace{3cm}
\includegraphics[width=0.2\textwidth]{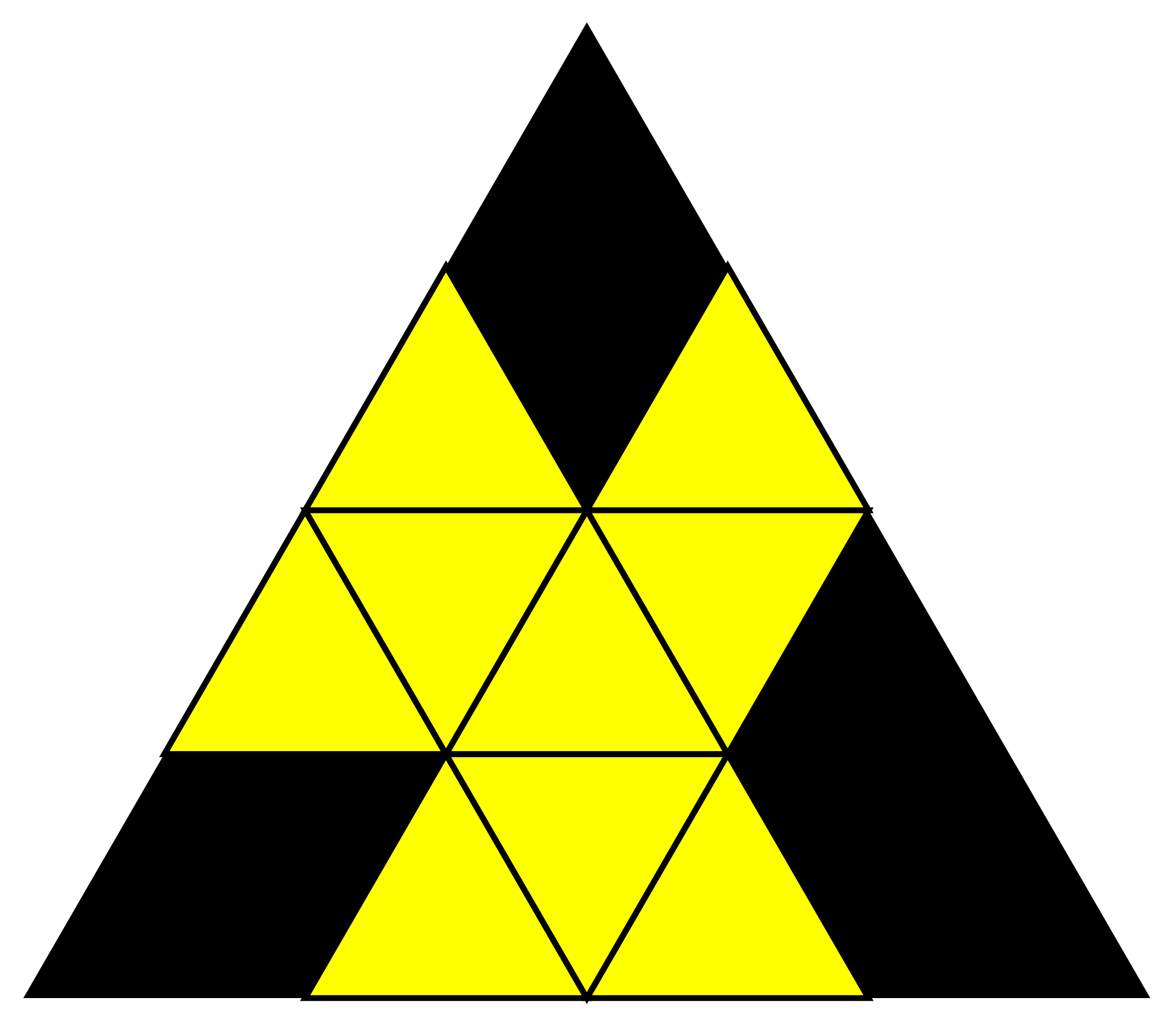}
\caption{The sets $\W_{1} \cup \W'_1$ and  $\Y_{1} \cup \Y'_1$ 
for Example~\ref{Beispiel2}.}
\label{fig:Bsp3WY1}
\end{center}
\end{figure}
This triangular labyrinth patterns system is $\{1,2\}$-blocked and 
$\{2,3\}$-blocked but not $\{1,3\}$-blocked. 
The white and yellow triangles of level $2$ and $4$ are depicted in 
Figure~\ref{fig:Bsp3WY23}.
\begin{figure}[h]
\begin{center}
\includegraphics[width=0.3\textwidth]{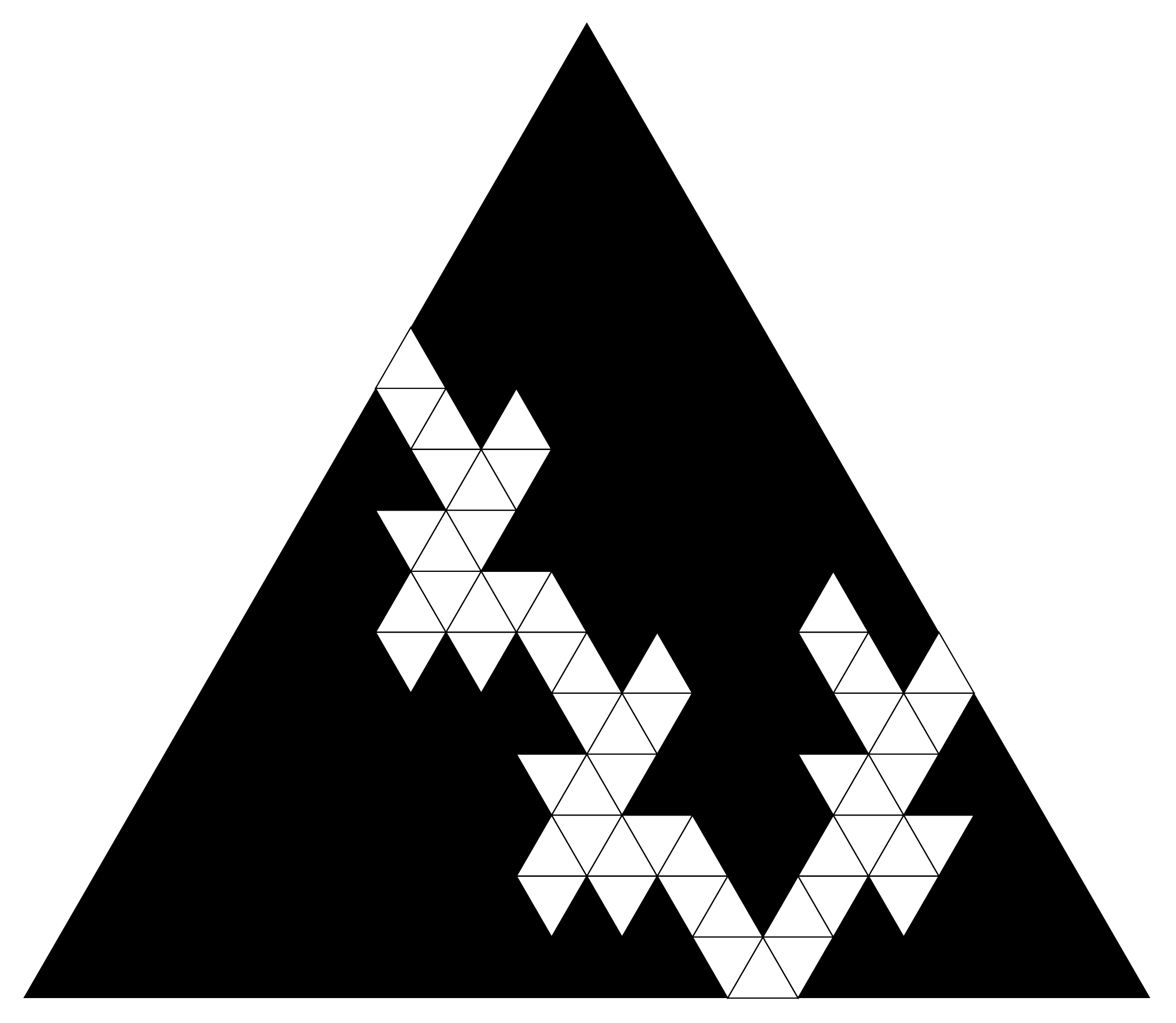}
\hspace{2cm}
\includegraphics[width=0.3\textwidth]{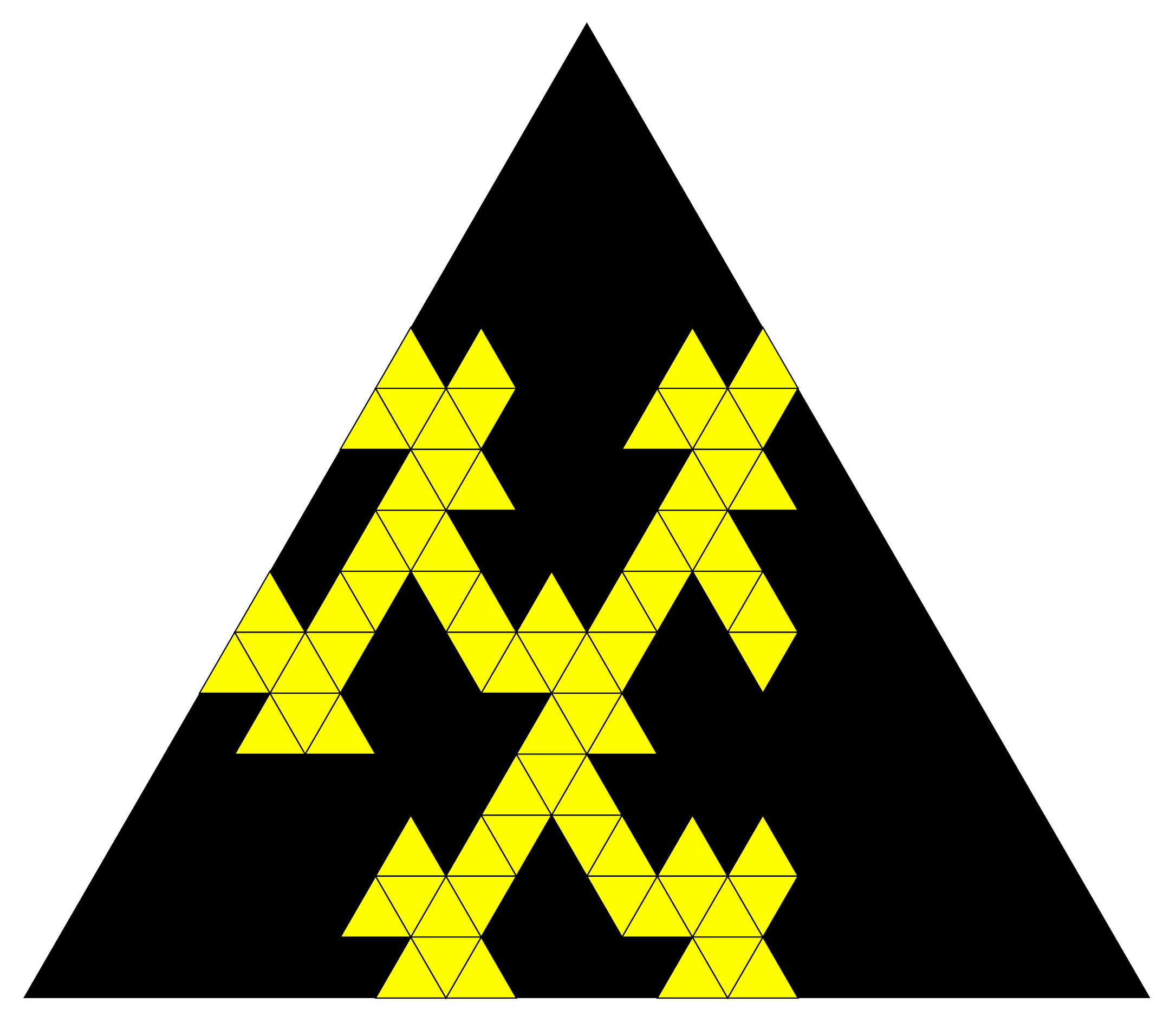}
\includegraphics[width=0.3\textwidth]{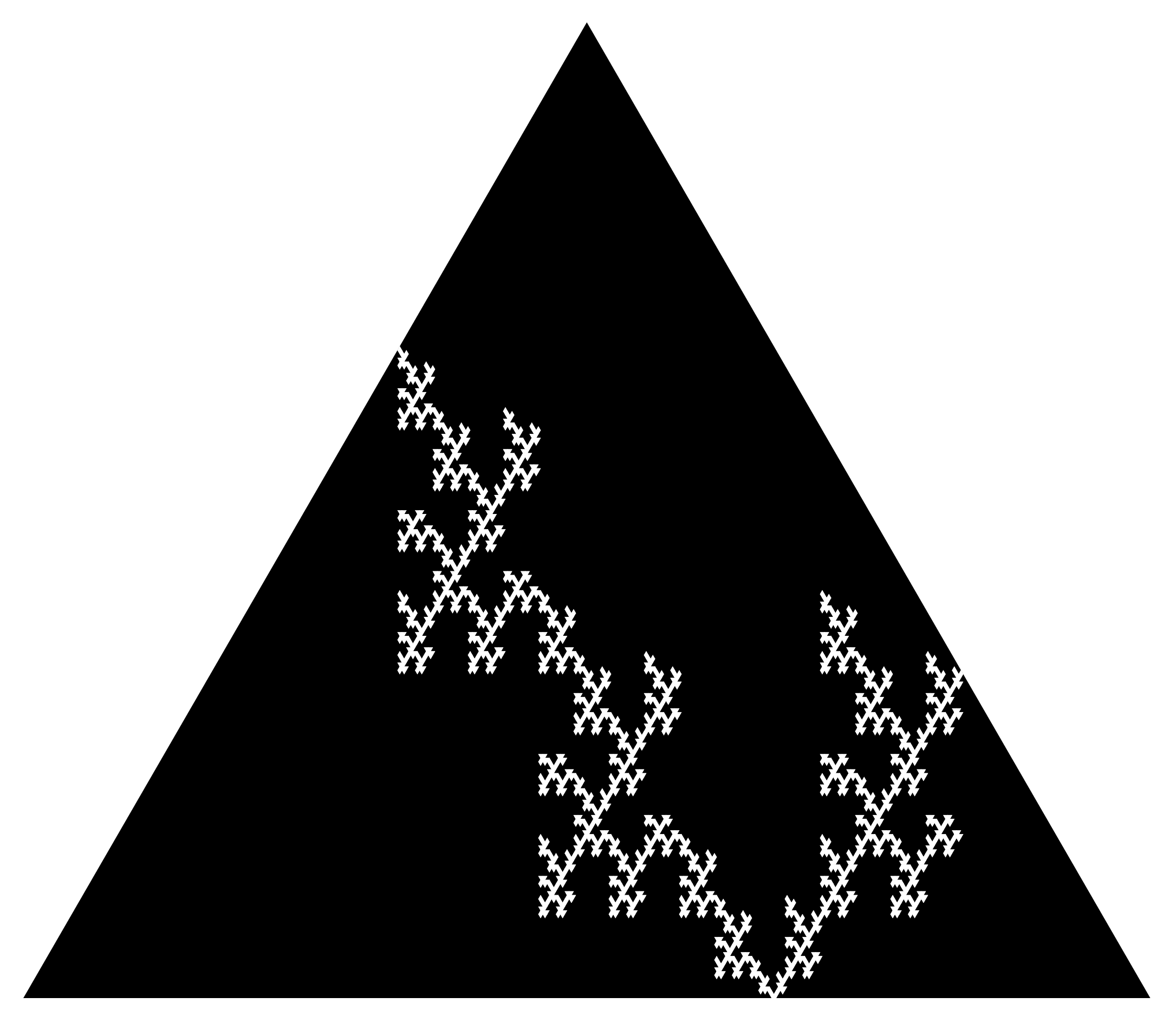}
\hspace{2cm}
\includegraphics[width=0.3\textwidth]{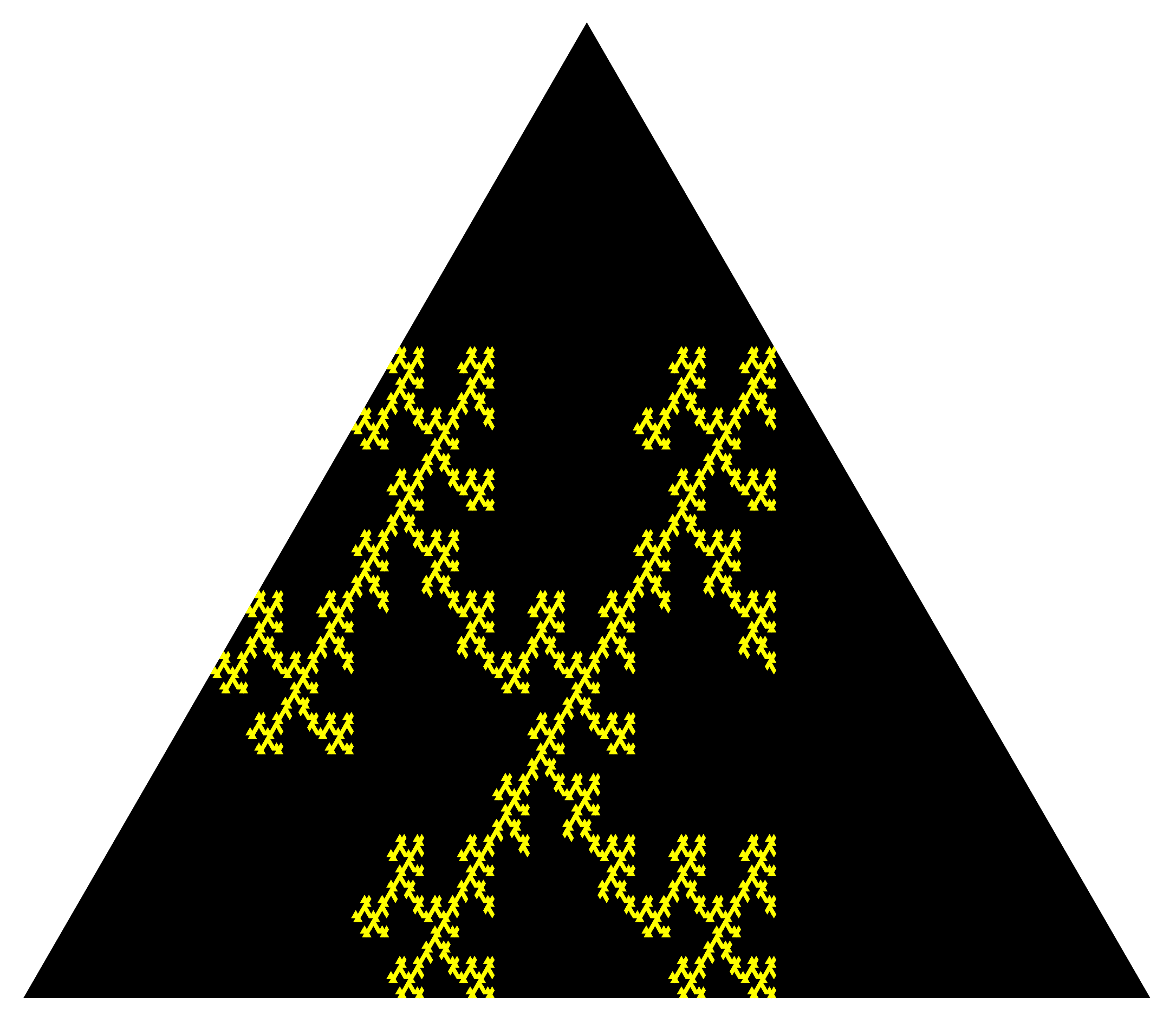}
\caption{The white and yellow triangles of level $2$ (top) and of level $4$ (bottom) for  
Example~\ref{Beispiel3}.}\label{fig:Bsp3WY23}
\end{center}
\end{figure}
\end{example}

Proposition \ref{prop:1x_ungeblockt} and Lemma \ref{lemma:projections_of_arcs_between_exits} immediately yield the following result.

\begin{corollary}\label{cor:1x_unblocked_arcs_triangle_exits}
Under the assumptions of Proposition \ref{prop:1x_ungeblockt}, for $\pi=\{i,j \}\in {\mathcal P}$, for all $n \ge 1$, the length of the arc that connects the exits of type $i$ and $j$ of any triangle $W \in \W_n \cup \W'_n$ (of any triangle
$Y \in \Y_n \cup \Y'_n$) is finite, while the length of every other arc connecting exits in $\linfw$  ($\linfy$, respectively) of these triangles is infinite. 
\end{corollary}

Now, we extend the above results to arbitrary non-trivial arcs in such partially blocked triangular labyrinth fractals.

\begin{proposition}\label{prop:arcs_in_1x_unblocked}
Under the assumptions of Proposition \ref{prop:1x_ungeblockt}, the length of the arc in the fractal dendrite that connects two distinct points in $\linfw$ ($\linfy$) is finite if and only if the arc coincides with the straight line segment defined by the two points, that is contained in the fractal and parallel to the side $\rho_k$ of $T_1$, where $\pi = \{i,j\}$ and $ \{k\}= \{ 1,2,3 \}\setminus \{i,j\}$.
\end{proposition}
\begin{proof}
Since one of the implications is trivial, let us prove the other one. Let, w.l.o.g., $a\in \linfw$ be an arc that connects the distinct points ${\mathbf X}, {\mathbf Y} \in \linfw$ and has finite length. For $n \ge 1$, let $p_n$ be the path in $\G(\W_n \cup \W'_n)$ that connects $W_n({\mathbf X})$ and $W_n({\mathbf Y})$, chosen as in Lemma \ref{lemma:arc_construction}. Since ${\mathbf X}\ne {\mathbf Y}$, there exists some $n_0\ge 1$ such that $p_{n_0}$ consists of at least three triangles of $\W_n \cup \W'_n$. Let then $W$ be a triangle in the path $p_{n_0}$, such that $W \ne W_n({\mathbf X}),W_n({\mathbf Y})$ and let $\gamma \in \P$ be the type of the triangle $W$ with respect to the path $p_{n_0}$. By the arc construction given in Lemma \ref{lemma:arc_construction} and the definition of the exits of a triangle it follows that $a \cap W$ is an arc between two exits of $W$. If $\gamma \ne \pi$, then it follows from Corollary \ref{cor:1x_unblocked_arcs_triangle_exits} that the length of the curve $a \cap \W$ is infinite, which contradicts the finite length of $a$.  Therefore, all triangles in $p_{n_0}$ distinct from $W_{n_0}({\mathbf X})$ and $W_{n_0}({\mathbf Y})$ can only be of type $\pi$ with respect to $p_{n_0}$. Due to the fact that the labyrinth patterns system is not $\pi$-blocked, it follows, by also applying Lemma \ref{lemma:arc_construction},  that the arc $a \subseteq \linfw$ that connects ${\mathbf X}$ and ${\mathbf Y}$ in $ \linfw$ is a straight line segment, parallel to the side $\rho_k$ of $T_1$.
\end{proof}


\begin{thebibliography}{10}

\bibitem{cristealeobacher_arcs}
{\sc L.~L. {Cristea} and G.~{Leobacher}}, {\em {On the length of arcs in
  labyrinth fractals.}}, {Monatsh. Math.}, 185 (2018), pp.~575--590.

\bibitem{supermix}
\leavevmode\vrule height 2pt depth -1.6pt width 23pt, {\em {Supermixed
  labyrinth fractals.}}, {J. Fractal Geom.}, 7 (2020), pp.~183--218.


\bibitem{laby_4x4}
{\sc L.~L. {Cristea} and B.~{Steinsky}}, {\em {Curves of infinite length in $4
  \times 4$-labyrinth fractals.}}, {Geom. Dedicata}, 141 (2009), pp.~1--17.

\bibitem{laby_oigemoan}
{\sc L.~L. {Cristea} and B.~{Steinsky}}, {\em {Curves of infinite length in
  labyrinth fractals.}}, {Proc. Edinb. Math. Soc., II. Ser.}, 54 (2011),
  pp.~329--344.

\bibitem{mixlaby}
\leavevmode\vrule height 2pt depth -1.6pt width 23pt, {\em {Mixed labyrinth
  fractals.}}, {Topology Appl.}, 229 (2017), pp.~112--125.

\bibitem{Edgar}
{\sc G.~{Edgar}}, {\em {Measure, topology, and fractal geometry. 2nd ed.}}, New
  York, NY: Springer, 2nd ed.~ed., 2008.

\bibitem{EM}
{\sc G.~A. {Edgar} and R.~D. {Mauldin}}, {\em {Multifractal decompositions of
  digraph recursive fractals.}}, {Proc. Lond. Math. Soc. (3)}, 65 (1992),
  pp.~604--628.

\bibitem{falconerbook}
{\sc K.~{Falconer}}, {\em {Fractal geometry. Mathematical foundations and
  applications. 3rd ed.}}, Hoboken, NJ: John Wiley \& Sons, 3rd ed.~ed., 2014.

\bibitem{freiberghamblyhutchinson_Vvariable}
{\sc U.~{Freiberg}, B.~M. {Hambly}, and J.~E. {Hutchinson}}, {\em {Spectral
  asymptotics for $V$-variable Sierpinski gaskets.}}, {Ann. Inst. Henri
  Poincar\'e, Probab. Stat.}, 53 (2017), pp.~2162--2213.

\bibitem{tarafdar_multifractalNaCl2013}
{\sc A.~Giri, M.~Dutta~Choudhury, T.~Dutta, and S.~Tarafdar}, {\em Multifractal
  growth of crystalline nacl aggregates in a gelatin medium}, Crystal Growth \&
  Design, 13 (2013), pp.~341--345.

\bibitem{JanaGarcia_lithiumdendrite2017}
{\sc A.~{Jana} and R.~E. {Garc\'{i}a}}, {\em Lithium dendrite growth mechanisms
  in liquid electrolytes}, Nano Energy, 41 (2017), pp.~552 -- 565.

\bibitem{Kuratowski}
{\sc K.~Kuratowski}, {\em {Topology, Volume II}}, Academic Press, 1968.

\bibitem{LauLuoRao}
{\sc K.-S. {Lau}, J.~J. {Luo}, and H.~{Rao}}, {\em {Topological structure of
  fractal squares.}}, {Math. Proc. Camb. Philos. Soc.}, 155 (2013), pp.~73--86.

\bibitem{MW}
{\sc R.~D. {Mauldin} and S.~C. {Williams}}, {\em {Hausdorff dimension in graph
  directed constructions.}}, {Trans. Am. Math. Soc.}, 309 (1988), pp.~811--829.

\bibitem{PotapovPotapovPotapov_dec2017}
{\sc A.~Potapov and V.~Potapov}, {\em Fractal radioelement's, devices and
  systems for radar and future telecommunications: Antennas, capacitor,
  memristor, smart 2d frequency-selective surfaces, labyrinths and other
  fractal metamaterials}, Journal of International Scientific Publications:
  Materials, Methods \& Technologies, 11 (2017), pp.~492--512.

\bibitem{GrachevPotapovGerman2013}
{\sc A.~A. Potapov, V.~A. German, and V.~I. Grachev}, {\em Fractal labyrinths
  as a basis for reconstruction planar nanostructures}, in 2013 International
  Conference on Electromagnetics in Advanced Applications (ICEAA), Sept 2013,
  pp.~949--952.

\bibitem{PotapovGermanGrachev2013}
\leavevmode\vrule height 2pt depth -1.6pt width 23pt, {\em ``{N}ano -'' and
  radar signal processing: Fractal reconstruction complicated images, signals
  and radar backgrounds based on fractal labyrinths}, in 2013 14th
  International Radar Symposium (IRS), vol.~2, June 2013, pp.~941--946.

\bibitem{PotapovZhang2016}
{\sc A.~A. Potapov and W.~Zhang}, {\em Simulation of new ultra-wide band
  fractal antennas based on fractal labyrinths}, in 2016 CIE International
  Conference on Radar (RADAR), Oct 2016, pp.~1--5.

\bibitem{Samuel_selfsimilar_dendrites}
{\sc M.~{Samuel}, A.~V. {Tetenov}, and D.~A. {Vaulin}}, {\em {Self-similar
  dendrites generated by polygonal systems in the plane.}}, {Sib. \`Elektron.
  Mat. Izv.}, 14 (2017), pp.~737--751.

\bibitem{SeegerHoffmannEssex2009_randomKoch}
{\sc S.~{Seeger}, K.~H. {Hoffmann}, and C.~{Essex}}, {\em {Random walks on
  random Koch curves.}}, {J. Phys. A, Math. Theor.}, 42 (2009), p.~11.

\bibitem{Tarafdar_modelporstructrepeatedSC2001}
{\sc S.~{Tarafdar}, A.~{Franz}, C.~{Schulzky}, and K.~H. {Hoffmann}}, {\em
  {Modelling porous structures by repeated Sierpinski carpets.}}, {Physica A},
  292 (2001), pp.~1--8.

\bibitem{Tricot_book}
{\sc C.~{Tricot}}, {\em {Curves and fractal dimension. With a foreword by
  Michel Mend\`es France. Transl. from the French.}}, New York, NY:
  Springer-Verlag, 1995.

\bibitem{lipschitz_gaskets}
{\sc Z.~{Zhu}, Y.~{Xiong}, and L.~{Xi}}, {\em {Lipschitz equivalence of
  self-similar sets with triangular pattern.}}, {Sci. China, Math.}, 54 (2011),
  pp.~1019--1026.

\end{thebibliography}
\end{document}